\pdfoutput=1
\documentclass[11pt]{article}
\title{\vspace{-1cm} \linespread{1.15} \bfseries \large THOMASON FILTRATION VIA $T(1)$-LOCAL $\TC$}
\author{\MakeUppercase Hyungseop Kim}
\date{}

\let\originallhook=\lhook

\usepackage[colorlinks]{hyperref} 
\usepackage[backend=biber, block=space, style=alphabetic, sorting=nyt, maxbibnames=99, maxalphanames=4, minalphanames=3, doi=false, isbn=false, url=false, eprint=false]{biblatex} 
\usepackage{xurl}
\hypersetup{breaklinks=true}

\DeclareFieldFormat*{title}{\mkbibemph{#1}}
\DeclareFieldFormat*{journaltitle}{#1}
\DeclareFieldFormat*{booktitle}{#1}

\renewbibmacro{in:}{%
  \ifentrytype{article}
    {}
    {\printtext{\bibstring{in}\intitlepunct}}}

\DeclareFieldFormat[article,periodical]{volume}{\mkbibbold{#1}}
\DeclareFieldFormat[article,periodical]{number}{\bibstring{number}~#1}
\renewbibmacro*{journal+issuetitle}{%
  \usebibmacro{journal}%
  \setunit*{\addspace}%
  \iffieldundef{series}
    {}
    {\newunit
     \printfield{series}%
     \setunit{\addspace}}%
  \printfield{volume}%
  \setunit{\addspace}%
  \usebibmacro{issue+date}%
  \setunit{\addcolon\space}%
  \usebibmacro{issue}%
  \setunit{\addcomma\space}%
  \printfield{number}%
  \setunit{\addcomma\space}%
  \printfield{eid}%
  \newunit}

\DeclareFieldFormat[article,periodical]{pages}{#1}

\DeclareFieldFormat{mrnumber}{%
  MR\addcolon\space
  \ifhyperref
    {\href{http://www.ams.org/mathscinet-getitem?mr=#1}{\nolinkurl{#1}}}
    {\nolinkurl{#1}}}

\renewbibmacro*{doi+eprint+url}{%
  \iftoggle{bbx:doi}
    {\printfield{doi}}
    {}%
  \newunit\newblock
  \printfield{mrnumber}%
  \newunit\newblock
  \iftoggle{bbx:eprint}
    {\usebibmacro{eprint}}
    {}%
  \newunit\newblock
  \iftoggle{bbx:url}
    {\usebibmacro{url+urldate}}
    {}}

\DeclareFieldFormat{postnote}{\mknormrange{#1}}
\DeclareFieldFormat{volcitepages}{\mknormrange{#1}}
\DeclareFieldFormat{multipostnote}{\mknormrange{#1}}

\usepackage{amsthm,amssymb,amsmath,mathrsfs,stmaryrd}  
\usepackage[T1]{fontenc}
\usepackage{lmodern}
\usepackage{xcolor}
\usepackage[all,matrix,color,arrow,pdf]{xy}  
\usepackage{graphicx}
\usepackage{mathtools}
\usepackage{geometry}
\usepackage{tikz}  
\usetikzlibrary{calc, intersections}
\usepackage{enumerate}
\usepackage{pdfpages}
\usepackage[utf8]{inputenc}
\usepackage{longtable}
\usepackage{setspace}
\usepackage{hyperref}
\usepackage{titlesec}
\usepackage{eucal} 
\usepackage{tikz-cd}
\usepackage{slashed}
\usepackage{relsize}
\usepackage[bbgreekl]{mathbbol}
\usepackage{amsfonts}
\usepackage{changepage}
\hypersetup{linkcolor=[RGB]{144,0,32}, citecolor=[RGB]{33,66,171}}

\DeclareSymbolFontAlphabet{\mathbb}{AMSb}
\DeclareSymbolFontAlphabet{\mathbbl}{bbold}
\newcommand{\prism}{\mathbbl{\Delta}}
\def\prismc{\widehat{\prism}}

\let\lhook=\originallhook

\newcommand\sbullet[1][.5]{\mathbin{\vcenter{\hbox{\scalebox{#1}{$\bullet$}}}}} 

\DeclareMathAlphabet{\mathpzc}{OT1}{pzc}{m}{it}

\def\ker{\mathrm{ker}}

\def\Map{\mathrm{Map}}

\def\Aut{\mathrm{Aut}}

\def\Sp{\operatorname{Sp}}
\def\Spec{\operatorname{Spec}}

\def\Spf{\operatorname{Spf}}

\def\tr{\operatorname{tr}}

\def\K{\mathcal{K}}

\def\Dh{\widehat{\mathcal{D}}}

\def\DFh{\widehat{\mathcal{DF}}}

\def\DFhc{\widehat{\mathcal{DF}}^{\mathrm{c}}}

\def\op{\mathrm{op}}

\def\Gal{\operatorname{Gal}}
\def\N{\operatorname{N}}

\def\colim{\operatorname{colim}}

\def\can{\mathrm{can}}

\def\Ring1{\text{Ring}1}
\def\CRing1{\text{CRing}1}
\def\Mod{\text{Mod}}

\def\Alg{\text{Alg}}
\def\CAlg{\text{CAlg}}

\def\an{\mathrm{an}}

\def\bbA{\mathbb{A}}
\def\bbC{\mathbb{C}}

\def\bbE{\mathbb{E}}
\def\bbF{\mathbb{F}}

\def\bbN{\mathbb{N}}
\def\bbZ{\mathbb{Z}}

\def\bbQ{\mathbb{Q}}

\def\bbS{\mathbb{S}}

\def\inf{\operatorname{inf}}

\def\Frac{\mathrm{Frac}}

\def\+1{\xrightarrow{+1}}

\def\dt{\otimes^{\mathrm{L}}}
\def\dth{\widehat{\otimes}^{\mathrm{L}}}
\def\otimesh{\widehat{\otimes}}

\def\RG{\text{R}\Gamma}

\def\et{\mathrm{\'et}}

\def\syn{\mathrm{syn}}

\def\arc{\mathrm{arc}}

\def\arcp{\arc_{p}}
\def\perfd{\mathrm{perfd}}

\def\gr{\operatorname{gr}}

\def\Cat{\text{Cat}}

\def\perf{\text{perf}}
\def\ex{\text{ex}}

\def\Pr{\text{Pr}}

\def\Sp{\text{Sp}}

\def\N{\text{N}}
\def\Fun{\text{Fun}}
\def\fib{\text{fib}}

\def\Cpl{\operatorname{Cpl}}
\def\cpl{\mathrm{cpl}}

\def\Shv{\operatorname{Shv}}

\def\hyp{\mathrm{hyp}}

\def\sm{\mathrm{sm}}
\def\poly{\mathrm{poly}}

\def\THH{\mathrm{THH}}
\def\TC{\mathrm{TC}}

\def\TR{\mathrm{TR}}
\def\tr{\mathrm{tr}}

\def\TP{\mathrm{TP}}

\def\K{\mathrm{K}}

\def\Fil{\mathrm{Fil}}

\def\QSyn{\mathrm{QSyn}}

\def\qrsp{\mathrm{qrsp}}

\def\cyc{\mathrm{cyc}}

\newgeometry{vmargin={23mm}, hmargin={19mm} }

\theoremstyle{definition}
\newtheorem{example}{Example}[section]  

\newtheorem{notation}[example]{Notation}

\newtheorem{proposition}[example]{Proposition}
\newtheorem{lemma}[example]{Lemma}

\newtheorem{theorem}[example]{Theorem}
\newtheorem{corollary}[example]{Corollary}

\newtheorem{remarkn}[example]{Remark} 
\newtheorem{construction}[example]{Construction}

\newtheorem*{ack}{Acknowledgements}
\newtheorem*{conv}{Conventions}

\theoremstyle{remark}

\renewenvironment{abstract}{\noindent\begin{center}\begin{minipage}{0.85\linewidth}\small{\scshape Abstract.}}{\end{minipage}\end{center}}

\makeatletter
\def\blfootnote{\xdef\@thefnmark{}\@footnotetext}
\makeatother

\begin{document}
\maketitle
\begin{abstract}
We construct a natural filtration on $T(1)$-local $\TC$ for any animated commutative rings using prismatic cohomology and descent theory. In the course of the construction, we also study some general properties of prismatic cohomology complexes over perfect prisms after inverting distinguished generators. The construction is intrinsic to $\TC$ and recovers Thomason's spectral sequence for $T(1)$-local algebraic K-theory via the cyclotomic trace map; as a consequence, we also recover the \'etale comparison for prismatic cohomology.
\end{abstract}

{
  \let\clearpage\relax
  \small
  \tableofcontents
}

\section{Introduction}
In this paper,\blfootnote{Mathematics Subject Classification 2020: 19D55 (primary), 14F30 19E20 (secondary)} we construct a filtration on $T(1)$-local topological cyclic homology of derived commutative rings through prismatic and descent methods, and study its relationship with Thomason's filtration on $T(1)$-local algebraic K-theory.\blfootnote{Keywords: algebraic K-theory, topological cyclic homology, motivic filtration, prismatic cohomology, \'etale comparison}  

\subsection{Motivation}
One of the major technical difficulties in studying algebraic K-theory of schemes is in its failure of satisfying the \'etale (or stronger) descent property. Being a localizing invariant, nonconnective algebraic K-theory $\K$ satisfies Nisnevich (hence in particular Zariski) descent, cf. \cite[Th. 8.1]{tt} and \cite[Prop. 5.15]{etalek}, yet fails to satisfy finite Galois descent. \\
\indent One way of alleviating this problem is to study localizing invariants which satisfy \'etale descent and approximate algebraic K-theory sufficiently well. Fix a prime number $p$, and let $T(1)\simeq(\bbS/p)[v_{1}^{-1}]$ be a height 1 telescope. In \cite{tho,tt}, Thomason showed that for a $\bbZ[1/p]$-scheme $X$ satisfying some finiteness conditions, $T(1)$-local\footnote{Note that $T(1)$-localization coincides with $K(1)$-localization, where $K(1)$ is the Morava K-theory of height $1$ at $p$.} algebraic K-theory $L_{T(1)}\K$ on $X$ satisfies \'etale hyperdescent, and there is a spectral sequence starting from the $E_{2}$-page $E_{2}^{i,j} = H^{i}_{\et}(X,\bbZ_{p}(j))$ and converging to $\pi_{-i+2j}L_{T(1)}\K(X)$ coming from the \'etale descent spectral sequence for $L_{T(1)}\K$. This was further generalized and clarified through Clausen and Mathew, cf. \cite[Th. 7.14]{etalek}. Moreover, by (now-proven) Quillen--Lichtenbaum conjecture, cf. \cite{ro} and \cite{rost,voe}, $L_{T(1)}\K(X)$ agrees with $\K(X)$ for sufficiently high degrees for such $X$. On the other hand, by chromatic redshift at height 1 \cite{mit}, telescopic localizations at height higher than 1 of algebraic K-theory vanishes for derived schemes (which are Zariski locally $T(1)$-acyclic), and hence they do not provide meaningful approximations of algebraic K-theory of ordinary or derived schemes. \\
\indent At the residue characteristic $p$, topological cyclic homology $\TC$ provides a useful approximation of algebraic K-theory through the cyclotomic trace map $\tr:\K\to\TC$. Upon $T(1)$-localization, the difference between algebraic K-theory and topological cyclic homology vanishes for $p$-adic derived rings. More precisely, for a $p$-Henselian derived commutative ring $R$, i.e., a derived commutative ring $R$ whose underlying ordinary commutative ring $\pi_{0}(R)$ is $p$-Henselian, the map $L_{T(1)}\K(R)\to L_{T(1)}\TC(R)$ induced by the cyclotomic trace map is an equivalence. By Dundas--Goodwillie--McCarthy theorem \cite{dgm}, one reduces the statement to the case of ordinary rings. Then, Clausen--Mathew--Morrow rigidity \cite{cmm} says that the fiber of the cyclotomic trace map $\tr:\K_{\geq0}\to\TC$ of $R$ is equivalent to that of $R/p$, and from the consequence $L_{T(1)}\K(\bbF_{p})\simeq 0$ of Quillen's computation \cite{quil}, one knows that the $T(1)$-localization of the fiber vanishes, hence verifies the claim. Combining this with \cite[Th. 2.16]{bcm} or \cite[Cor. 4.22]{lmmt}, Bhatt, Clausen and Mathew proved that there is a natural equivalence
\begin{equation}\label{eq:bcm}
L_{T(1)}\K(R^{\wedge_{p}}[1/p])\simeq L_{T(1)}\TC(R)
\end{equation}
for any derived commutative ring $R$ \cite[Th. 2.17]{bcm}. Although the statement \emph{loc. cit.} is for ordinary commutative rings, note that the equivalence (\ref{eq:bcm}) holds for any derived commutative rings as noted above.\\
\indent The left hand side of the equivalence (\ref{eq:bcm}) admits a further structure through the aforementioned result of Thomason. Under certain finiteness conditions imposed on the input ring, the (pro-)\'etale Postnikov filtration for $L_{T(1)}\K$ on the input ring endows the left hand side with a structure of a filtered spectrum, whose associated spectral sequence, cf. \cite[Constr. 1.2.2.6]{ha}, gives Thomason's descent spectral sequence. In particular, the associated graded pieces of the filtration are given by $p$-adic \'etale cohomology complexes at the generic fiber of the input ring. Through the equivalence (\ref{eq:bcm}), the right hand side is automatically equipped with the corresponding filtration. \\
\indent Our goal in this paper is to explain a different, prismatic construction of a natural exhaustive filtration on $L_{T(1)}\TC(R)$ for each derived commutative ring $R$. The construction is intrinsic to $\TC$, yet through the equivalence (\ref{eq:bcm}) recovers the Thomason filtration on $L_{T(1)}\K$, enhancing the equivalence to that of filtered spectra. In particular, we recover the \'etale comparison theorem for prismatic cohomology \cite[Th. 9.1]{bs2} as a consequence, confirming the expectation of \cite[Rem. 2.19]{bcm}. In the course of constructing the filtration, we also study some properties of prismatic cohomology complexes over perfect prisms away from the Hodge-Tate locus.\\
\indent Zariski, or more generally Nisnevich descent property of localizing invariants indicates that localizing invariants behave, in a sense, as cohomology theories of noncommutative motives. For commutative inputs, e.g., commutative rings or schemes, study of filtrations on localizing invariants provides a way of approaching localizing invariants through their associated graded pieces, which would behave as cohomology complexes. For $\TC$ and other Hochschild homology type localizing invariants, \cite{bms2} constructed the motivic filtrations on them and studied relevant cohomology complexes in the $p$-adic setting. In recent years, constructions \emph{loc. cit.} have been further generalized to certain global and derived (or $\bbE_{\infty}$-) settings \cite{morin,bl,hrw}; our construction is in accordance with the approaches of \cite{bms2,hrw} in spirit. On the other hand, the consequent recovery of the étale comparison theorem for prismatic cohomology proceeds in the reverse direction, approaching cohomology theories through localizing invariants and their filtrations. In particular, this aligns with the theme of studying comparison theorems for $p$-adic cohomology theories via algebraic K-theory, as pioneered in \cite{nizcrys,nizss}. 

\subsection{Main results}\label{subsec:mainresults}

We provide a construction of a natural filtration of the following form on $T(1)$-local $\TC$ for animated commutative rings, i.e., (connective) derived commutative rings:

\begin{theorem}[Theorem \ref{th:statementT(1)TCfilt}]  \label{th:introT(1)TCfilt}
For each animated commutative ring $R$, there exists a natural exhaustive multiplicative $\bbZ$-indexed descending filtration $\Fil^{\sbullet}L_{T(1)}\TC(R)$ on the $p$-complete spectrum $L_{T(1)}\TC(R)$, such that for each $n\in\bbZ$, there is a natural equivalence  
\begin{equation*}
\gr^{n}L_{T(1)}\TC(R)\simeq\RG_{\et}(\Spec R^{\wedge_{p}}[1/p],\bbZ_{p}(n))[2n].
\end{equation*} 
The filtration $\Fil^{\sbullet}L_{T(1)}\TC(R)$ is complete if $\pi_{0}(R^{\wedge_{p}})[1/p]$ has finite Krull dimension and admits a uniform bound on the mod $p$ virtual cohomological dimensions of the residue fields. \\
\indent Moreover, the functor $R\mapsto \Fil^{\sbullet}L_{T(1)}\TC(R)$ commutes with sifted colimits, induces equivalences on taking derived $p$-completion and taking $\pi_{0}$ of animated commutative rings, and satisfies $p$-complete fppf descent.
\end{theorem}

\begin{remarkn}
As $L_{T(1)}\TC$ is truncating on $T(1)$-acyclic spectra \cite[Cor. 4.22]{lmmt}, Theorem \ref{th:introT(1)TCfilt} also implies that for each $T(1)$-acyclic, connective, and homotopy commutative $\bbE_{1}$-ring $R$, there exists a natural exhaustive multiplicative $\bbZ$-indexed descending filtration $\Fil^{\sbullet}L_{T(1)}\TC(R) := \Fil^{\sbullet}L_{T(1)}\TC(\pi_{0}(R))$ on $L_{T(1)}\TC(R)\simeq L_{T(1)}\TC(\pi_{0}(R))$ such that for each $n\in\bbZ$, there is a natural equivalence  
\begin{equation*}
\gr^{n}L_{T(1)}\TC(R)\simeq\RG_{\et}(\Spec \pi_{0}R^{\wedge_{p}}[1/p],\bbZ_{p}(n))[2n],
\end{equation*} 
and that the filtration $\Fil^{\sbullet}L_{T(1)}\TC(R)$ is complete if $\pi_{0}(R^{\wedge_{p}})[1/p]$ has finite Krull dimension and admits a uniform bound on the mod $p$ virtual cohomological dimensions of the residue fields.
\end{remarkn}

\begin{remarkn}
The filtration $\Fil^{\sbullet}L_{T(1)}\TC$ provides an analogue of the $p$-adic \'etale motivic filtration on $\TC(-,\bbZ_{p})$ of \cite{bms2} for $T(1)$-local $\TC$. Note that directly $T(1)$-localizing the motivic filtration on $\TC(R,\bbZ_{p})$ only results in the trivial filtration; each of the associated graded pieces $\gr^{n}\TC(R,\bbZ_{p})\simeq \RG_{\syn}(\Spf R^{\wedge_{p}},\bbZ_{p}(n))[2n]$ is an $H\bbZ_{p}$-module, and hence vanishes after $T(1)$-localization. Thus, one needs an alternative way of reflecting $T(1)$-localization on cohomology complexes which appear as associated graded pieces of the filtration.  
\end{remarkn}

Our construction of the filtration $\Fil^{\sbullet}L_{T(1)}\TC$ does not use the Thomason filtration on $L_{T(1)}\K$ and the cyclotomic trace map relating $L_{T(1)}\TC$ with $L_{T(1)}\K$. In fact, the point of our approach is that the filtration can be constructed and described through computations of $T(1)$-local $\TC$ via prismatic cohomology and $\arcp$-hyperdescent. \\
\indent Recall that a map $R\to S$ of (derived) $p$-complete commutative rings is an $\arcp$-cover (resp. a $p$-complete $\arc$-cover) if for any map $R\to V$ into a $p$-complete valuation ring $V$ of rank $1$ with $0\neq p\in V$ (resp. a $p$-complete valuation ring $V$ of rank $\leq1$), there is an extension $V\hookrightarrow W$ of $p$-complete valuation rings of rank $\leq 1$ and a map $S\to W$ which fit into the commutative diagram 
\begin{equation*}
\begin{tikzcd}
R \arrow[r] \arrow[d] & V \arrow[d, hook] \\
S \arrow[r] & W,
\end{tikzcd}
\end{equation*}  
 cf. \cite{arc,cs}. This notion of covers endow the opposite category of $p$-complete commutative rings a finitary Grothendieck topology, called the $\arcp$-topology (resp. the $p$-complete $\arc$-topology). Write $(L_{T(1)}\TC)^{\sharp}$ for an $\arcp$-hypersheafification of $L_{T(1)}\TC$ on $p$-complete commutative rings. We first construct the following filtration-completed version $\Fil^{\sbullet}(L_{T(1)}\TC)^{\sharp}$ of $\Fil^{\sbullet}L_{T(1)}\TC$ as a filtration on $(L_{T(1)}\TC)^{\sharp}$: 

\begin{theorem}[Theorem \ref{th:statementfiltonarcpsheaf}] \label{th:introfiltonarcpsheaf}
For each animated commutative ring $R$, there is a natural exhaustive and complete multiplicative $\bbZ$-indexed descending filtration $\Fil^{\sbullet}(L_{T(1)}\TC)^{\sharp}(R)$ on the $p$-complete spectrum $(L_{T(1)}\TC)^{\sharp}(\pi_{0}(R^{\wedge_{p}}))$, such that for each $n\in\bbZ$ and a perfect prism $(A,(d))$ corresponding to a perfectoid ring $\overline{A}$, there is a natural equivalence 
\begin{equation*}
\gr^{n}(L_{T(1)}\TC)^{\sharp}(R)\simeq \left(\prism_{R/A}\{n\}[1/d]\right)^{F=1}[2n]
\end{equation*} 
for an animated commutative $\overline{A}$-algebra $R$. Moreover, the functor $R\mapsto \Fil^{\sbullet}(L_{T(1)}\TC)^{\sharp}(R)$ induces an equivalence on taking $\pi_{0}$ of the derived $p$-completion of animated commutative rings, and its restriction to $p$-complete commutative rings satisfies $\arcp$-hyperdescent.
\end{theorem}

By pulling back the filtration $\Fil^{\sbullet}(L_{T(1)}\TC)^{\sharp}(R)$ on $(L_{T(1)}\TC)^{\sharp}(\pi_{0}(R^{\wedge_{p}}))$ along the natural map $h:L_{T(1)}\TC(R)\to (L_{T(1)}\TC)^{\sharp}(\pi_{0}(R^{\wedge_{p}}))$ induced from the $\arcp$-hypersheafification, we obtain a filtration $\Fil^{\sbullet}L_{T(1)}\TC(R)$ on $L_{T(1)}\TC(R)$ which realizes $\Fil^{\sbullet}(L_{T(1)}\TC)^{\sharp}(R)$ as its completion; this gives a characterization of our construction in Theorem \ref{th:introT(1)TCfilt}. We will also provide another characterization of the filtration as a construction left Kan extended from finitely generated polynomial rings through an independent analysis. \\
\indent Using the characterization of $\Fil^{\sbullet}L_{T(1)}\TC$ as a pullback of the filtration from Theorem \ref{th:introfiltonarcpsheaf}, we can verify the following comparison of the filtration with the Thomason filtration on the $T(1)$-local algebraic K-theory side:

\begin{theorem}[Theorem \ref{th:statementcompwiththomason}] \label{th:introcompwiththomason}
There is a natural equivalence 
\begin{equation*}
\Fil^{\sbullet}L_{T(1)}\K(R^{\wedge_{p}}[1/p]) \simeq \Fil^{\sbullet}L_{T(1)}\TC(R)
\end{equation*}
of filtered $p$-complete spectra for an animated commutative ring $R$, which after taking the underlying $p$-complete spectra recovers the equivalence 
\begin{equation*}
L_{T(1)}\K(R^{\wedge_{p}}[1/p])\simeq L_{T(1)}\K(R^{\wedge_{p}}) \underset{\sim}{\xrightarrow{\tr}} L_{T(1)}\TC(R)
\end{equation*} 
induced by the cyclotomic trace map. Here, $\Fil^{\sbullet}L_{T(1)}\K(R^{\wedge_{p}}[1/p])$ is the Thomason filtration on the $p$-complete spectrum $L_{T(1)}\K(R^{\wedge_{p}}[1/p])$. In fact, there is a natural diagram 
\begin{equation*}
\begin{tikzcd}
\Fil^{\sbullet}L_{T(1)}\K(R^{\wedge_{p}}[1/p]) \arrow[r, "\sim"'] \arrow[d] & \Fil^{\sbullet}L_{T(1)}\TC(R) \arrow[d]  \\
L_{T(1)}\K(R^{\wedge_{p}}[1/p])\simeq L_{T(1)}\K(R^{\wedge_{p}}) \arrow[r, "\tr"', "\sim"] & L_{T(1)}\TC(R) 
\end{tikzcd}
\end{equation*}
of $\bbE_{\infty}$-algebra objects of filtered $p$-complete spectra for an animated commutative ring $R$. 
\end{theorem}

In particular, by taking the $n$-th associated graded objects of the equivalence, we have a natural equivalence $\gr^{n}L_{T(1)}\K(R^{\wedge_{p}}[1/p]) \simeq \gr^{n}L_{T(1)}\TC(R)$ for an animated commutative ring $R$. Since we constructed $\Fil^{\sbullet}L_{T(1)}\TC(R)$ as a pullback of $\Fil^{\sbullet}(L_{T(1)}\TC)^{\sharp}$, the right hand side of the equivalence is by definition given by $\gr^{n}(L_{T(1)}\TC)^{\sharp}(R)$ of Theorem \ref{th:introfiltonarcpsheaf}, while the left hand side of the equivalence is naturally given by $\RG_{\et}(\Spec R^{\wedge_{p}}[1/p], \bbZ_{p}(n))[2n]$ due to Thomason's computation \cite{tt}. Thus, we have the following consequence, which can be regarded as a K-theoretic proof of the \'etale comparison theorem for prismatic cohomology:

\begin{corollary}[Corollary \ref{cor:etcomp}] \label{cor:introetcomp}
For each $n\in\bbZ$ and a perfect prism $(A,(d))$ corresponding to a perfectoid ring $\overline{A}$, there is a natural equivalence 
\begin{equation*}
\gr^{n}(\tr)[-2n]:\RG_{\et}(\Spec R^{\wedge_{p}}[1/p],\bbZ_{p}(n)) \simeq \left(\prism_{R/A}\{n\}[1/d]\right)^{F=1}
\end{equation*}
for an animated commutative $\overline{A}$-algebra $R$.
\end{corollary}

We will also explain an algebraic proof of the above form of \'etale comparison in \ref{subsec:etcomp} using the prismatic logarithm map. It turns out that the comparison map $\gr^{n}(\tr)[-2n]$ of Corollary \ref{cor:introetcomp} agrees with the comparison map $\log_{\prism}^{\otimes n}$ of Proposition \ref{prop:etcomp} in subsection \ref{subsec:etcomp} constructed through prismatic logarithm, cf. Remark \ref{rem:etcompmaps}. Thus, we know that our construction of $\Fil^{\sbullet}L_{T(1)}\TC$ and its comparison with the Thomason filtration in Theorem \ref{th:introcompwiththomason} recovers the \'etale comparison theorem (as well as the involved comparison maps) for prismatic cohomology from the filtered equivalence $L_{T(1)}\K(R^{\wedge_{p}}[1/p])\simeq L_{T(1)}\TC(R)$, as expected in \cite[Rem. 2.19]{bcm}. \\
\indent Let us summarize our approach to the main results stated above. We start by constructing and describing the filtration $\Fil^{\sbullet}(L_{T(1)}\TC)^{\sharp}$ in \ref{subsec:filtconstr} by considering the case of quasiregular semiperfectoid algebras over $\mathcal{O}_{C}$, where one can take $C = \widehat{\overline{\bbQ_{p}}}$ for instance. For such rings, $T(1)$-localization of $\TP$ is reflected in inverting a distinguished generator $d$ on the prismatic cohomology complexes. Thus, we can use the results of \ref{subsec:dinvprism} concerning generalities of prismatic cohomology complexes to describe the Postnikov filtration on $T(1)$-local $\TC$ for such rings, using that after $T(1)$-localization, $\TP$ on commutative rings admits a Frobenius endomorphism whose fixed point recovers $L_{T(1)}\TC$. Then, we construct the filtration on $(L_{T(1)}\TC)^{\sharp}$ through right Kan extension from the case of quasiregular semiperfectoid $\mathcal{O}_{C}$-algebras; in fact, we verify that $L_{T(1)}\TC$ is already an $\arcp$-hypersheaf on the latter class of rings. In \ref{subsec:compwithhyp}, we construct the filtration $\Fil^{\sbullet}_{\mathrm{pb}}L_{T(1)}\TC$ on $L_{T(1)}\TC$ as a pullback of $\Fil^{\sbullet}(L_{T(1)}\TC)^{\sharp}$ along the $\arcp$-hypersheafification map. Using $\Fil^{\sbullet}(L_{T(1)}\TC)^{\sharp}$, we verify in \ref{subsec:compwiththomason} that the Thomason filtration on $T(1)$-local algebraic K-theory is equivalent to the filtration given by $\Fil^{\sbullet}_{\mathrm{pb}}L_{T(1)}\TC$ in a way compatible with the cyclotomic trace map, and recover the \'etale comparison theorem as a consequence. \\
\indent On the other hand, in \ref{subsec:filtconstr2}, we provide another construction of a filtration on $L_{T(1)}\TC$, which we write as $\Fil^{\sbullet}_{\mathrm{lke}}L_{T(1)}\TC$, through left Kan extension from finitely generated polynomial rings. The construction is independent from the construction of $\Fil^{\sbullet}_{\mathrm{pb}}L_{T(1)}\TC$ and relies on the two facts, namely (i) the \'etale comparison theorem from \ref{subsec:etcomp} or \ref{subsec:compwiththomason} and (ii) the fact that $\arcp$-hypersheafification does not change the values of $L_{T(1)}\TC$ for finitely generated polynomial rings, which we verify through \ref{subsec:progal} and \ref{subsec:arcphyp} by using descent methods. For the sake of expositional convenience, we prove that the two filtrations $\Fil^{\sbullet}_{\mathrm{lke}}L_{T(1)}\TC$ and $\Fil^{\sbullet}_{\mathrm{pb}}L_{T(1)}\TC$ are equivalent to each other immediately after constructing the latter filtration in \ref{subsec:compwithhyp}, although we do not use the properties of $\Fil^{\sbullet}_{\mathrm{lke}}L_{T(1)}\TC$ in the later subsection \ref{subsec:compwiththomason}. We write $\Fil^{\sbullet}L_{T(1)}\TC$ for the common filtration we obtain through the equivalence between these two filtrations. This finishes the outline of the verification of the main results.   

\begin{conv}
We fix a prime $p$. We write $\Sp$ for the $\infty$-category of spectra and write $\Sp^{\Cpl(p)}$ for its stable subcategory of $p$-complete spectra. Similarly, for an ordinary commutative ring $A$, we write $\mathcal{D}(A)^{\Cpl(p)}$ for the $\infty$-category of $p$-complete $A$-modules. For an ordinary commutative ring, its $p$-completeness refers to its $p$-completeness as an $\bbE_{\infty}$-ring, or equivalently derived $p$-completeness. We follow \cite{bs2,bl} for the theory of prismatic cohomology complexes and mostly follow \cite{bl} for pertinent notations. For each bounded prism $(A,I)$, we denote $\Dh(A)$ for the stable subcategory of the derived $\infty$-category $\mathcal{D}(A)$ spanned by $(p)+I$-complete objects, i.e., $\Dh(A)\simeq\Mod_{A}^{\Cpl((p)+I)}$. In particular, for the crystalline prism $(\bbZ_{p},(p))$, notations $\Dh(\bbZ_{p})$ and $\mathcal{D}(\bbZ_{p})^{\Cpl(p)}$ will be used interchangeably. 
\end{conv}

\begin{ack}
The author is grateful to Ben Antieau, Elden Elmanto, Akhil Mathew, and Matthew Morrow for many helpful communications and discussions at various stages of the project, and to the anonymous referee for their careful reading and many comments that significantly improved the clarity of the article. The author was supported by the CMK Foundation and the European Research Council (ERC) under the European Union's Horizon 2020 research and innovation programme (grant agreement No. 101001474). 
\end{ack}

\section{Prismatic cohomology complexes over perfect prisms} 
In this section, we study some general properties of prismatic cohomology complexes over perfect prisms. In \ref{subsec:dinvprism}, we observe that, away from the Hodge-Tate locus, Nygaard filtration becomes constant (Proposition \ref{prop:candisom}) and Nygaard completion can be decompleted after taking Frobenius fixed point (Proposition \ref{prop:nygdc}). Then, we verify a twisted variant of the \'etale comparison theorem of \cite{bs2} in \ref{subsec:etcomp} (Proposition \ref{prop:etcomp}). 

\subsection{$d$-inverted prismatic cohomology complexes}\label{subsec:dinvprism}
\indent First, let us fix some notations. For each bounded prism $(A,I)$, we write $\DFh(A) = \Fun(\N\bbZ^{\op},\Dh(A))$ for the $\infty$-category of $\bbZ$-indexed descending filtered objects of $\Dh(A)$, and write $\DFhc(A)$ for its stable subcategory spanned by filtration-complete objects. We will only consider perfect prisms, i.e., prisms whose $\delta$-ring is perfect, and write $\overline{A} = A/d$ for the perfectoid ring corresponding to the perfect prism $(A,(d))$. \\
\indent For an (ordinary) commutative ring $R$, let $\CAlg^{\an}_{R}\simeq \mathcal{P}_{\Sigma}(\CAlg^{\poly}_{R})$ be the $\infty$-category of derived commutative $R$-algebras, also known as animated commutative $R$-algebras, obtained as a sifted closure of the category $\CAlg^{\poly}_{R}$ of finitely generated polynomial $R$-algebras. Note that $\CAlg^{\an}_{R}$ is equivalent to the underlying $\infty$-category of the simplicial model category of simplicial commutative $R$-algebras \cite[Cor. 5.5.9.3]{htt}. For $R=\bbZ$, we write $\CAlg^{\an}$ for $\CAlg^{\an}_{\bbZ}$. 
\begin{remarkn}
By construction, there is an equivalence $\Fun_{\Sigma}(\CAlg^{\an}_{R},\mathcal{D})\simeq\Fun(\CAlg^{\poly}_{R},\mathcal{D})$ for any $\infty$-category $\mathcal{D}$ admitting sifted colimits, given by restrictions along the Yoneda embedding $\CAlg^{\poly}_{R}\hookrightarrow\CAlg^{\an}_{R}$ in one direction and by left Kan extensions in the reverse direction. Here, the left hand side of the equivalence is the full subcategory of $\Fun(\CAlg^{\an}_{R},\mathcal{D})$ spanned by sifted colimit preserving functors. 
\end{remarkn}
Now, let $(A,(d))$ be a perfect prism. Consider the Nygaard filtration construction $\Fil^{\sbullet}_{\N}F^{\ast}(\prism_{-/A}\{n\}):\CAlg^{\an}_{\overline{A}}\to\DFh(A)$, cf. \cite[Sec. 15]{bs2} and \cite[5.1]{bl}, and its filtration completion; the latter gives a lifting $\Fil^{\sbullet}_{\N}\prismc_{-}\{n\}:\CAlg^{\an}_{\overline{A}}\to\DFhc(A)$ of the filtration-completed Nygaard filtration in $\Dh(\bbZ_{p})$ to $\DFhc(A)$, cf. \cite[Th. 5.6.2 and Def. 5.8.5]{bl}. 

\begin{remarkn}
For each $n\in\bbZ$, $\Fil^{\sbullet}_{\N}\prismc_{-}\{n\}:\CAlg^{\an}_{\overline{A}}\to\DFhc(A)$ commutes with sifted colimits and satisfies $p$-completely faithfully flat descent. This follows from the analogous properties of $\gr^{m}_{\N}\prismc_{-}\{n\}\simeq\gr^{m}_{\N}F^{\ast}(\prism_{-/A}\{n\}):\CAlg^{\an}_{\overline{A}}\to\Dh(A)$ for each $m\in\bbZ$. 
\end{remarkn}

Through composition with the inclusion $\Dh(A)\to\mathcal{D}(A)^{\Cpl(p)}$, we view each $\Fil^{m}_{\N}F^{\ast}(\prism_{-/A}\{n\})$ and $\Fil^{m}_{\N}\prismc_{-}\{n\}$ as functors $\CAlg^{\an}_{\overline{A}}\to\mathcal{D}(A)^{\Cpl(p)}$.

\begin{lemma}\label{lem:fibid}
Let $(A,(d))$ be a perfect prism, and let $m,n\in\bbZ$.\\
(1) For each $R\in\CAlg^{\an}_{\overline{A}}$, the natural map $\fib(\Fil^{m}_{\N}F^{\ast}(\prism_{R/A}\{n\})\to F^{\ast}(\prism_{R/A}\{n\}))\to\fib(\Fil^{m}_{\N}\prismc_{R}\{n\} \to \prismc_{R}\{n\})$ 
between the fibers of canonical maps given by filtrations is an equivalence in $\mathcal{D}(A)^{\Cpl(p)}$. \\
(2) Let $F:\CAlg^{\an}_{\overline{A}}\to\mathcal{D}(A)^{\Cpl(p)}$ be a common fiber of (1). Then, $F$ commutes with sifted colimits and satisfies $p$-completely faithfully flat descent. \\
(3) Let $F[\frac{1}{d}]:\CAlg^{\an}_{\overline{A}}\to\Dh(\bbZ_{p})$ be the composition of $F$ in (2) with the functor $M\mapsto M[\frac{1}{d}]:\mathcal{D}(A)^{\Cpl(p)}\to\Dh(\bbZ_{p})$. Then, $F[\frac{1}{d}]$ commutes with sifted colimits. 
\begin{proof}
As the object $\Fil^{\sbullet}_{\N}\prismc_{R}\{n\}$ is given as a completion of the filtered object $\Fil^{\sbullet}F^{\ast}(\prism_{R/A}\{n\})$, the natural diagram
\begin{equation*}
\begin{tikzcd}
\Fil^{m}_{\N}F^{\ast}(\prism_{R/A}\{n\}) \arrow[r] \arrow[d] & F^{\ast}(\prism_{R/A}\{n\}) \arrow[d] \\
\Fil^{m}_{\N}\prismc_{R}\{n\} \arrow[r] & \prismc_{R}\{n\}
\end{tikzcd}
\end{equation*}
is a pullback square in $\mathcal{D}(A)^{\Cpl(p)}$, and hence (1) follows. \\
\indent (2) follows from the corresponding properties of the functor $R\mapsto \gr^{m}_{\N}F^{\ast}(\prism_{R/A}\{n\}):\CAlg^{\an}_{\overline{A}}\to\mathcal{D}(A)^{\Cpl(p)}$; since its composition with $\mathcal{D}(A)^{\Cpl(p)}\to\Dh(\bbZ_{p})$ satisfies such properties, it suffices to note that the functor $\mathcal{D}(A) = \Mod_{A}(\mathcal{D}(\bbZ_{p}))\to\mathcal{D}(\bbZ_{p})$ is conservative and that $\mathcal{D}(A)^{\Cpl(p)}\to\mathcal{D}(\bbZ_{p})^{\Cpl(p)}=\Dh(\bbZ_{p})$ commutes with small limits and colimits. The latter statement follows from the corresponding properties of $\mathcal{D}(A)\to\mathcal{D}(\bbZ_{p})$, the fact that $p$-complete objects are closed under small limits of modules (hence the case of limit follows), and that $p$-completion is independent of the underlying modules, i.e., is preserved by $\mathcal{D}(A)\to\mathcal{D}(\bbZ_{p})$ \cite[7.3.3.6]{sag} (hence the case of colimit follows). \\
\indent Finally, we check that the functor $M\mapsto M[\frac{1}{d}]$ commutes with small colimits, which together with (2) verifies (3). Note that the functor is expressed as a composition $\mathcal{D}(A)^{\Cpl(p)}\xrightarrow{M\mapsto M[d^{-1}]}\mathcal{D}(A)^{\Cpl(p)}\to\mathcal{D}(\bbZ_{p})^{\Cpl(p)} = \Dh(\bbZ_{p})$. The first functor $M\mapsto \colim(M\xrightarrow{d}M\xrightarrow{d}\cdots)$ (where the colimit is computed in $\mathcal{D}(A)^{\Cpl(p)}$) by definition commutes with colimits in $\mathcal{D}(A)^{\Cpl(p)}$, and we already recalled during the proof of (2) that the second functor in the composition commutes with colimits. 
\end{proof}
\end{lemma}

\begin{remarkn}\label{rem:invertd}
In (3) of Lemma \ref{lem:fibid}, the image of the functor $M\mapsto M[\frac{1}{d}]$ in $\Dh(\bbZ_{p})$ is by definition given as $\colim(M\xrightarrow{d}M\xrightarrow{d}\cdots)^{\wedge_{p}}$, where $M\xrightarrow{d}M$ is the image of the multiplication by $d$ map in $\mathcal{D}(\bbZ_{p})$ and the colimit is taken in $\mathcal{D}(\bbZ_{p})$. 
\end{remarkn}

Following \cite{bl}, we write $\CAlg^{\QSyn}$ for the full subcategory of the category $\CAlg$ of (ordinary) commutative rings spanned by $p$-quasisyntomic rings. The full subcategory of the category $\CAlg_{\overline{A}}$ of commutative $\overline{A}$-algebras spanned by $p$-quasisyntomic rings which are $\overline{A}$-algebras is denoted by $\CAlg^{\QSyn}_{\overline{A}/}$. We also write $\CAlg^{\qrsp}_{\overline{A}}$ (resp. $\CAlg^{\perfd}_{\overline{A}}$) for the full subcategory of $\CAlg^{\QSyn}_{\overline{A}/}$ spanned by quasiregular semiperfectoid rings (resp. perfectoid rings) which are $\overline{A}$-algebras. 

\begin{proposition}\label{prop:candisom}
Let $(A,(d))$ be a perfect prism, and let $R\in\CAlg^{\an}_{\overline{A}}$. Then, for each $n\geq 0$, there are equivalences $\Fil^{n}_{\N}(\prismc_{R}\{n\})[\frac{1}{d}]\xrightarrow{\sim}\prismc_{R}\{n\}[\frac{1}{d}]$ and $\Fil^{n}_{\N}F^{\ast}(\prism_{R/A}\{n\})[\frac{1}{d}]\xrightarrow{\sim}F^{\ast}\prism_{R/A}\{n\}[\frac{1}{d}]$ in $\Dh(\bbZ_{p})$ induced by the canonical maps. 
\begin{proof}
By Lemma \ref{lem:fibid} (1), it suffices to check that $F[\frac{1}{d}]\simeq0$. Since $F[\frac{1}{d}]:\CAlg^{\an}_{\overline{A}}\to\Dh(\bbZ_{p})$ is left Kan extended from $\CAlg^{\QSyn}_{\overline{A}/}$ by Lemma \ref{lem:fibid} (3), it suffices to check the vanishing of $F(R)[\frac{1}{d}]$ for $R\in\CAlg^{\QSyn}_{\overline{A}/}$. To verify this, first suppose that $R$ is a quasiregular semiperfectoid ring admitting an algebra structure over $\overline{A}$. In this case, $d^{n}$ acts as zero on the cofiber, concentrated in degree 0, of the canonical map $\Fil^{n}_{\N}F^{\ast}(\prism_{R/A}\{n\})\hookrightarrow F^{\ast}\prism_{R/A}\{n\}$, cf. \cite[Cor. 5.6.3]{bl}. To deduce the case of a general $R\in\CAlg^{\QSyn}_{\overline{A}/}$, take a $p$-quasisyntomic cover $R\to S$ into a quasiregular semiperfectoid $S\in\CAlg^{\qrsp}_{\overline{A}}$. Note that each term $S'$ of the associated $p$-completed Cech nerve is quasiregular semiperfectoid \cite[Lem. 4.30]{bms2}, and hence each $F(S')$ is concentrated in cohomological degree 1. Thus, combining that $F$ on $\CAlg^{\QSyn}_{\overline{A}/}$ satisfies descent with respect to the $p$-completed Cech nerve associated with $R\to S$ by Lemma \ref{lem:fibid} (2), that totalizations commute with filtered colimits in the coconnective part of derived $\infty$-categories, and the previous observation that $d^{n}$ acts as zero on each $F(S')$, we conclude that $F(R)[\frac{1}{d}]\simeq0$ as desired. 
\end{proof}
\end{proposition}

For $R\in\CAlg^{\an}_{\overline{A}}$, let $\Fil^{\sbullet}\varphi\{n\}:\Fil^{\sbullet}_{\N}F^{\ast}(\prism_{R/A}\{n\})\to d^{\sbullet-n}\prism_{R/A}\{n\}$ be the relative Frobenius morphism. Also, write $\varphi\{n\} = \Fil^{n}\varphi\{n\}:\Fil^{n}_{\N}F^{\ast}(\prism_{R/A}\{n\})\to\prism_{R/A}\{n\}$. Due to completeness of the target filtered object, $\Fil^{\sbullet}\varphi\{n\}$ induces a morphism $\Fil^{\sbullet}\widetilde{\varphi}\{n\}:\Fil^{\sbullet}_{\N}\prismc_{R}\{n\}\to d^{\sbullet-n}\prism_{R/A}\{n\}$, and in particular gives $\widetilde{\varphi}\{n\} = \Fil^{n}\widetilde{\varphi}\{n\}:\Fil^{n}_{\N}\prismc_{R}\{n\}\to\prism_{R/A}\{n\}$. Composition of the latter map with the natural map to the completion gives $\widehat{\varphi}\{n\}:\Fil^{n}_{\N}\prismc_{R}\{n\}\to\prismc_{R}\{n\}$. Understood as maps in $\mathcal{D}(A)^{\Cpl(p)}$, these maps induce endomorphisms $\varphi\{n\}[\frac{1}{d}]$ and $\widehat{\varphi}\{n\}[\frac{1}{d}]$ (which we abbreviate as $F$ and $\widehat{F}$ respectively in Construction \ref{constr:frobfix} below) after inverting $d$ through equivalences provided by Proposition \ref{prop:candisom}.  

\begin{construction}[Frobenius fixed point modules]\label{constr:frobfix}
Fix $R\in\CAlg^{\an}_{\overline{A}}$ and $n\in\bbZ$. \\
(1) We write $\left(\prismc_{R}\{n\}[\frac{1}{d}]\right)^{F=1}$ for an equalizer of the identity map of $\prismc_{R}\{n\}[\frac{1}{d}]$ and the endomorphism 
\begin{equation*}
\widehat{F}:\prismc_{R}\{n\}[1/d]\simeq \Fil^{n}_{\N}(\prismc_{R}\{n\})[1/d]\xrightarrow{\widehat{\varphi}\{n\}\left[\frac{1}{d}\right]}\prismc_{R}\{n\}[1/d]
\end{equation*}
in $\Dh(\bbZ_{p})$.  \\
(2) We similarly define $\left(\prism_{R/A}\{n\}[\frac{1}{d}]\right)^{F=1}$ as follows. First, consider the map 
\begin{equation*}
F:F^{\ast}(\prism_{R/A}\{n\})[1/d]\simeq \Fil^{n}_{\N}F^{\ast}(\prism_{R/A}\{n\})[1/d]\xrightarrow{\varphi\{n\}\left[\frac{1}{d}\right]}\prism_{R/A}\{n\}[1/d]
\end{equation*}
in $\Dh(\bbZ_{p})$, where the first equivalence is provided by Proposition \ref{prop:candisom}. \\
\indent Then, let $\varphi$ be the Frobenius automorphism of $A$, and consider the natural $\varphi$-semilinear equivalence $c:M\to F^{\ast}M$ for each $M\in\Dh(A)$ given by tensoring with $\varphi$, where $F^{\ast}M\simeq A\dth_{\varphi, A}M$ in $\Dh(A)$, corresponding to the identity map of $F^{\ast}M$. Since $\varphi(d)\equiv d^{p}$ mod $(p)$, the $\varphi$-semilinear equivalence $c$ induces an equivalence $c:M[\frac{1}{d}]\simeq (F^{\ast}M)[\frac{1}{d}]$, still denoted by $c$, in $\Dh(\bbZ_{p})$ (see Remark \ref{rem:invertd} for our convention). \\
\indent Under the identification $c:\prism_{R/A}\{n\}[\frac{1}{d}]\simeq F^{\ast}(\prism_{R/A}\{n\})[\frac{1}{d}]$ from the previous paragraph, we define $\left(\prism_{R/A}\{n\}[\frac{1}{d}]\right)^{F=1}$ as an equalizer of the identity map of $\prism_{R/A}\{n\}[\frac{1}{d}]$ and the endomorphism $F$ of $\prism_{R/A}\{n\}[\frac{1}{d}]$ in $\Dh(\bbZ_{p})$. More precisely, we have a pullback square
\begin{equation*}
\begin{tikzcd}
\left(\prism_{R/A}\{n\}[\frac{1}{d}]\right)^{F=1} \arrow[r] \arrow[d] &[4em] F^{\ast}(\prism_{R/A}\{n\})[\frac{1}{d}]  \arrow[d,"(c^{-1}{,}~1)"] \\ [2ex]
F^{\ast}(\prism_{R/A}\{n\})[\frac{1}{d}] \arrow[r,"(F{,}~1)"'] & \prism_{R/A}\{n\}[\frac{1}{d}]\times F^{\ast}(\prism_{R/A}\{n\})[\frac{1}{d}]
\end{tikzcd}
\end{equation*}
in $\Dh(\bbZ_{p})$ describing the object $\left(\prism_{R/A}\{n\}[\frac{1}{d}]\right)^{F=1}$. 
\end{construction}

\begin{proposition}\label{prop:nygdc}
Let $(A,(d))$ be a perfect prism, and let $R\in\CAlg^{\an}_{\overline{A}}$. Then, for each $n\in\bbZ$, the natural map $\left(\prism_{R/A}\{n\}[\frac{1}{d}]\right)^{F=1}\to \left(\prismc_{R}\{n\}[\frac{1}{d}]\right)^{F=1}$ is an equivalence in $\Dh(\bbZ_{p})$. 
\begin{proof}
We provide an argument based on the idea of \cite[Proof of Prop. 7.4.6]{bl}. Analogous to the maps $F$ and $\widehat{F}$ of Construction \ref{constr:frobfix}, write $\widetilde{F}$ for the map
\begin{equation*}
\widetilde{F}:\prismc_{R}\{n\}[1/d]\simeq \Fil^{n}_{\N}(\prismc_{R}\{n\})[1/d]\xrightarrow{\widetilde{\varphi}\{n\}\left[\frac{1}{d}\right]}\prism_{R/A}\{n\}[1/d]. 
\end{equation*}
Also, let $\imath:F^{\ast}(\prism_{R/A}\{n\})[\frac{1}{d}]\to\prismc_{R}\{n\}[\frac{1}{d}]$ be the canonical map induced from the map into the Nygaard completion. Consider the diagram 
\begin{equation*}
\begin{tikzcd}
F^{\ast}(\prism_{R/A}\{n\})[\frac{1}{d}] \arrow[r,"(F{,}~1)"] \arrow[d,"\imath"'] &[3em] \prism_{R/A}\{n\}[\frac{1}{d}]\times F^{\ast}(\prism_{R/A}\{n\})[\frac{1}{d}] \arrow[d,"1\times\imath"] &[2em] F^{\ast}(\prism_{R/A}\{n\})[\frac{1}{d}] \arrow[l,"(c^{-1}{,}~1)"'] \arrow[d,"="]\\ [2ex]
\prismc_{R}\{n\}[\frac{1}{d}] \arrow[r, "(\widetilde{F}{,}~1)"] \arrow[d,"="'] & \prism_{R/A}\{n\}[\frac{1}{d}]\times\prismc_{R}\{n\}[\frac{1}{d}] \arrow[d,"(\imath\circ c)\times1"] & F^{\ast}(\prism_{R/A}\{n\})[\frac{1}{d}] \arrow[l,"(c^{-1}{,}~\imath)"] \arrow[d,"\imath"]\\ [2ex]
\prismc_{R}\{n\}[\frac{1}{d}] \arrow[r,"(\widehat{F}{,}~1)"'] & \prismc_{R}\{n\}[\frac{1}{d}]\times\prismc_{R}\{n\}[\frac{1}{d}] & \prismc_{R}\{n\}[\frac{1}{d}] \arrow[l,"(1{,}~1)"]
\end{tikzcd}
\end{equation*}
in $\Dh(\bbZ_{p})$. Composition of the two maps between limits of the rows of the diagram gives the comparison map $\left(\prism_{R/A}\{n\}[1/d]\right)^{F=1}\to\left(\prismc_{R}\{n\}[1/d]\right)^{F=1}$ by Construction \ref{constr:frobfix}. The map from the limit of the first row to the limit of the second row is an equivalence, since the upper left square is a pullback square. In fact, the upper left square fits into the diagram
\begin{equation*}
\begin{tikzcd}
F^{\ast}(\prism_{R/A}\{n\})[\frac{1}{d}] \arrow[r,"(F{,}~1)"] \arrow[d,"\imath"'] &[3em] \prism_{R/A}\{n\}[\frac{1}{d}]\times F^{\ast}(\prism_{R/A}\{n\})[\frac{1}{d}] \arrow[r,"\mathrm{pr}_{2}"] \arrow[d,"1\times\imath"] &[2em] F^{\ast}(\prism_{R/A}\{n\})[\frac{1}{d}] \arrow[d,"\imath"] \\ [2ex]
\prismc_{R}\{n\}[\frac{1}{d}] \arrow[r,"(\widetilde{F}{,}~1)"'] & \prism_{R/A}\{n\}[\frac{1}{d}]\times\prismc_{R}\{n\}[\frac{1}{d}] \arrow[r,"\mathrm{pr}_{2}"'] & \prismc_{R}\{n\}[\frac{1}{d}]
\end{tikzcd}
\end{equation*}
as the left side square, and since the right side square of the diagram is a pullback square by construction and the outer rectangle of the diagram is trivially a pullback square, the claim follows. The map from the limit of the second row to the limit of the third row is an equivalence, since the lower right square is a pullback square by construction. Combining these verifies the claimed result.   
\end{proof}
\end{proposition}

\subsection{\'Etale comparison}\label{subsec:etcomp}

We write $\CAlg^{\an,p-\cpl}$ for the full subcategory of $\CAlg^{\an}$ spanned by $p$-complete animated commutative rings, and write $\CAlg^{p-\cpl}$ for its further subcategory spanned by $p$-complete animated commutative rings which are ordinary commutative rings. Note that the inclusions $\CAlg^{p-\cpl}\hookrightarrow\CAlg^{\an,p-\cpl}$ and $\CAlg^{\an,p-\cpl}\hookrightarrow\CAlg^{\an}$ are right adjoint functors admitting left adjoints $\pi_{0}$ and $(-)^{\wedge_{p}}$ respectively. We will also consider the slice categories $\CAlg^{\an,p-\cpl}_{\overline{A}}$ and $\CAlg^{p-\cpl}_{\overline{A}}$ under a perfectoid ring $\overline{A}$. Recall that perfectoid rings form a basis for the $\arcp$-topology on $\CAlg^{p-\cpl}$ \cite[Lem. 8.8]{bs2}; \emph{loc. cit.} uses that semiperfectoids form a basis for the $p$-complete fpqc topolgy on $\CAlg^{p-\cpl}$, and that the natural map from a semiperfectoid to its perfectoidization is an $\arcp$-equivalence.  \\
\indent We record the following twisted variant of the comparison result \cite[Th. 9.1]{bs2} for \'etale and prismatic cohomology, taking Tate twists $\bbZ_{p}(n)$ on the \'etale cohomology side and Breuil-Kisin twists $A\{n\}$ on the prismatic cohomology side into account: 

\begin{proposition}(\'Etale comparison)\label{prop:etcomp}
Let $(A,(d))$ be a perfect prism, and let $n\in\bbZ$. Then, for $R\in\CAlg^{\an,p-\cpl}_{\overline{A}}$, there is a natural equivalence 
\begin{equation}\label{eq:etcomp}
 \RG_{\et}(\Spec R[1/p],\bbZ_{p}(n))\simeq\left(\prism_{R/A}\{n\}[1/d]\right)^{F=1}
\end{equation} 
in $\Dh(\bbZ_{p})$.
\end{proposition}

This has also appeared in \cite[Th. 4.11]{mw} and \cite[Th. 4.4]{bouis}; we provide an alternative algebraic proof below which in particular identifies the comparison map with the prismatic logarithm map. Moreover, we will later give a K-theoretic proof of this comparison in Corollary \ref{cor:etcomp} which does not rely on these algebraic approaches. 

\begin{remarkn}\label{rem:etcomp}
We will not use Proposition \ref{prop:etcomp} for the most part of this article; the exceptions are all for $\Fil^{\sbullet}_{\mathrm{lke}}L_{T(1)}\TC$ which we do not use for the K-theoretic proof of the \'etale comparison. More precisely, the following two are the only cases we use Proposition \ref{prop:etcomp}, for the sake of expositional convenience:  \\
(1) Description of $\gr^{n}(L_{T(1)}\TC)^{\sharp}$ in Construction \ref{constr:filtonarcpsheaf} in terms of the \'etale cohomology at the generic fiber. This is for Proposition \ref{prop:filtvialke} (2) concerning $\gr^{n}\Fil_{\mathrm{lke}}L_{T(1)}\TC$, and we otherwise use the direct description via prismatic cohomology.\\
(2) Corollary \ref{cor:T(1)TCgrlke} and its consequences, i.e., the description of $\gr^{n}\Fil_{\mathrm{lke}}L_{T(1)}\TC$ in Proposition \ref{prop:filtvialke} (2), the comparison of $\Fil^{\sbullet}_{\mathrm{lke}}L_{T(1)}\TC$ with $\Fil^{\sbullet}(L_{T(1)}\TC)^{\sharp}$ in Proposition \ref{prop:filtcompletecomparison}, and Theorem \ref{th:T(1)TCfilt} (2). In Corollary \ref{cor:T(1)TCgrlke}, we check that the functor $R\mapsto \RG_{\et}(\Spec R^{\wedge_{p}}[1/p],\bbZ_{p}(n))$ on $\CAlg^{\an}$ (not necessarily over a perfectoid base ring) commutes with sifted colimits using Proposition \ref{prop:etcomp}.
\end{remarkn}

Before giving an algebraic proof of Proposition \ref{prop:etcomp}, let us record the following $\arcp$-hyperdescent property of the right hand side functor of (\ref{eq:etcomp}) in the étale comparison theorem, which will also be used later in this paper. 

\begin{lemma}\label{lem:prismarcphypersheaf}
Let $(A,(d))$ be a perfect prism, and let $n\in\bbZ$. Then, the functor
\begin{equation*}
R\mapsto \left(\prism_{R/A}\{n\}[1/d]\right)^{F=1}:\CAlg^{\an, p-\cpl}_{\overline{A}}\to\Dh(\bbZ_{p})
\end{equation*} 
factors through $\pi_{0}:\CAlg^{\an,p-\cpl}_{\overline{A}}\to\CAlg^{p-\cpl}_{\overline{A}}$ and is an $\arcp$-hypersheaf on $\CAlg^{p-\cpl}_{\overline{A}}$.
\begin{proof}
First, by \cite[Lem. 9.2]{bs2}, there is a canonical equivalence $\left(\prism_{R/A}\{n\}[\frac{1}{d}]\right)^{F=1}\simeq \left(\prism_{R/A,\perf}\{n\}[\frac{1}{d}]\right)^{F=1}$ for $R\in\CAlg^{\an,p-\cpl}_{\overline{A}}$. The functor $\prism_{-/A,\perf}\{n\}:\CAlg^{\an,p-\cpl}_{\overline{A}}\to\Dh(A)$ is truncating, i.e., it factors through the functor $\pi_{0}$ \cite[Prop. 8.5]{bs2}, from which the first part of the statement holds. The functor $\prism_{-/A,\perf}\{n\}$ satisfies $p$-complete $\arc$-descent property on $\CAlg^{p-\mathrm{cpl}}_{\overline{A}}$ \cite[Cor. 8.11]{bs2} and its values are coconnective \cite[Lem. 8.4]{bs2}. Thus, as totalizations commute with filtered colimits in the coconnective part of derived $\infty$-categories, the functor $\left(\prism_{-/A,\perf}\{n\}[\frac{1}{d}]\right)^{F=1}:\CAlg^{p-\cpl}_{\overline{A}}\to\Dh(\bbZ_{p})$ satisfies $p$-complete $\arc$-descent, and as the functor vanishes on $\bbF_{p}$-algebras by \cite[Prop. 8.5]{bs2}, it satisfies $\arcp$-descent. Finally, as the functor $\left(\prism_{-/A,\perf}\{n\}[\frac{1}{d}]\right)^{F=1}$ takes values in the coconnective part, it is an $\arcp$-hypersheaf. 
\end{proof}
\end{lemma}

\begin{proof}[Proof of Proposition \ref{prop:etcomp}, algebraically]
We closely follow arguments in \cite[Sec. 9]{bs2}; note that the case of $n=0$ is precisely \cite[Th. 9.1]{bs2}. First, observe that both left and right side functors of the equivalence (\ref{eq:etcomp}) are right Kan extended from the full subcategory $\mathcal{C} = \CAlg^{\perfd}_{\overline{A},\bbZ_{p}^{\cyc}}$ of perfectoid $\overline{A}$-algebras each of which admits a $\bbZ_{p}^{\cyc}$-algebra structure. In fact, both functors factor through $\pi_{0}:\CAlg^{\an,p-\cpl}_{\overline{A}}\to\CAlg^{p-\cpl}_{\overline{A}}$ and are $\arcp$-hypersheaves on $\CAlg^{p-\cpl}_{\overline{A}}$. For the functor $R\mapsto \RG_{\et}(\Spec R[\frac{1}{p}],\bbZ_{p}(n))$, the latter claim follows from \cite[Cor. 6.17]{arc}, while Lemma \ref{lem:prismarcphypersheaf} explains the case of $R\mapsto \left(\prism_{R/A}\{n\}[\frac{1}{d}]\right)^{F=1}$. Thus, it suffices to construct the equivalence (\ref{eq:etcomp}) on $\mathcal{C}$.\\
\indent For $R\in\mathcal{C}$, suppose that the map $\overline{A}\to R$ corresponds to a map $(A,(d))\to (A'',(d''))$ of perfect prisms; in particular, $R\simeq \overline{A''}$. Then, the natural map $\prism_{R/A}\{n\}[\frac{1}{d}]\to\prism_{R/A''}\{n\}[\frac{1}{d''}] = A''\{n\}[\frac{1}{d''}]$ is an equivalence in $\mathcal{D}(A)^{\Cpl(p)}$ respecting Frobenii. Through the same argument applied to $\bbZ_{p}^{\cyc}\to R$, we have a natural equivalence $\prism_{R/A}\{n\}[\frac{1}{d}]\simeq\prism_{R/A'}\{n\}[\frac{1}{d'}]$ in $\Dh(\bbZ_{p})$ respecting Frobenii, where $(A',(d')) = (\bbZ_{p}[[q-1]],([p]_{q}))_{\perf}$ is the perfect prism corresponding to $\bbZ_{p}^{\cyc}$. To construct the comparison map, consider the composition 
\begin{equation*}
T_{p}(R^{\times})^{\otimes_{\bbZ_{p}}n}\xrightarrow{\log_{\prism}^{\otimes_{\bbZ_{p}} n}} A''\{1\}^{\otimes_{\bbZ_{p}}n}\to A''\{n\}\to A''\{n\}[1/d'']
\end{equation*} 
for each $n\geq0$. Here, $\log_{\prism}:T_{p}(R^{\times}) = T_{p}(\overline{A''}^{\times})\to A''\{1\}$ is the prismatic logarithm map associated with $(A'',(d''))$. By the invariance of $\log_{\prism}$ under Frobenius \cite[Prop. 2.5.18]{bl}, this composite map factors through $\left(A''\{n\}[\frac{1}{d''}]\right)^{F=1}$, and by the naturality of $\log_{\prism}$, the resulting induced map 
\begin{equation}\label{eq:etcomplog1}
\log_{\prism}^{\otimes n}: T_{p}(R^{\times})^{\otimes_{\bbZ_{p}} n} \to \left(\prism_{R/A}\{n\}[1/d]\right)^{F=1} 
\end{equation}
collects to a map of presheaves on $\mathcal{C}$. After $\arcp$-sheafification, the source by definition gives a sheaf $R\mapsto \RG_{\arcp}(\Spec R,\bbZ_{p}(n))$, cf. \cite[Constr. 7.3]{matTR}. Also, there is a natural equivalence $\RG_{\arcp}(\Spec R,\bbZ_{p}(n))\simeq \RG_{\et}(\Spec R[1/p],\bbZ_{p}(n))$ \cite[Prop. 7.4]{matTR}. Composing this with the sheaf map induced from the previous map $\log_{\prism}^{\otimes n}$ of presheaves in (\ref{eq:etcomplog1}), we have the comparison map 
\begin{equation}\label{eq:etcomplog2}
\RG_{\et}(\Spec R[1/p],\bbZ_{p}(n))\simeq \RG_{\arcp}(\Spec R,\bbZ_{p}(n))\xrightarrow{\log_{\prism}^{\otimes n}} \left(\prism_{R/A}\{n\}[1/d]\right)^{F=1}
\end{equation} 
of $\arcp$-sheaves on $R\in\mathcal{C}$. \\
\indent Let us further describe the target $R\mapsto\left(\prism_{R/A}\{n\}[\frac{1}{d}]\right)^{F=1}\simeq\left(\prism_{R/A'}\{n\}[\frac{1}{d'}]\right)^{F=1}:\mathcal{C}\to\Dh(\bbZ_{p})$ for any $n\in\bbZ$. Recall that $\Fil^{n}_{\N}F^{\ast}\prism_{R/A'}\{n\}[\frac{1}{d'}]\simeq F^{\ast}\prism_{R/A'}\{n\}[\frac{1}{d'}]$ by Proposition \ref{prop:candisom}, and that composition of its inverse with the Frobenius map gives a description of the endomorphism $F$. More precisely, consider the commutative diagram
\begin{equation*}
\begin{tikzcd}
F^{\ast}(\prism_{R/A'}\{n\})[\frac{1}{d'}] \arrow[d, "\simeq"'] & \Fil^{n}_{\N}F^{\ast}(\prism_{R/A'}\{n\})[\frac{1}{d'}] \arrow[l, "\sim"'] \arrow[r, "\varphi\{n\}{[\frac{1}{d'}]}"] \arrow[d, "\simeq"] &[+25pt] \prism_{R/A'}\{n\}[\frac{1}{d'}] \arrow[d, "\simeq"] \\ [2ex]
F^{\ast}(A'\{n\})\dt_{A'}F^{\ast}\prism_{R/A'}[\frac{1}{d'}] & F^{\ast}(A'\{n\})\dt_{A'}\Fil^{n}_{\N}F^{\ast}\prism_{R/A'}[\frac{1}{d'}] \arrow[l, "\sim"] \arrow[r, "\varphi_{A'\{n\}}\otimes \varphi{[\frac{1}{d'}]}"'] & A'\{n\}\dt_{A'}\prism_{R/A'}[\frac{1}{d'}]
\end{tikzcd}
\end{equation*}  
whose upper horizontal composition gives $A'[1/d']$-linear $F = \varphi\{n\}[\frac{1}{d'}]$. Using the generator $e_{A'}^{n}$ of $A'\{n\}$ as a free $A'$-module as in \cite[Not. 2.6.3]{bl}, the lower horizontal composition can be computed as a map $\lambda e_{A'}^{n}\otimes x/(d')^{r}\mapsto \lambda e^{n}_{A'}\otimes (d')^{n}x/(d')^{n+r}\mapsto \frac{\varphi(\lambda)}{(d')^{n}}e^{n}_{A'}\otimes (d')^{n}\varphi(x)/(d')^{n+r} = \varphi(\lambda)e^{n}_{A'}\otimes \varphi(x)/(d')^{n+r}$ (where $\lambda\in A'$ and $x\in\Fil^{n}_{\N}F^{\ast}\prism_{R/A'}$). Thus, through the trivialization of $A'\{n\}$ via $e^{n}_{A'}$, one has $\varphi\{n\}[\frac{1}{d'}] \simeq \frac{1}{(d')^{n}}\cdot\varphi\{0\}[\frac{1}{d'}]$, and hence there is an equivalence $\left(\prism_{R/A'}[\frac{1}{d'}]\right)^{F=1}\simeq \left(\prism_{R/A'}\{n\}[\frac{1}{d'}]\right)^{F=1}$ induced by the multiplication of $(q-1)^{n}e_{A'}^{n}$. \\
\indent Note that $\log_{\prism}(\epsilon) = (q-1)e_{A'}$ for $\epsilon = (q^{p}, q, q^{1/p},...)\in T_{p}(A')^{\times}$, so the equivalence above is the multiplication of $\log_{\prism}(\epsilon)^{n}$. On the other hand, for $R\in\mathcal{C}$, we know there is an equivalence $\RG_{\et}(\Spec R[1/p], \bbZ_{p})\simeq\RG_{\et}(\Spec R[1/p],\bbZ_{p}(n))$ induced by the multiplication of $\epsilon^{n}$ for each $n\in\bbZ$, cf. \cite[Rem. 8.2.4]{bl}. Thus, we have a natural diagram
\begin{equation}\label{eq:etcomplogdiag}
\begin{tikzcd}
\RG_{\et}(\Spec R[1/p],\bbZ_{p}) \arrow[r, "\epsilon^{n}", "\sim"'] \arrow[d, "\log_{\prism}^{\otimes 0}"'] & \RG_{\et}(\Spec R[1/p],\bbZ_{p}(n)) \arrow[d, "\log_{\prism}^{\otimes n}"] \\
\left(\prism_{R/A}[1/d]\right)^{F=1} \arrow[r, "\log_{\prism}(\epsilon)^{n}"', "\sim"] & \left(\prism_{R/A}\{n\}[1/d]\right)^{F=1}
\end{tikzcd}
\end{equation}
for $R\in\mathcal{C}$ and $n\in\bbZ$; for $n\geq0$ this follows from our construction of the map (\ref{eq:etcomplog2}), and for $n<0$ we define $\log_{\prism}^{\otimes n}$ to be the map fitting into the diagram above. Note that by right Kan extension, we also have the comparison map of (\ref{eq:etcomp}) on $\CAlg^{\an,p-\cpl}_{\overline{A}}$ which we denote as 
\begin{equation}\label{eq:etcomplog3}
\log_{\prism}^{\otimes n}: \RG_{\et}(\Spec R[1/p],\bbZ_{p}(n))\to \left(\prism_{R/A}\{n\}[1/d]\right)^{F=1}
\end{equation} 
for $R\in\CAlg^{\an,p-\cpl}_{\overline{A}}$ and $n\in\bbZ$. \\
\indent Thus, in order to to check that the map $\RG_{\et}(\Spec R[\frac{1}{p}],\bbZ_{p}(n))\to \left(\prism_{R/A}\{n\}[\frac{1}{d}]\right)^{F=1}$ of (\ref{eq:etcomplog3}) is an equivalence, it suffices to check the property for $R\in\mathcal{C}$ and for the case of $n=0$ given by $\log_{\prism}^{\otimes 0}$, which by definition is the case of \cite[Th. 9.1]{bs2}. In \emph{loc. cit.}, this case was verified by checking that modulo $p$ the map induces an equivalence $\RG_{\et}(\Spec R[\frac{1}{p}],\bbF_{p})\simeq \left(R^{\flat}[\frac{1}{d}]\right)^{F=1}$, which is a consequence of the Artin-Schreier fiber sequence $\RG_{\et}(\Spec R^{\flat}[\frac{1}{d}],\bbF_{p}) \to R^{\flat}[\frac{1}{d}] \xrightarrow{\varphi-1} R^{\flat}[\frac{1}{d}]$ and the natural equivalence $\RG_{\et}(\Spec R[\frac{1}{p}],\bbF_{p})\simeq \RG_{\et}(\Spec R^{\flat}[\frac{1}{d}],\bbF_{p})$. 
\end{proof}

\section{Thomason filtration and $\arcp$-hyperdescent}

In this section, we verify the main results stated in \ref{subsec:mainresults} and \ref{subsec:statement}. In \ref{subsec:filtconstr}, we construct and describe a filtration on an $\arcp$-hypersheafification of $L_{T(1)}\TC$ (Construction \ref{constr:filtonarcpsheaf}). After verifying that $\arcp$-hypersheafification does not change values of $L_{T(1)}\TC$ for finitely generated polynomial rings through \ref{subsec:progal} and \ref{subsec:arcphyp}, we construct a filtration on $L_{T(1)}\TC$ from the previous filtration via left Kan extension in \ref{subsec:filtconstr2} (Construction \ref{constr:filtvialke}). On the other hand, we could directly construct a filtration on $L_{T(1)}\TC$ via pullback from the filtration on its $\arcp$-hypersheafification (Construction \ref{constr:filtviapb}), and we prove that this filtration is equivalent to the former filtration on $L_{T(1)}\TC$ in \ref{subsec:compwithhyp} and obtain various basic properties of the common filtration as a consequence (Theorem \ref{th:T(1)TCfilt}). In \ref{subsec:compwiththomason}, we provide a comparison of our filtration with the Thomason filtration (Theorem \ref{th:compwiththomason}) and obtain the \'etale comparison for prismatic cohomology as a consequence (Corollary \ref{cor:etcomp}).    

\subsection{Statement of the main theorems}\label{subsec:statement}

Let us fix relevant notations and restate more precise versions of the main results in \ref{subsec:mainresults} to be proved in this section. Write $\DFh(\bbS) = \left(\Sp^{\Cpl(p)}\right)^{\mathrm{fil}} = \Fun(\N\bbZ^{\op},\Sp^{\Cpl(p)})$ for the $\infty$-category of $\bbZ$-indexed descending filtrations of $p$-complete spectra; we will call its objects simply as filtered $p$-complete spectra. Note that $\DFh(\bbS)$ has a natural symmetric monoidal structure whose underlying tensor product is given by the Day convolution. We will often call objects of the $\infty$-category $\CAlg(\DFh(\bbS))$, i.e., $\bbE_{\infty}$-algebra objects of filtered $p$-complete spectra, as multiplicative filtered $p$-complete spectra.

\begin{theorem}[Theorem \ref{th:T(1)TCfilt} and Corollary \ref{cor:compwiththomason}] \label{th:statementT(1)TCfilt}
There is a functor $\Fil^{\sbullet}L_{T(1)}\TC$ and a map $\Fil^{\sbullet}L_{T(1)}\TC\to L_{T(1)}\TC$ of $\Fun(\CAlg^{\an},\CAlg(\DFh(\bbS)))$ into the target $L_{T(1)}\TC\in\Fun(\CAlg^{\an},\CAlg(\Sp^{\Cpl(p)}))$ viewed as a constant filtration, which satisfy the following properties: 
\begin{adjustwidth}{12pt}{}
(1) (Exhaustiveness) The map $\Fil^{\sbullet}L_{T(1)}\TC\to L_{T(1)}\TC$ realizes $\Fil^{\sbullet}L_{T(1)}\TC$ as an exhaustive filtration of $L_{T(1)}\TC$. \\
(2) (Completeness) The object $\Fil^{\sbullet}L_{T(1)}\TC(R)$ is filtration-complete for $R\in\CAlg^{\an}$ satisfying the following finiteness condition:
\begin{adjustwidth}{13pt}{}
($\ast$) $\pi_{0}R^{\wedge_{p}}[1/p]$ has finite Krull dimension and admits a uniform bound on the mod $p$ virtual cohomological dimensions of the residue fields.
\end{adjustwidth}
(3) (Preservation of sifted colimits) $\Fil^{\sbullet}L_{T(1)}\TC$ and $L_{T(1)}\TC$ commute with sifted colimits. \\
(4) ($p$-complete fppf descent) $\Fil^{\sbullet}L_{T(1)}\TC$ and $L_{T(1)}\TC$ induce equivalences on taking functors $\pi_{0}:\CAlg^{\an}\to\CAlg$ and $(-)^{\wedge_{p}}:\CAlg^{\an}\to \CAlg^{\an,p-\cpl}$, and satisfy $p$-complete fppf descent.\\
(5) (Associated graded pieces) For each $n\in\bbZ$, write $\gr^{n}L_{T(1)}\TC$ for the $n$-th associated graded objects functor for $\Fil^{\sbullet}L_{T(1)}\TC$. There is a natural equivalence
\begin{equation*}
\gr^{n}L_{T(1)}\TC(R)\simeq\RG_{\et}(\Spec R^{\wedge_{p}}[1/p],\bbZ_{p}(n))[2n]
\end{equation*} 
in $\Dh(\bbZ_{p})$ for $R\in\CAlg^{\an}$. 
\end{adjustwidth}
\end{theorem}

As indicated earlier, our construction of $\Fil^{\sbullet}L_{T(1)}\TC$ does not rely on the existing Thomason filtration on $L_{T(1)}\K$ and the cyclotomic trace map. Rather, we use the following filtration-completed version $\Fil^{\sbullet}(L_{T(1)}\TC)^{\sharp}$ on an $\arcp$-hypersheafification $(L_{T(1)}\TC)^{\sharp}$ as a useful tool for constructing the filtration on $L_{T(1)}\TC$. Below, we write $\DFhc(\bbS)\subseteq\DFh(\bbS)$ for the full subcategory of $\DFh(\bbS)$ spanned by filtration-complete objects, i.e., filtration-complete filtered $p$-complete spectra.

\begin{theorem}[Construction \ref{constr:filtonarcpsheaf} and Theorem \ref{th:T(1)TCfilt}]\label{th:statementfiltonarcpsheaf}
There exists a functor $\Fil^{\sbullet}(L_{T(1)}\TC)^{\sharp}$ and a map $\Fil^{\sbullet}(L_{T(1)}\TC)^{\sharp}\to (L_{T(1)}\TC)^{\sharp}(\pi_{0}((-)^{\wedge_{p}}))$ of $\Fun(\CAlg^{\an},\CAlg(\DFh(\bbS)))$ which satisfies the following properties:
\begin{adjustwidth}{12pt}{}
(1) The functor $\Fil^{\sbullet}(L_{T(1)}\TC)^{\sharp}$ induces an equivalence on $R\mapsto \pi_{0}(R^{\wedge_{p}}):\CAlg^{\an}\to\CAlg^{p-\cpl}$, and its restriction to $\CAlg^{p-\cpl}$ is an $\arcp$-hypersheaf.\\
(2) The map $\Fil^{\sbullet}(L_{T(1)}\TC)^{\sharp}\to (L_{T(1)}\TC)^{\sharp}(\pi_{0}((-)^{\wedge_{p}}))$ endows the target with a complete and exhaustive filtration given by the source. In particular, $\Fil^{\sbullet}(L_{T(1)}\TC)^{\sharp}$ defines an object of $\Fun(\CAlg^{\an},\DFhc(\bbS))$. \\
(3) For each $n\in\bbZ$ and a perfect prism $(A,(d))$, there is a natural equivalence 
\begin{equation*}
\gr^{n}(L_{T(1)}\TC)^{\sharp}(R)\simeq \left(\prism_{R/A}\{n\}[1/d]\right)^{F=1}[2n]
\end{equation*} 
in $\Dh(\bbZ_{p})$ for $R\in\CAlg^{\an}_{\overline{A}}$.\\
(4) There is a map 
\begin{equation*}
\Fil^{\sbullet}L_{T(1)}\TC\to \Fil^{\sbullet}(L_{T(1)}\TC)^{\sharp}
\end{equation*} 
of $\Fun(\CAlg^{\an},\CAlg(\DFh(\bbS)))$ which realizes the target as a filtration-completion of the source, and induces the natural map $h:L_{T(1)}\TC\to (L_{T(1)}\TC)^{\sharp}(\pi_{0}((-)^{\wedge_{p}}))$ after taking the underlying objects of filtrations.
\end{adjustwidth}
\end{theorem}

Finally, we verify the following comparison with the Thomason filtration on $L_{T(1)}\K$ and obtain the \'etale comparison theorem for prismatic cohomology as a consequence:

\begin{theorem}[Theorem \ref{th:compwiththomason} and Corollary \ref{cor:etcomp}]
\label{th:statementcompwiththomason}
There is a natural diagram 
\begin{equation*}
\begin{tikzcd}
(\Fil^{\sbullet}L_{T(1)}\K)(R^{\wedge_{p}}[1/p]) \arrow[r, "\sim"'] \arrow[d] & (\Fil^{\sbullet}L_{T(1)}\TC)(R) \arrow[d]  \\
L_{T(1)}\K(R^{\wedge_{p}}[1/p])\simeq L_{T(1)}\K(R^{\wedge_{p}}) \arrow[r, "\tr"', "\sim"] & L_{T(1)}\TC(R) 
\end{tikzcd}
\end{equation*}
of $\CAlg(\DFh(\bbS))$ for $R\in\CAlg^{\an}$, where the left vertical arrow is the Thomason filtration $\Fil^{\sbullet}L_{T(1)}\K(R^{\wedge_{p}}[1/p])$ on $L_{T(1)}\K(R^{\wedge_{p}}[1/p])$. Taking the underlying objects of the upper horizontal map of filtered objects recovers the equivalence $L_{T(1)}\K(R^{\wedge_{p}}[1/p])\simeq L_{T(1)}\K(R^{\wedge_{p}}) \underset{\sim}{\xrightarrow{\tr}} L_{T(1)}\TC(R)$ of underlying $p$-complete spectra in the lower horizontal arrow. In particular, taking the associated graded objects of the upper horizontal equivalence gives the natural equivalence
\begin{equation*}
\RG_{\et}(\Spec R^{\wedge_{p}}[1/p],\bbZ_{p}(n)) \simeq \left(\prism_{R/A}\{n\}[1/d]\right)^{F=1}
\end{equation*}
for $R\in\CAlg^{\an}_{\overline{A}}$ and $n\in\bbZ$.
\end{theorem}

\subsection{Construction of the filtration I}\label{subsec:filtconstr}

We start with an analogue of Proposition \ref{prop:candisom} and Construction \ref{constr:frobfix} for localizing invariants. Let $R$ be a connective $\bbE_{1}$-ring. Recall that its $p$-complete topological Hochschild homology $\THH(R,\bbZ_{p})$ is a connective $p$-typical cyclotomic spectrum admitting the cyclotomic Frobenius $\varphi:\THH(R,\bbZ_{p})\to\THH(R,\bbZ_{p})^{tC_{p}}$ \cite{ns}. Taking homotopy $S^{1}$-invariants for the $S^{1}$-equivariant map $\varphi$ gives the Frobenius map $\varphi^{hS^{1}}:\TC^{-}(R,\bbZ_{p})\to \TP(R,\bbZ_{p})$, where $\TC^{-}(R,\bbZ_{p})\simeq \THH(R,\bbZ_{p})^{hS^{1}}$ and $\TP(R,\bbZ_{p})\simeq \THH(R,\bbZ_{p})^{tS^{1}}\simeq (\THH(R,\bbZ_{p})^{tC_{p}})^{h(S^{1}/C_{p})}$, cf. \cite[Lem. II.4.2]{ns}. On the other hand, there is a canonical map $\can:\TC^{-}(R,\bbZ_{p})\to\TP(R,\bbZ_{p})$ coming from the Tate construction.   

\begin{construction}[Frobenius on $T(1)$-local $\TP$]\label{constr:T(1)TP}
Let $R$ be a connective $\bbE_{1}$-ring which is $T(1)$-acyclic, e.g., a connective $\bbE_{1}$-algebra over $H\bbZ$, and consider the fiber sequence 
\begin{equation*}
\THH(R,\bbZ_{p})_{hS^{1}}[1]\to \TC^{-}(R,\bbZ_{p})\xrightarrow{\can}\TP(R,\bbZ_{p})
\end{equation*}
of $p$-complete spectra. Since the full subcategory of $T(1)$-acyclic spectra in $\Sp$ is stable under colimits and $p$-completions of spectra, we know the first term of this fiber sequence is $T(1)$-acyclic. In particular, the map $\can:L_{T(1)}\TC^{-}(R)\to L_{T(1)}\TP(R)$ obtained as the $T(1)$-localization of the canonical map $\TC^{-}(R,\bbZ_{p})\to\TP(R,\bbZ_{p})$ is an equivalence. Through this equivalence, we obtain an endomorphism
\begin{equation*}
F:L_{T(1)}\TP(R)\simeq L_{T(1)}\TC^{-}(R)\xrightarrow{L_{T(1)}\varphi^{hS^{1}}}L_{T(1)}\TP(R)
\end{equation*}
of $L_{T(1)}\TP(R)$. \\
\indent By Nikolaus-Scholze description of topological cyclic homology \cite{ns}, we know $\TC(R,\bbZ_{p})$ is described as an equalizer of $\varphi^{hS^{1}}$ and the canonical map. Taking $T(1)$-localization, we have
\begin{align*}
L_{T(1)}\TC(R) &\simeq \fib(L_{T(1)}\TC^{-}(R)\xrightarrow{L_{T(1)}\varphi^{hS^{1}}-\can}L_{T(1)}\TP(R))\\
                         &\simeq \fib(L_{T(1)}\TP(R)\xrightarrow{F-1}L_{T(1)}\TP(R)),
\end{align*}
or in short, $L_{T(1)}\TC(R)\simeq (L_{T(1)}\TP(R))^{F=1}$ in $\Sp^{\Cpl(p)}$.
\end{construction}

Now, we can directly construct and describe a motivic filtration on $T(1)$-local $\TC$ for quasiregular semiperfectoid rings which are algebras over the perfectoid ring $\mathcal{O}_{C}$ for any complete and algebraically closed nonarchimedean valued field $C$ of mixed characteristic $(0,p)$, e.g., $\bbC_{p} = \widehat{\overline{\bbQ_{p}}}$.\\
\indent Recall that for $R\in\CAlg^{\qrsp}$, the spectra $\THH(R,\bbZ_{p})$, $\TC^{-}(R,\bbZ_{p})$, and $\TP(R,\bbZ_{p})$ are even, and that the maps 
\begin{equation*}
\pi_{2n}(\varphi^{hS^{1}}),\ \pi_{2n}(\can):\pi_{2n}\TC^{-}(R,\bbZ_{p})\to\pi_{2n}\TP(R,\bbZ_{p})
\end{equation*} 
are naturally identified with the maps 
\begin{equation*}
\widehat{\varphi}\{n\},\ \can:\Fil^{n}_{\N}\prismc_{R}\{n\}\to\prismc_{R}\{n\}
\end{equation*} 
respectively for each $n\in\bbZ$, cf. \cite[Section 7]{bms2} and \cite[Section 13]{bs2}. \\
\indent Let $C$ be a complete and algebraically closed nonarchimedean valued field of mixed characteristic $(0,p)$, and let $\mathcal{O}_{C}$ be its ring of integers, which is a perfectoid valuation ring. By \cite[Th. 3.9]{amm} and \cite{susloc}, $\K_{\geq 0}(\mathcal{O}_{\bbC_{p}},\bbZ_{p})\simeq ku^{\wedge_{p}}$. Thus, for any $M\in\Mod_{\K_{\geq 0}(\mathcal{O}_{C},\bbZ_{p})}(\Sp^{\Cpl(p)})$, which in particular is a $ku^{\wedge_{p}}$-module, $L_{T(1)}M\simeq M[\beta^{-1}]$ by inverting a Bott element $\beta\in \pi_{2}\K(\mathcal{O}_{\bbC_{p}},\bbZ_{p})\simeq\pi_{2}\K(\mathcal{O}_{C},\bbZ_{p})\simeq\bbZ_{p}$ in $\Sp^{\Cpl(p)}$, cf. \cite[Constr. 2.8]{bcm}. \\
\indent Finally, let us fix some notations for objects pertinent to $C$. Take a compatible system of primitive $p$-power roots of unity in $\mathcal{O}_{C}$, and let $\epsilon\in\mathcal{O}_{C}^{\flat}$ be the element defined by the system. Then, we can take $d = \frac{[\epsilon]-1}{[\epsilon^{1/p}]-1}\in A_{\inf}(\mathcal{O}_{C}) = W(\mathcal{O}_{C}^{\flat})$ for a distinguished generator for the perfect prism $(A,(d)) = (A_{\inf}(\mathcal{O}_{C}),(d))$ corresponding to $\mathcal{O}_{C}$. Write $\CAlg^{\qrsp}_{\bbZ_{p},\mathcal{O}_{C}}$ for the full subcategory of $\CAlg^{p-\cpl}$ spanned by quasiregular semiperfectoid rings each of which admits an $\mathcal{O}_{C}$-algebra structure. In general, for any perfectoid ring $\overline{B}$, write $\CAlg^{\qrsp}_{\bbZ_{p},\overline{B}}$ (resp. $\CAlg^{\perfd}_{\bbZ_{p},\overline{B}}$) for the full subcategory of $\CAlg^{p-\cpl}$ spanned by objects of $\CAlg^{\qrsp}_{\overline{B}}$ (resp. $\CAlg^{\perfd}_{\overline{B}}$). 

\begin{construction}[Case of quasiregular semiperfectoid $\mathcal{O}_{C}$-algebras]\label{constr:qrsp}
Let $C$ be a complete and algebraically closed nonarchimedean valued field of mixed characteristic $(0,p)$, and let $\mathcal{O}_{C}$ be its valuation ring. For each $R\in\CAlg^{\qrsp}_{\bbZ_{p},\mathcal{O}_{C}}$, consider the filtered object 
\begin{align*}
\Fil^{\sbullet}L_{T(1)}\TC(R)&=\fib(L_{T(1)}\varphi^{hS^{1}}-\can:\tau_{\geq 2\sbullet}L_{T(1)}\TC^{-}(R)\to\tau_{\geq 2\sbullet}L_{T(1)}\TP(R))\\
                                            &\simeq \fib(F-1:\tau_{\geq 2\sbullet}L_{T(1)}\TP(R)\to\tau_{\geq 2\sbullet}L_{T(1)}\TP(R))
\end{align*} 
of $\Sp^{\Cpl(p)}$. Since $L_{T(1)}\TC^{-}(R)\simeq L_{T(1)}\TP(R)$ are even spectra by the preceding discussions, this construction gives a Postnikov filtration on $L_{T(1)}\TC(R)$. In particular, the filtered object $\Fil^{\sbullet}L_{T(1)}\TC(R)$ is complete and its underlying object is $L_{T(1)}\TC(R)$, i.e., $\colim_{n\to-\infty}\Fil^{n}L_{T(1)}\TC(R)\simeq L_{T(1)}\TC(R)$ in $\Sp^{\Cpl(p)}$ naturally on $R$. We have the following description of the associated graded pieces of this filtered object: 
\begin{lemma}
For each $n\in\bbZ$, there is a natural equivalence 
\begin{equation*}
\gr^{n}L_{T(1)}\TC(R)\simeq\left(\prism_{R/A}\{n\}[1/d]\right)^{F=1}[2n]
\end{equation*}
in $\Dh(\bbZ_{p})$ for $R\in\CAlg^{\qrsp}_{\bbZ_{p},\mathcal{O}_{C}}$.
\begin{proof}
Recall that the cyclotomic trace map sends a generator $\beta\in\pi_{2}\K(\mathcal{O}_{C},\bbZ_{p})$ to $([\epsilon]-1)$ (resp. $([\epsilon^{1/p}]-1)$) in $\pi_{\ast}\TP(\mathcal{O}_{C},\bbZ_{p})$ (resp. $\pi_{\ast}\TC^{-}(\mathcal{O}_{C},\bbZ_{p})$) up to graded units \cite[Th. 1.3.6]{hn}. From the description of $T(1)$-localizations for $\K_{\geq0}(\mathcal{O}_{C},\bbZ_{p})$-modules $\TC^{-}(R,\bbZ_{p})$ and $\TP(R,\bbZ_{p})$ above, we have natural isomorphisms
\begin{equation*}
\pi_{2n}L_{T(1)}\TC^{-}(R)\simeq \Fil^{n}_{\N}\prismc_{R}\{n\}[1/d]~~\text{and}~~\pi_{2n}L_{T(1)}\TP(R)\simeq\prismc_{R}\{n\}[1/d] 
\end{equation*} 
for each $n\in\bbZ$ as in \cite[Cor. 2.11]{bcm}. Through the description of maps $\varphi^{hS^{1}}$ and $\can$ for quasiregular semiperfectoid rings recalled above, their $T(1)$-localizations $L_{T(1)}\varphi^{hS^{1}} = \varphi^{hS^{1}}[\beta^{-1}]$ and $\can = \can[\beta^{-1}]:L_{T(1)}\TC^{-}(R)\to L_{T(1)}\TP(R)$ for $R\in\CAlg^{\qrsp}_{\mathbb{Z}_{p},\mathcal{O}_{C}}$ can be described as $\pi_{2n}L_{T(1)}\varphi^{hS^{1}} = \widehat{\varphi}\{n\}[\frac{1}{d}]$ and $\pi_{2n}(\can) = \can:\Fil^{n}_{\N}\prismc_{R}\{n\}[\frac{1}{d}]\to\prismc_{R}\{n\}[\frac{1}{d}]$. In particular, we have a natural identification 
\begin{equation*}
\begin{tikzcd}
\pi_{2n}F: \arrow[d, equal] &[-2.5em] \pi_{2n}L_{T(1)}\TP(R) \arrow[d,"\sim"'] & \pi_{2n}L_{T(1)}\TC^{-}(R) \arrow[r, "\pi_{2n}L_{T(1)}\varphi^{hS^{1}}"] \arrow[l,"\can"',"\sim"] \arrow[d,"\sim"] &[3em]  \pi_{2n}L_{T(1)}\TP(R) \arrow[d,"\sim"] \\ 
\widehat{F}: & \prismc_{R}\{n\}[1/d] & \Fil^{n}_{\N}(\prismc_{R}\{n\})[1/d] \arrow[r, "\widehat{\varphi}\{n\}{\left[\frac{1}{d}\right]}"'] \arrow[l,"\can","\sim"'] &\prismc_{R}\{n\}[1/d]
\end{tikzcd}
\end{equation*} 
of $\pi_{2n}F$ and $\widehat{F}$. Through this identification and Construction \ref{constr:frobfix} (1), we have 
\begin{equation*}
\gr^{n}L_{T(1)}\TC(R)\simeq \fib(\pi_{2n}(L_{T(1)}\TP(R))[2n]\xrightarrow{\pi_{2n}F-1}\pi_{2n}(L_{T(1)}\TP(R))[2n])\simeq \left(\prismc_{R}\{n\}[1/d]\right)^{F=1}[2n].
\end{equation*} 
By Proposition \ref{prop:nygdc}, we can decomplete the Nygaard completion inside the Frobenius fixed locus and identify the associated graded piece object with $\left(\prism_{R/A}\{n\}[\frac{1}{d}]\right)^{F=1}[2n]$. \end{proof}
\end{lemma}
\noindent Note that if we use Proposition \ref{prop:etcomp}, the $n$-th associated graded object can further be identified with $\RG_{\et}(\Spec R[\frac{1}{p}],\bbZ_{p}(n))[2n]$.
\end{construction}

For general $p$-complete commutative rings, we will construct filtrations on the $\arcp$-hypersheafification of $L_{T(1)}\TC$. In fact, $L_{T(1)}\TC$ turns out to be $\arcp$-hypercomplete on quasiregular semiperfectoid $\mathcal{O}_{C}$-algebras (Proposition \ref{prop:T(1)TCarcphypqrsp}), and the filtrations for general $p$-complete commutative rings will agree with the filtrations from Construction \ref{constr:qrsp} on $\CAlg^{\qrsp}_{\mathbb{Z}_{p},\mathcal{O}_{C}}$. First, let us check the case of perfectoid $\mathcal{O}_{C}$-algebras.   

\begin{lemma}\label{lem:T(1)TCarcphypperfd}
Let $C$ be as in Construction \ref{constr:qrsp}. Then, $L_{T(1)}\TC$ is an $\arcp$-hypersheaf on $\CAlg^{\perfd}_{\mathbb{Z}_{p},\mathcal{O}_{C}}$. 
\begin{proof}
First, observe that $R\mapsto \tau_{\geq 1}L_{T(1)}\TC(R)$ is an $\arcp$-hypersheaf on $\CAlg^{\perfd}_{\bbZ_{p},\mathcal{O}_{C}}$. In fact, each $R\mapsto\tau_{[2i-1,2i]}L_{T(1)}\TC(R)\simeq \left(\prism_{R/A}\{i\}[\frac{1}{d}]\right)^{F=1}[2i]$ for $i\geq 1$ is an $\arcp$-hypersheaf on $\CAlg^{\perfd}_{\bbZ_{p},\mathcal{O}_{C}}$ by Lemma \ref{lem:prismarcphypersheaf}. Thus, each presheaf $\tau_{[1,2n]}L_{T(1)}\TC(-)$ is a hypersheaf, and so is the limit $\tau_{\geq 1}L_{T(1)}\TC(-) = \lim_{n\to\infty}\tau_{[1,2n]}L_{T(1)}\TC(-)$. \\
\indent Combining this with the facts that ($0$-)truncated spectra are $T(1)$-acyclic and that $T(1)$-local spectra are closed under limits of spectra, the claimed result follows, cf. \cite[Proof of Prop. 5.10]{matTR}. In fact, let $R\to R^{\sbullet}$ be an $\arcp$-hypercover in $\CAlg^{\perfd}_{\mathbb{Z}_{p},\mathcal{O}_{C}}$, and consider the fiber sequence 
\begin{equation*}
A_{T(1)}(\tau_{\geq 1}L_{T(1)}\TC(R^{r}))\to \tau_{\geq 1}L_{T(1)}\TC(R^{r})\to L_{T(1)}(\tau_{\geq 1}L_{T(1)}\TC(R^{r}))\simeq L_{T(1)}\TC(R^{r})
\end{equation*}
for each $[r]\in\Delta_{s}$, cf. \cite[Not. A.5.1.1]{sag}. Here, the natural equivalence describing the third term of the sequence is obtained from the fiber sequence 
\begin{equation*}
L_{T(1)}(\tau_{\geq 1}L_{T(1)}\TC(R^{r}))\to L_{T(1)}(L_{T(1)}\TC(R^{r}))\to L_{T(1)}(\tau_{<1}L_{T(1)}\TC(R^{r}))\simeq 0,
\end{equation*}
whose second term is naturally equivalent to $L_{T(1)}\TC(R^{r})$ via the localization map $L_{T(1)}\TC(R^{r})\to L_{T(1)}(L_{T(1)}\TC(R^{r}))$. Thus, the second map of the first fiber sequence is equivalent to a 1-connective cover map, and the fiber term $A_{T(1)}(\tau_{\geq 1}L_{T(1)}\TC(R^{r}))$ is $0$-truncated. By taking limits on the first fiber sequence, we have a fiber sequence 
\begin{equation*}
\lim_{\Delta_{s}}A_{T(1)}(\tau_{\geq 1}L_{T(1)}\TC(R^{\sbullet}))\to \tau_{\geq 1}L_{T(1)}\TC(R)\to \lim_{\Delta_{s}}L_{T(1)}\TC(R^{\sbullet}), 
\end{equation*}
where we used that $\tau_{\geq1}L_{T(1)}\TC(-)$ is an $\arcp$-hypersheaf to describe the middle term of the sequence. Since the first term is $T(1)$-acyclic and the third term is $T(1)$-local, we know that the third term is a $T(1)$-localization of the middle term, i.e., $L_{T(1)}\TC(R)\simeq \lim_{\Delta_{s}}L_{T(1)}\TC(R^{\sbullet})$.  
\end{proof}
\end{lemma}

We need some preparations to extend this hyperdescent property on perfectoid $\mathcal{O}_{C}$-algebras to quasiregular semiperfectoid $\mathcal{O}_{C}$-algebras. 

\begin{lemma}\label{lem:perfdarccov}
Let $R\to S$ be a $p$-complete arc-cover in $\CAlg^{p-\cpl}$ such that $R$ and $S$ are quasiregular semiperfectoid rings. Then, $R_{\perfd}\to S_{\perfd}$ is a $p$-complete arc-cover in $\CAlg^{\perfd}$. 
\begin{proof}
Here, $(-)_{\perfd}$ denotes the perfectoidization of quasiregular semiperfectoid rings, cf. \cite[Cor. 7.3]{bs2} or \cite[Section 8]{bs2} for the general case. To check that the induced map $R_{\perfd}\to S_{\perfd}$ is a $p$-complete arc-cover, let $R_{\perfd}\to V$ be a map into a $p$-complete rank $\leq1$ valuation ring $V$ whose field of fractions is algebraically closed. Composing it with the natural map $R \to R_{\perfd}$ gives a map $R\to V$. As $R\to S$ is a $p$-complete $\arc$-cover, there is a diagram
\begin{equation*}
\begin{tikzcd}
R \arrow[r] \arrow[rr, bend left=25] \arrow[d] & R_{\perfd} \arrow[r] \arrow[d] & V \arrow[d, hook]\\
S \arrow[r] \arrow[rr, bend right=25] & S_{\perfd} \arrow[r, dashrightarrow] & W,
\end{tikzcd}
\end{equation*}
with commutative outer square, where $V\hookrightarrow W$ is an extension of valuation rings and $W$ satisfies the same assumptions as $V$, cf. \cite[2.2.1]{cs}. Since both $V$ and $W$ are perfectoid, applying perfectoidization $(-)_{\perfd}$ to the square gives the desired right side commutative square in the diagram. 
\end{proof}
\end{lemma}  

\begin{lemma}\label{lem:spgrperfd}
Let $S\in\CAlg^{p-\cpl}$ be a semiperfectoid ring, and let $(A,(d))$ be a perfect prism whose corresponding perfectoid ring $\overline{A}$ maps to $S$. Then, for each $n\in\bbZ$, there is a natural equivalence $\left(\prism_{S/A}\{n\}[\frac{1}{d}]\right)^{F=1}\simeq \left(\prism_{S_{\perfd}/A}\{n\}[\frac{1}{d}]\right)^{F=1}$. 
\begin{proof}
Since $(-[\frac{1}{d}])^{F=1}$ is constant on perfection \cite[Lem. 9.2]{bs2}, it suffices to check that the natural map $(\prism_{S/A})_{\perf}\to(\prism_{S_{\perfd}/A})_{\perf}$ is an equivalence in $\Dh(A)$. Since $S_{\perfd}$ is perfectoid \cite[Cor. 7.3 and Prop. 8.5]{bs2}, $(\prism_{S_{\perfd}/A})_{\perf}\simeq (W(S_{\perfd}^{\flat}))_{\perf} = W(S_{\perfd}^{\flat})$ and modulo $d$ the map induces an equivalence $S_{\perfd}\simeq W(S_{\perfd}^{\flat})/d$ in $\Dh(\bbZ_{p})$. Thus, the original map is an equivalence.  
\end{proof}
\end{lemma}

\begin{proposition}\label{prop:T(1)TCinvperfdqrsp}
Let $C$ be as in Construction \ref{constr:qrsp}. For $S\in\CAlg^{\qrsp}_{\mathbb{Z}_{p},\mathcal{O}_{C}}$, the natural map $S\to S_{\perfd}$ induces an equivalence $L_{T(1)}\TC(S)\simeq L_{T(1)}\TC(S_{\perfd})$. 
\begin{proof}
Consider the functor $\Fil^{\sbullet}:=S\mapsto \Fil^{\sbullet}L_{T(1)}\TC(S)$ on $\CAlg^{\qrsp}_{\mathbb{Z}_{p},\mathcal{O}_{C}}$ provided by Construction \ref{constr:qrsp}, which assigns a Postnikov filtration on each $L_{T(1)}\TC(S)$. Due to the exhaustiveness of the filtration, it suffices to check that the functor induces an equivalence when applied to the perfectoidization map $S\to S_{\perfd}$. By Lemma \ref{lem:spgrperfd}, each associated graded pieces functor induces equivalences on perfectoidization maps. Thus, by induction, $\Fil^{n}/\Fil^{n+m}$ has the same property for each $n\in\bbZ$ and $m\geq 1$. Since $\Fil^{n}\simeq \lim_{m\to\infty}\Fil^{n}/\Fil^{n+m}$ by completeness of the filtration, $\Fil^{n}$ induces equivalences on perfectoidization maps. 
\end{proof}
\end{proposition}

\begin{proposition}\label{prop:T(1)TCarcphypqrsp}
Let $C$ be as in Construction \ref{constr:qrsp}. Then, $L_{T(1)}\TC$ is an $\arcp$-hypersheaf on $\CAlg^{\qrsp}_{\mathbb{Z}_{p},\mathcal{O}_{C}}$. 
\begin{proof}
We first check that $L_{T(1)}\TC$ is an $\arcp$-sheaf on $\CAlg^{\qrsp}_{\mathbb{Z}_{p},\mathcal{O}_{C}}$. It suffices to check that it is a $p$-complete arc sheaf and vanishes on $\bbF_{p}$-algebras, and the latter condition follows from the vanishing $L_{T(1)}\K(\bbF_{p})\simeq0$. To check the first condition, take a $p$-complete arc-cover $R\to S$ in $\CAlg^{\qrsp}_{\mathbb{Z}_{p},\mathcal{O}_{C}}$. By Proposition \ref{prop:T(1)TCinvperfdqrsp} and Lemma \ref{lem:perfdarccov}, the question of descent with respect to the Cech nerve of $R\to S$ reduces to that of a $p$-complete arc cover $R_{\perfd}\to S_{\perfd}$ in $\CAlg^{\perfd}_{\mathcal{O}_{C}}$, cf. \cite[Prop. 8.13]{bs2}. The latter case is guaranteed by Lemma \ref{lem:T(1)TCarcphypperfd}. \\
\indent To verify that $L_{T(1)}\TC$ is an $\arcp$-hypersheaf on $\CAlg^{\qrsp}_{\mathbb{Z}_{p},\mathcal{O}_{C}}$, first note that for each $i\geq1$, the presheaf $R\mapsto\tau_{[2i-1, 2i]}L_{T(1)}\TC(R)\simeq \left(\prism_{R/A}\{i\}[\frac{1}{d}]\right)^{F=1}[2i]$ is an $\arcp$-hypersheaf by Lemma \ref{lem:prismarcphypersheaf}. By induction, each $R\mapsto\tau_{[1,2n]}L_{T(1)}\TC(R)$ is an $\arcp$-hypersheaf, and hence the presheaf limit 
\begin{equation*}
\lim_{n\to\infty}\tau_{[1,2n]}L_{T(1)}\TC(-)\simeq \tau_{\geq 1}L_{T(1)}\TC(-)
\end{equation*} 
is an $\arcp$-hypersheaf on $\CAlg^{\qrsp}_{\mathbb{Z}_{p},\mathcal{O}_{C}}$. Now, in the fiber sequence of presheaves 
\begin{equation*}
\tau_{\geq 1}L_{T(1)}\TC(-)\to L_{T(1)}\TC(-)\to \tau_{\leq 0}L_{T(1)}\TC(-),
\end{equation*} 
the third object is also an $\arcp$-sheaf, and due to its truncatedness, it is an $\arcp$-hypersheaf. Thus, the middle object $L_{T(1)}\TC$ is an $\arcp$-hypersheaf on $\CAlg^{\qrsp}_{\mathbb{Z}_{p},\mathcal{O}_{C}}$.  
\end{proof}
\end{proposition}

Now, we can extend our construction of the filtration to the case of general animated commutative rings.  

\begin{construction}[Filtration on the $\arcp$-hypersheafification]
\label{constr:filtonarcpsheaf}
Let $C = \bbC_{p} = \widehat{\overline{\mathbb{Q}_{p}}}$, and let $(\widetilde{A},(\widetilde{d}))$ be a perfect prism corresponding to $\mathcal{O}_{C}$. \\
(1) The natural map $\Fil^{\sbullet}L_{T(1)}\TC\to L_{T(1)}\TC$ lifts to a map of presheaves of multiplicative filtered $p$-complete spectra. Construction \ref{constr:qrsp} gives a presheaf $\Fil^{\sbullet}L_{T(1)}\TC = \left(R\mapsto \Fil^{\sbullet}L_{T(1)}\TC(R)\right)$ of filtration-complete filtered $p$-complete spectra and a map of presheaves $\Fil^{\sbullet}L_{T(1)}\TC\to L_{T(1)}\TC$ of filtered $p$-complete spectra on $\CAlg^{\qrsp}_{\mathbb{Z}_{p},\mathcal{O}_{C}}$. Here, the target of the latter map $L_{T(1)}\TC$ is understood to be equipped with the constant filtration. On the other hand, note that $T(1)$-local $\TC$ defines an object $L_{T(1)}\TC\in\Fun(\CAlg^{\an},\CAlg(\Sp^{\Cpl(p)}))$ which is an equalizer of the endomorphisms $F$ and $id$ of $L_{T(1)}\TP$ viewed as an object of $\Fun(\CAlg^{\an},\CAlg(\Sp^{\Cpl(p)}))$. By Proposition \ref{prop:postnikovmult} and Remark \ref{rem:postnikovmult}, we know that $\tau_{\geq 2\sbullet}F:\tau_{\geq 2\sbullet}L_{T(1)}\TP\to\tau_{\geq 2\sbullet}L_{T(1)}\TP$ and the natural map $\tau_{\geq 2\sbullet}L_{T(1)}\TP\to L_{T(1)}\TP$ into the constant filtered object are maps in $\Fun(\CAlg^{\qrsp}_{\mathbb{Z}_{p},\mathcal{O}_{C}}, \CAlg(\DFh(\bbS)))$. Now, we can view the functor $\Fil^{\sbullet}L_{T(1)}\TC:\CAlg^{\qrsp}_{\bbZ_{p},\mathcal{O}_{C}}\to \DFh(\bbS)$ as an underlying presheaf of a presheaf valued in multiplicative filtered $p$-complete spectra, for which we retain the same notation by rewriting it as an equalizer $\Fil^{\sbullet}L_{T(1)}\TC:=\mathrm{eq}(\tau_{\geq 2\sbullet}F,id:\tau_{\geq 2\sbullet}L_{T(1)}\TP\to\tau_{\geq 2\sbullet}L_{T(1)}\TP)$ in $\Fun(\CAlg^{\qrsp}_{\mathbb{Z}_{p},\mathcal{O}_{C}}, \CAlg(\DFh(\bbS)))$. In particular, the natural map $\Fil^{\sbullet}L_{T(1)}\TC\to L_{T(1)}\TC$ is in fact a map of $\CAlg(\DFh(\bbS))$-valued presheaves. \\
(2) The presheaf map $\Fil^{\sbullet}L_{T(1)}\TC\to L_{T(1)}\TC$ is in fact a map of $\arcp$-hypersheaves. Since the associated graded object presheaves $\gr^{n}L_{T(1)}\TC\simeq (R\mapsto \left(\prism_{R/\widetilde{A}}\{n\}[1/\widetilde{d}]\right)^{F=1}[2n])$ of the presheaf $\Fil^{\sbullet}L_{T(1)}\TC$ are $\arcp$-hypersheaves on $\CAlg^{\qrsp}_{\bbZ_{p},\mathcal{O}_{C}}$ for all $n\in\bbZ$ by Lemma \ref{lem:prismarcphypersheaf}, we know from the filtration-completeness that $\Fil^{\sbullet}L_{T(1)}\TC$ is an $\arcp$-hypersheaf. By Proposition \ref{prop:T(1)TCarcphypqrsp}, the map $\Fil^{\sbullet}L_{T(1)}\TC\to L_{T(1)}\TC$ is a morphism of $\arcp$-hypersheaves.  \\
(3) Write $(L_{T(1)}\TC)^{\sharp}$ for an $\arcp$-hypersheafification of $L_{T(1)}\TC$ on $\CAlg^{p-\cpl}$. We define a multiplicative, exhaustive and complete filtration $\Fil^{\sbullet}(L_{T(1)}\TC)^{\sharp}$ on $(L_{T(1)}\TC)^{\sharp}$ via right Kan extension as follows. Since $\CAlg^{\qrsp}_{\mathbb{Z}_{p},\mathcal{O}_{C}}$ form a basis of $\CAlg^{p-\cpl}$ equipped with the $\arcp$-topology, the map 
\begin{equation*}
\Fil^{\sbullet}L_{T(1)}\TC\to L_{T(1)}\TC
\end{equation*} 
in $\Shv^{\hyp}_{\CAlg(\DFh(\bbS))}\left((\CAlg^{\qrsp}_{\mathbb{Z}_{p},\mathcal{O}_{C}})^{\op}_{\arcp}\right)$ right Kan extends to the corresponding map and the objects 
\begin{equation*}
\Fil^{\sbullet}(L_{T(1)}\TC)^{\sharp}\to (L_{T(1)}\TC)^{\sharp}
\end{equation*} 
in $\Shv^{\hyp}_{\CAlg(\DFh(\bbS))}\left((\CAlg^{p-\cpl})^{\op}_{\arcp}\right)$ via the equivalence of categories $\Shv^{\hyp}_{\CAlg(\DFh(\bbS))}\left((\CAlg^{p-\cpl})^{\op}_{\arcp}\right)\simeq \Shv^{\hyp}_{\CAlg(\DFh(\bbS))}\left((\CAlg^{\qrsp}_{\mathbb{Z}_{p},\mathcal{O}_{C}})^{\op}_{\arcp}\right)$ of hypersheaves valued in multiplicative filtered $p$-complete spectra given by restrictions and right Kan extensions. Here, $\Fil^{\sbullet}(L_{T(1)}\TC)^{\sharp}$ on $\CAlg^{p-\cpl}$ is defined to be a right Kan extension of $\Fil^{\sbullet}L_{T(1)}\TC$ on $\CAlg^{\qrsp}_{\mathbb{Z}_{p},\mathcal{O}_{C}}$. By construction, $\Fil^{\sbullet}(L_{T(1)}\TC)^{\sharp}$ is an $\arcp$-Postnikov filtration $\tau^{\sharp}_{\geq2\sbullet-1}(L_{T(1)}\TC)^{\sharp}$ of the $\arcp$-hypersheaf $(L_{T(1)}\TC)^{\sharp}$, as its restriction on $\CAlg^{\qrsp}_{\mathbb{Z}_{p},\mathcal{O}_{C}}$ is equivalent to an $\arcp$-Postnikov filtration of $L_{T(1)}\TC$ on $\CAlg^{\qrsp}_{\bbZ_{p},\mathcal{O}_{C}}$, which is just the presheaf Postnikov filtration $\tau_{\geq 2\sbullet -1}L_{T(1)}\TC$ on $\CAlg^{\qrsp}_{\bbZ_{p},\mathcal{O}_{C}}$ as discussed above. In particular, the filtered object $\Fil^{\sbullet}(L_{T(1)}\TC)^{\sharp}(R)$ is an exhaustive and complete \cite[A.10]{matTR} filtration on $(L_{T(1)}\TC)^{\sharp}(R)$ for each $R\in\CAlg^{p-\cpl}$. \\
(4) By definition of $\Fil^{\sbullet}(L_{T(1)}\TC)^{\sharp}$ in (3) and the description of $n$-th associated graded pieces in Construction \ref{constr:qrsp}, there is a natural equivalence 
\begin{equation*}
\gr^{n}(L_{T(1)}\TC)^{\sharp}(R)\simeq \left(\prism_{R/A}\{n\}[1/d]\right)^{F=1}[2n]
\end{equation*}
for each $n\in\bbZ$ and $R\in\CAlg^{p-\cpl}_{\overline{A}}$ for any perfect prism $(A,(d))$, as the functor $R\mapsto \left(\prism_{R/A}\{n\}[1/d]\right)^{F=1}:\CAlg^{p-\cpl}_{\overline{A}}\to \Dh(\bbZ_{p})$ on the right hand side is an $\arcp$-hypersheaf by Lemma \ref{lem:prismarcphypersheaf}. Note that assuming Proposition \ref{prop:etcomp}, we can also say that there is a natural equivalence
\begin{equation*}
\gr^{n}(L_{T(1)}\TC)^{\sharp}(R)\simeq \RG_{\et}(\Spec R[1/p],\bbZ_{p}(n))[2n]
\end{equation*}
for each $n\in\bbZ$ and $R\in\CAlg^{p-\cpl}$. \\
(5) Finally, we naturally define the functor $\Fil^{\sbullet}(L_{T(1)}\TC)^{\sharp}:\CAlg^{\an}\to\CAlg(\DFh(\bbS))$ as 
\begin{equation*}
\Fil^{\sbullet}(L_{T(1)}\TC)^{\sharp}(R) = \Fil^{\sbullet}(L_{T(1)}\TC)^{\sharp}(\pi_{0}(R^{\wedge_{p}}))
\end{equation*} 
for $R\in\CAlg^{\an}$. By definition, it is equipped with a map $\Fil^{\sbullet}(L_{T(1)}\TC)^{\sharp}\to L_{T(1)}\TC(\pi_{0}((-)^{\wedge_{p}}))$ of presheaves of multiplicative filtered $p$-complete spectra. Note that the above two descriptions of the associated graded objects functor remain the same. Also, viewed as a functor into $\DFh(\bbS)$, the functor factors through $\DFhc(\bbS)$ as its values are filtration-complete, and induces $\Fil^{\sbullet}(L_{T(1)}\TC)^{\sharp}:\CAlg^{\an}\to\DFhc(\bbS)$. 
\end{construction}

\begin{remarkn}
(1) As the natural map $S\to S_{\perfd}$ is an $\arcp$-equivalence for any semiperfectoid ring $S$, one has $(L_{T(1)}\TC)^{\sharp}(S)\simeq (L_{T(1)}\TC)^{\sharp}(S_{\perfd})$ and similarly $\Fil^{\sbullet}(L_{T(1)}\TC)^{\sharp}(S)\simeq \Fil^{\sbullet}(L_{T(1)}\TC)^{\sharp}(S_{\perfd})$ for the filtration.  \\
(2) As $\CAlg^{\QSyn}_{\mathcal{O}_{C}/}$ admits $\CAlg^{\qrsp}_{\bbZ_{p},\mathcal{O}_{C}}$ as a basis for the $\arcp$-topology, one knows that over $\CAlg^{\QSyn}_{\mathcal{O}_{C}/}$, $\arcp$-sheafification of $L_{T(1)}\TC$ agrees with the $\arcp$-hypersheafification $(L_{T(1)}\TC)^{\sharp}$. In fact, note that restrictions and right Kan extensions induce an equivalence $\Shv_{\mathcal{D}}\left((\CAlg^{\QSyn}_{\mathcal{O}_{C}/})^{\op}_{\arcp}\right)\simeq\Shv_{\mathcal{D}}\left((\CAlg^{\qrsp}_{\bbZ_{p},\mathcal{O}_{C}})^{\op}_{\arcp}\right)$ extending the equivalence between their full subcategories of hypersheaves for any presentable $\infty$-category $\mathcal{D}$, since for each $R\in\CAlg^{\QSyn}_{\mathcal{O}_{C}/}$ we can find a quasisyntomic cover $R\to S$, which in particular is an $\arcp$-cover, into $S\in\CAlg^{\qrsp}_{\bbZ_{p},\mathcal{O}_{C}}$ such that all the terms of the associated Cech nerve remain to be in $\CAlg^{\qrsp}_{\bbZ_{p},\mathcal{O}_{C}}$, cf. \cite[Lem. C.3]{hoy} or \cite[Proof of Prop. 4.31]{bms2}.
\end{remarkn}

\begin{remarkn}\label{rem:mathewfilt1}
In \cite{matTR}, a natural filtration on an $\arcp$-hypersheafification of $T(1)$-local $\TR$ was constructed through right Kan extension from the case of perfectoid rings. In \emph{ibid}., it was also proved that for the case of formally smooth $\mathcal{O}_{K}$-algebras, a natural filtration on $T(1)$-local $\TC$ can be obtained by taking Frobenius fixed point on the former filtration. Our approach to Construction \ref{constr:filtonarcpsheaf} provides an analogue of this idea for $\TP$ and quasiregular semiperfectoid rings in place of $\TR$ and perfectoid rings, using the fact that at least after $T(1)$-localization there is a Frobenius endomorphism on $L_{T(1)}\TP$ whose fixed point recovers $L_{T(1)}\TC$ for commutative rings, as explained in Construction \ref{constr:T(1)TP}. 
\end{remarkn}

\subsection{Pro-Galois descent for $T(1)$-local $\TC$}\label{subsec:progal}

Construction \ref{constr:filtonarcpsheaf} naturally leads us to ask that for which rings the values of $T(1)$-local $\TC$ are unaffected by $\arcp$-hypersheafification. When studying this question in \ref{subsec:arcphyp}, we will use the following pro-Galois descent result for $T(1)$-local $\TC$ at the generic fiber, a variant of \cite[Proof of Th. 6.8]{matTR}:

\begin{proposition}[Pro-Galois descent at the generic fiber]\label{prop:progal}
Let $B$ be a commutative ring, and let $B_{\infty}$ be a commutative $B$-algebra equipped with a continuous action of a profinite group $G$ with respect to the discrete topology on $B_{\infty}$. Let $R$ be a commutative $B$-algebra, and let $R_{\infty}$ be the $R$-algebra $B_{\infty}\otimes_{B}R$ equipped with the induced action of $G$. \\
\indent Assume the following conditions:
\begin{adjustwidth}{12pt}{}
(1) The map $B[1/p]\to B_{\infty}[1/p]$ is a $G$-Galois extension. \\
(2) $B[1/p]$ has finite Krull dimension and has globally bounded virtual $p$-cohomological dimensions of residue fields.\\
(3) $G$ has finite virtual $p$-cohomological dimension and admits a cofinal set $(N_{i}\unlhd G)_{I}$ of open normal subgroups of $G$ indexed by a filtered set $I$, such that for each $i\in I$, the ring map $R\to R_{\infty}^{N_{i}}$ is finite and finitely presented. 
\end{adjustwidth}
Then, $L_{T(1)}\TC$ satisfies descent with respect to the Cech nerve of the map $R\to R_{\infty}$, i.e., we have a natural equivalence of $p$-complete spectra
\begin{equation}\label{eq:T(1)TCprogal}
L_{T(1)}\TC(R)\simeq \lim(L_{T(1)}\TC(R_{\infty})\rightrightarrows L_{T(1)}\TC(R_{\infty}\otimes_{R}R_{\infty})~\substack{\rightarrow\\[-1em] \rightarrow \\[-1em] \rightarrow}\cdots).
\end{equation}
\end{proposition}

The remainder of this subsection is devoted to the proof of Proposition \ref{prop:progal}; we retain the notations and the assumptions of the proposition throughout this subsection.\\
\indent Let us start by noting that $T(1)$-local $\TC$, as a functor valued in $p$-complete spectra, preserves sifted colimits of connective ring spectra:

\begin{lemma}\label{lem:T(1)TClke}
The functor $R\mapsto L_{T(1)}\TC(R):\Alg_{\bbE_{1}}(\Sp_{\geq 0})\to \Sp^{\Cpl(p)}$ commutes with sifted colimits. 
\begin{proof}
Let $(R_{i})_{I}$ be a sifted diagram of connective $\bbE_{1}$-rings. By \cite[Cor. 2.15]{cmm}, the natural map $\colim_{I}\TC(R_{i})\to\TC(\colim_{I}R_{i})$ is a $p$-adic equivalence of spectra. Since $T(1)$-localization induces equivalences on $p$-adic equivalences of spectra, we have a natural equivalence 
\begin{equation*}
L_{T(1)}\left(\colim_{I}\TC(R_{i})^{\wedge_{p}}\right)\simeq L_{T(1)}\TC(\colim_{I}R_{i})
\end{equation*} 
of $p$-complete spectra. To further describe the source object, note that there is a natural equivalence of spectra $L_{T(1)}X\simeq (L_{1}X)^{\wedge_{p}}$ for $X\in\Sp$, cf. \cite[Prop. 2.11]{bous}. Since the localization $L_{1}$ is smashing by Hopkins-Ravenel theorem \cite[Th. 7.5.6]{rav}, we know the natural map 
\begin{equation*}
\colim'_{I}L_{T(1)}\TC(R_{i})\to L_{T(1)}\left(\colim_{I}\TC(R_{i})^{\wedge_{p}}\right)
\end{equation*} 
(here, $\colim'_{I}$ in the source of the map is a temporary notation for a colimit taken in $\Sp^{\Cpl(p)}$) is an equivalence of $p$-complete spectra. 
\end{proof}
\end{lemma}

Now, let us fix some standard notations and conventions. By $G$-Galois extensions, we mean pro-$G$-Galois extensions which in particular are ind-finite Galois, cf. \cite[Def. 8.1.1]{roggal}. We write $\mathcal{T}_{G}$ for the finitary site of finite continuous $G$-sets and jointly surjective maps as in \cite[Def. 4.1]{etalek}. Note that there is a natural map of sites $\rho:\mathcal{T}_{G}\to (\Spec B[1/p])_{\et}$ whose underlying functor maps a finite continuous $G$-set $T$ to the spectrum of the $B[1/p]$-algebra $B_{\infty}[1/p]^{\ker(G\to\Aut(T))}$. Let $\rho_{\ast}:\Shv_{\Sp^{\Cpl(p)}}((\Spec B[1/p])_{\et})\to\Shv_{\Sp^{\Cpl(p)}}(\mathcal{T}_{G})$ be the right adjoint functor between $\infty$-categories of $\Sp^{\Cpl(p)}$-valued sheaves induced from $\rho$. Since its left adjoint $\rho^{\ast}$ preserves $\infty$-connective objects, $\rho_{\ast}$ induces a functor $\rho_{\ast}:\Shv^{\hyp}_{\Sp^{\Cpl(p)}}((\Spec B[1/p])_{\et})\to\Shv^{\hyp}_{\Sp^{\Cpl(p)}}(\mathcal{T}_{G})$ between the full subcategories of hypercomplete objects. \\
\indent Consider the functor 
\begin{equation*}
F = \left(T\mapsto L_{T(1)}\TC(\Fun_{G}(T,R_{\infty}))\right):\mathcal{T}_{G}^{\op}\to\Sp^{\Cpl(p)}.
\end{equation*}
By definition, $F$ preserves finite products, i.e., $F\in\Fun^{\pi}(\mathcal{T}_{G}^{\op},\Sp^{\Cpl(p)})$. Through the natural equivalence $\Fun^{\pi}(\mathcal{T}_{G}^{\op},\Sp^{\Cpl(p)})\simeq\Fun(\mathcal{O}_{G}^{\op},\Sp^{\Cpl(p)})$ \cite[Constr. 4.5]{etalek}, $F$ is completely determined by its restriction to the orbit category $\mathcal{O}_{G}$ consisting of finite continuous nonempty transitive $G$-sets via right Kan extension. 

\begin{lemma}
The presheaf $F = L_{T(1)}\TC(\Fun_{G}(-,R_{\infty}))$ of $p$-complete spectra on $\mathcal{T}_{G}$ is a hypercomplete sheaf. 
\begin{proof}
Consider the presheaf 
\begin{equation*}
A = \left(\Spec S\mapsto L_{T(1)}\K(S\otimes_{B[1/p]}R[1/p])\right)
\end{equation*}
of $p$-complete spectra on $(\Spec B[1/p])_{\et}$ induced from the localizing invariant $L_{T(1)}\K(-\otimes_{HB[1/p]}HR[1/p]):\Cat_{HB[1/p]}^{\ex}\to\Sp^{\Cpl(p)}$ of $B[1/p]$-linear small stable $\infty$-categories. Due to the finiteness conditions on $B[1/p]$, the presheaf $A$ is a hypercomplete \'etale sheaf \cite[Th. 7.14]{etalek}. In particular, the image $\rho_{\ast}A$ of the sheaf $A$ is a hypercomplete object of $\Shv_{\Sp^{\Cpl(p)}}(\mathcal{T}_{G})$. As an object of $\Fun(\mathcal{O}_{G}^{\op},\Sp^{\Cpl(p)})$, the functor $\rho_{\ast}A$ can be described as
\begin{equation*}
G/H\mapsto L_{T(1)}\K(B_{\infty}[1/p]^{H}\otimes_{B[1/p]}R[1/p]).
\end{equation*}
Note that $F$ is a sheaf of $p$-complete spectra on $\mathcal{T}_{G}$ by \cite[Constr. 6.3]{matTR}. Moreover, observe that $F$ is a module over the hypersheaf $\rho_{\ast}A$; for each open subgroup $H\leq G$, we have maps of $\bbE_{\infty}$-rings 
\begin{align*}
L_{T(1)}\K(B_{\infty}[1/p]^{H}\otimes_{B[1/p]}R[1/p]) &\simeq L_{T(1)}\K(B_{\infty}^{H}\otimes_{B}R)\to L_{T(1)}\TC(B_{\infty}^{H}\otimes_{B}R)\\
& \to L_{T(1)}\TC((B_{\infty}\otimes_{B}R)^{H})\simeq L_{T(1)}\TC(\Fun_{G}(G/H,R_{\infty})).  
\end{align*} 
Now, as $G$ has finite virtual $p$-cohomological dimension, hypercompletion is smashing \cite[Prop. 4.17]{etalek}, and hence the sheaf $F$ is also hypercomplete.
\end{proof}
\end{lemma}

\begin{remarkn}
By \cite[Constr. 4.6 and Th. 4.26]{etalek}, there is a natural equivalence 
\begin{equation}\label{eq:postshpt}
F(\ast)\simeq \lim(\colim_{I}F(G/N_{i})\rightrightarrows \colim_{I}F(G/N_{i}\times G/N_{i}) ~\substack{\rightarrow\\[-1em] \rightarrow \\[-1em] \rightarrow}\cdots)
\end{equation}
computing the value of the hypersheaf $F$ at the point $\ast$. 
\end{remarkn}

Let us compute each of the terms in the limit on the right hand side of the equivalence (\ref{eq:postshpt}). The first term is 
\begin{equation}\label{eq:firstterm}
\colim_{I}F(G/N_{i})\simeq\colim_{I}L_{T(1)}\TC(R_{\infty}^{N_{i}})\simeq L_{T(1)}\TC(R_{\infty}).
\end{equation} 
Note that we used Lemma \ref{lem:T(1)TClke}, which in particular asserts that $L_{T(1)}\TC$ commutes with filtered colimits of rings. \\
\indent Let $k\geq 1$. The higher terms in the limit are 
\begin{align*}
\colim_{I}F(G/N_{i}\times (G/N_{i})^{\times k}) &= \colim_{I}L_{T(1)}\TC(\Fun_{G}(G/N_{i}\times (G/N_{i})^{\times k},R_{\infty}))\\
& \simeq \colim_{I}L_{T(1)}\TC(\Fun((G/N_{i})^{\times k},R_{\infty}))\\
& \simeq L_{T(1)}\TC(\Fun_{\mathrm{cts}}(G^{\times k},R_{\infty})). 
\end{align*} 
Here, we used that $\Fun_{G}(G/N_{i}\times (G/N_{i})^{\times k},R_{\infty})\simeq \Fun((G/N_{i})^{\times k}, R_{\infty})$; the isomorphism maps each $G$-equivariant map $\varphi:G/N_{i}\times (G/N_{i})^{\times k}\to R_{\infty}$ to the map $(\overline{h_{1}},...,\overline{h_{k}})\mapsto \varphi(\overline{1},\overline{h_{1}},...,\overline{h_{k}})$, while the inverse isomorphism maps each map $\psi:(G/N_{i})^{\times k}\to R_{\infty}$ to the $G$-equivariant map $(\overline{g},\overline{h_{1}},...,\overline{h_{k}})\mapsto g\psi(\overline{g^{-1}h_{1}},...,\overline{g^{-1}h_{k}})$. The final equivalence follows from $\Fun_{\mathrm{cts}}(G^{\times k},R_{\infty}) = \colim_{I}\Fun((G/N_{i})^{\times k},R_{\infty})$ and Lemma \ref{lem:T(1)TClke}.\\
\indent We further compute that each of the higher terms above is equivalent to $L_{T(1)}\TC(R_{\infty}^{\otimes_{R}(k+1)})$:

\begin{lemma}\label{lem:indgaloiscech}
There is a canonical equivalence of $p$-complete spectra
\begin{equation*}
L_{T(1)}\TC(R_{\infty}^{\otimes_{R}(k+1)})\simeq L_{T(1)}\TC(\Fun_{\mathrm{cts}}(G^{\times k},R_{\infty})).
\end{equation*}
\begin{proof}
For notational convenience, let us write $R_{i} = R_{\infty}^{N_{i}}$ for each $i\in I$. Consider the map $R_{i}^{\otimes_{R}(k+1)}\to \Fun((G/N_{i})^{\times k},R_{\infty})$ given by the composition 
\begin{equation*}
R_{i}^{\otimes_{R}(k+1)}\to \Fun((G/N_{i})^{\times k}, R_{i})\to \Fun((G/N_{i})^{\times k}, R_{\infty}).
\end{equation*}
Each of the maps in the composition is natural in $i\in I$, and hence after taking filtered colimits, we have an induced composite map 
 \begin{align*}
 R_{\infty}^{\otimes_{R}(k+1)} \to \colim_{i}\Fun((G/N_{i})^{\times k}, R_{i})\to\colim_{i}\Fun((G/N_{i})^{\times k},R_{\infty}) = \Fun_{\mathrm{cts}}(G^{\times k}, R_{\infty}).
 \end{align*} 
We verify that upon taking the functor $L_{T(1)}\TC$, this map gives an equivalence of $p$-complete spectra.\\
\indent First, observe that $T(1)$-local $\TC$ of the first natural map $R_{\infty}^{\otimes_{R}(k+1)} \to \colim_{i}\Fun((G/N_{i})^{\times k}, R_{i})$ of the composite is an equivalence. By Lemma \ref{lem:T(1)TClke}, it suffices to check that for each $i\in I$, the induced map $L_{T(1)}\TC(R_{i}^{\otimes_{R}(k+1)})\to L_{T(1)}\TC(\Fun((G/N_{i})^{\times k}, R_{i}))$ is an equivalence. Note that the $R$-algebra map $R_{i}^{\otimes_{R}(k+1)}\to\Fun((G/N_{i})^{\times k},R_{i})$ has the source and target that are finitely presented as $R$-modules \cite[Tag 0564]{stacks} and induces an isomorphism after inverting $p$ as the map $R[1/p]\to R_{i}[1/p]$ is $G/N_{i}$-Galois. Thus, the map $R_{i}^{\otimes_{R}(k+1)}\to\Fun((G/N_{i})^{\times k},R_{i})$ is a $p$-isogeny, and in particular induces an equivalence after applying $L_{T(1)}\TC$ by \cite[Prop. 2.2 (3)]{matTR}. \\
\indent It remains to check that $T(1)$-local $\TC$ of the second natural map of the composite is an equivalence; in fact, we check that the map $\colim_{i}\Fun((G/N_{i})^{\times k}, R_{i})\to\colim_{i}\Fun((G/N_{i})^{\times k},R_{\infty})$ is an isomorphism. For the surjectivity of the map, observe that for each element $(r_{g})_{g\in (G/N_{i})^{\times k}}$ of the target, due to the finiteness of $(G/N_{i})^{\times k}$, one can take $j\geq i$ so $r_{g}\in R_{j}$ for all $g\in (G/N_{i})^{\times k}$, and hence $(r_{g})_{g\in (G/N_{i})^{\times k}} = (r_{\overline{h}\in (G/N_{i})^{\times k}})_{h\in (G/N_{j})^{\times k}}$ is in the image of the map $\Fun((G/N_{j})^{\times k},R_{j})\to \Fun((G/N_{j})^{\times k},R_{\infty})$. For the injectivity of the map, note that the maps $\Fun((G/N_{i})^{\times k},R_{i})\to \Fun((G/N_{i})^{\times k}, R_{\infty})$ induced from the injective ring maps $R_{i}\to R_{\infty}$ are already injective for all $i\in I$, and hence the filtered colimit remains to be injective. 
\end{proof}
\end{lemma}

By Lemma \ref{lem:indgaloiscech} and the preceding discussion, we have a canonical equivalence
\begin{equation}\label{eq:higherterms}
\colim_{i}F(G/N_{i}\times (G/N_{i})^{\times k})\simeq L_{T(1)}\TC(\Fun_{\mathrm{cts}}(G^{\times k}, R_{\infty}))\simeq L_{T(1)}\TC(R_{\infty}^{\otimes_{R}(k+1)})
\end{equation}
for each $k\geq 0$; equivalence (\ref{eq:firstterm}) covers the case of $k=0$. Combining the equivalence (\ref{eq:higherterms}), $F(\ast)\simeq L_{T(1)}\TC(R)$, and the equivalence (\ref{eq:postshpt}) computing $F(\ast)$, we have the claimed natural equivalence (\ref{eq:T(1)TCprogal}). This finishes the proof of Proposition \ref{prop:progal}. 

\subsection{Values of the $\arcp$-hypersheafification of $T(1)$-local $\TC$}\label{subsec:arcphyp}
Now, we verify that for certain class of $p$-quasisyntomic rings $R$, we have $L_{T(1)}\TC(R^{\wedge_{p}})\simeq (L_{T(1)}\TC)^{\sharp}(R^{\wedge_{p}})$. Below, quasisyntomic means $p$-quasisyntomic and $p$-complete as in \cite{bms2}. 

\begin{proposition}\label{prop:T(1)TChyppresperfd}
Let $C$ be a complete and algebraically closed nonarchimedean valued field of mixed characteristic $(0,p)$, and let $\mathcal{O}_{C}$ be its ring of integers. Let $\overline{A}$ be a perfectoid $\mathcal{O}_{C}$-algebra. For a $p$-torsion free quasisyntomic ring $R$ which is $p$-completely finitely generated over $\overline{A}$, one has a natural equivalence $L_{T(1)}\TC(R)\simeq (L_{T(1)}\TC)^{\sharp}(R)$. 
\end{proposition}

Here, for a $p$-complete commutative ring $\overline{A}$ and $R\in\CAlg^{p-\cpl}_{\overline{A}}$, we say that $R$ is $p$-completely finitely generated over $\overline{A}$ if there is a surjection $\overline{A}\langle x_{1},...,x_{n}\rangle\twoheadrightarrow R$ for some $n\geq0$. To prove Proposition \ref{prop:T(1)TChyppresperfd}, we first reduce to the case of very small $R$. We say that a commutative $\overline{A}$-algebra $R$ is very small if it admits a surjection $\overline{A}\langle x_{1}^{\pm1},...,x_{n}^{\pm1}\rangle\twoheadrightarrow R$ from a $p$-completely finitely generated torus. 

\begin{remarkn}\label{rem:verysmall}
Let $\overline{A}$ be a $p$-complete commutative ring, and let $R$ be a $p$-complete commutative $\overline{A}$-algebra. \\
\noindent (1) If $R$ is $p$-completely finitely generated over $\overline{A}$, then $\Spf R$ is covered by finitely many affine open subsets which are very small. Suppose that $R$ admits a surjection $\overline{A}\langle x_{1},...,x_{n}\rangle\twoheadrightarrow R$. Note that $\Spf\overline{A}\langle x_{1},...,x_{n}\rangle$ is covered in Zariski topology by finitely many tori over $\overline{A}$ isomorphic to $\Spf\overline{A}\langle x_{1}^{\pm1},...,x_{n}^{\pm1}\rangle$ via translations, as the analogous statement is true for $\bbA^{n}_{\overline{A}/p} = \Spec(\overline{A}/p)[x_{1},...,x_{n}]$. Taking the pullback of this affine open cover along the closed embedding $\Spf R\hookrightarrow \Spf\overline{A}\langle x_{1},...,x_{n}\rangle$ gives an affine, finite open cover of $\Spf R$ consisting of very small $\overline{A}$-algebras. \\
\noindent (2) If $R\to R'$ and $R\to R''$ are $\overline{A}$-algebra maps such that both $R'$ and $R''$ are very small, then $R'\otimesh_{R}R''$ is also very small. In fact, if $R'$ and $R''$ admit surjections $\overline{A}\langle x_{1}^{\pm1},...,x_{n}^{\pm1}\rangle\twoheadrightarrow R'$ and $\overline{A}\langle x_{1}^{\pm1},...,x_{m}^{\pm1}\rangle\twoheadrightarrow R''$ respectively, then we have a surjection $\overline{A}\langle x_{1}^{\pm1},...,x_{n}^{\pm1},y_{1}^{\pm1},...,y_{m}^{\pm1}\rangle\simeq \overline{A}\langle x_{1}^{\pm1},...,x_{n}^{\pm1}\rangle\otimesh_{\overline{A}}\overline{A}\langle x_{1}^{\pm1},...,x_{m}^{\pm1}\rangle\twoheadrightarrow R'\otimesh_{R}R''$.
\end{remarkn}

\begin{proof}[Proof of Proposition \ref{prop:T(1)TChyppresperfd}, reduction to the case of very small $\overline{A}$-algebras]
For $R$ which admits a surjection $\overline{A}\langle x_{1},...,x_{n}\rangle\twoheadrightarrow R$, we can take a Zariski cover $R\to R'$ with $R'$ being a finite product of very small $\overline{A}$-algebras by Remark \ref{rem:verysmall} (1). Then, each of the terms of the $p$-completed Cech nerve for $R\to R'$ is a finite product of very small $\overline{A}$-algebras by Remark \ref{rem:verysmall} (2). Since both $L_{T(1)}\TC$ and $(L_{T(1)}\TC)^{\sharp}$ are Zariski sheaves, statement for $R$ is now reduced to that of very small $\overline{A}$-algebras which are $p$-torsion free and quasisyntomic. 
\end{proof}

For the case of very small algebras, we can apply the pro-Galois descent result of the subsection \ref{subsec:progal} to complete the proof of Proposition \ref{prop:T(1)TChyppresperfd}. 

\begin{proof}[Proof of Proposition \ref{prop:T(1)TChyppresperfd}]
By our previous reduction argument, it suffices to check the case when the $\overline{A}$-algebra $R$ is very small. Suppose $R$ admits a surjection $\overline{A}\langle x_{1}^{\pm1},...,x_{n}^{\pm1}\rangle\twoheadrightarrow R$. Consider the maps of commutative rings $B = B_{0}\to B_{1}\to B_{2}\to\cdots \to B_{\infty}$, where $B_{i} = \bbZ[\zeta_{p^{\infty}},x_{1}^{\pm 1/p^{i}},...,x_{n}^{\pm1/p^{i}}]$ for each $i\geq0$ and $B_{\infty} = \bbZ[\zeta_{p^{\infty}},x_{1}^{\pm 1/p^{\infty}},...,x_{n}^{\pm 1/p^{\infty}}] = \colim_{\bbN}B_{i}$. Note that each finite and finitely presented $B$-algebra $B_{i}$ admits a natural action of the group $\mu_{p^{i}}^{n}$ making $B[1/p]\to B_{i}[1/p]$ a $\mu_{p^{i}}^{n}$-Galois extension; $(\zeta_{1},...,\zeta_{n})\in\mu_{p^{i}}^{n}$ acts on $x_{1}^{m_{1}/p^{i}}\cdots x_{n}^{m_{n}/p^{i}}\in B_{i}$ as a multiplication by $\zeta_{1}^{m_{1}}\cdots \zeta_{n}^{m_{n}}$. These actions are compatible with each other, and makes $B[1/p]\to B_{\infty}[1/p]$ a (pro-)Galois extension with profinite Galois group $G:=\bbZ_{p}(1)^{n} = \lim(\cdots\xrightarrow{(-)^{p}}\mu_{p^{2}}^{n}\xrightarrow{(-)^{p}}\mu_{p}^{n})$ which has finite $p$-cohomological dimension and acts continuously on the $B$-algebra $B_{\infty}$ endowed with the discrete topology. \\
\indent Fix a map $\bbZ[\zeta_{p^{\infty}}]\to\mathcal{O}_{C}$ determined by a choice of a compatible system of primitive $p$-power roots of unity of the target. Through this, we have a map $B\to \overline{A}\langle x^{\pm1}_{1},...,x^{\pm1}_{n}\rangle$ making $R$, a quotient of the target, an algebra over $B$. Let $R_{\infty} = B_{\infty}\otimes_{B}R$ and let $R'_{\infty} = R_{\infty}^{\wedge_{p}}$. By construction, $R_{\infty}'$ admits a map $\overline{A}\langle x_{1}^{\pm 1/p^{\infty}},...,x_{n}^{\pm 1/p^{\infty}}\rangle\to R_{\infty}'$ from a perfectoid $\mathcal{O}_{C}$-algebra $\overline{A}\langle x_{1}^{\pm 1/p^{\infty}},...,x_{n}^{\pm 1/p^{\infty}}\rangle$, which is surjective by derived Nakayama and the surjectivity of $(\overline{A}/p)[x_{1}^{\pm1},...,x_{n}^{\pm1}]\twoheadrightarrow R/p$. In particular, $R'_ {\infty}\in\CAlg^{\qrsp}_{\mathbb{Z}_{p},\mathcal{O}_{C}}$. \\
\indent Now, by Proposition \ref{prop:progal}, we have an equivalence 
\begin{equation*}
L_{T(1)}\TC(R)\simeq \lim(L_{T(1)}\TC(R'_{\infty})\rightrightarrows L_{T(1)}\TC(R'_{\infty}\otimesh_{R}R'_{\infty})~\substack{\rightarrow\\[-1em] \rightarrow \\[-1em] \rightarrow}\cdots)
\end{equation*} 
using that $R_{\infty}^{\otimesh_{R}(k+1)}\simeq (R'_{\infty})^{\otimesh_{R}(k+1)}$ by our notation. On the other hand, since $R\to R'_{\infty}$ is a quasisyntomic, and hence an $\arcp$-cover, we have an equivalence 
\begin{equation*}
(L_{T(1)}\TC)^{\sharp}(R)\simeq \lim((L_{T(1)}\TC)^{\sharp}(R'_{\infty})\rightrightarrows (L_{T(1)}\TC)^{\sharp}(R'_{\infty}\otimesh_{R}R'_{\infty})~\substack{\rightarrow\\[-1em] \rightarrow \\[-1em] \rightarrow}\cdots).
\end{equation*} 
Since each of the terms $(R'_{\infty})^{\otimesh_{R}(k+1)}$ is in $\CAlg^{\qrsp}_{\mathbb{Z}_{p},\mathcal{O}_{C}}$, we know $L_{T(1)}\TC$ and $(L_{T(1)}\TC)^{\sharp}$ agree on it by Proposition \ref{prop:T(1)TCarcphypqrsp}. Thus, the right hand sides of the equivalences are equivalent to each other, and hence $L_{T(1)}\TC(R)\simeq (L_{T(1)}\TC)^{\sharp}(R)$ naturally as desired. 
\end{proof}

\begin{proposition}\label{prop:T(1)TChyppres}
Let $\mathcal{O}_{K}$ be a complete discrete valuation ring of mixed characteristic $(0,p)$ whose residue field $\kappa$ satisfies $[\kappa:\kappa^{p}]<\infty$. Suppose that $R$ is a $p$-torsion free quasisyntomic ring which is $p$-completely finitely generated over $\mathcal{O}_{K}$. Then, $L_{T(1)}\TC(R)\simeq(L_{T(1)}\TC)^{\sharp}(R)$.
\begin{proof}
We argue as in the proof of Proposition \ref{prop:T(1)TChyppresperfd}. Write $B = \mathcal{O}_{K}$, and let $B_{\infty} = \mathcal{O}_{\overline{K}}$ equipped with the action of $G = \Gal(\overline{K}/K)$, where $\overline{K}$ is an algebraic closure of $K = \Frac(\mathcal{O}_{K})$. Note that $B_{\infty}$ is a filtered colimit of $\mathcal{O}_{L}$ over finite extensions $L$ of $K$, and that each $\mathcal{O}_{L}$ is finite free as a $B$-module. Under the present assumption on $K$, the group $G$ has finite $p$-cohomological dimension \cite{go}. Also, let $R_{\infty} = B_{\infty}\otimes_{B}R$, and write $R_{\infty}' = R_{\infty}^{\wedge_{p}}$ for its $p$-completion; by construction, $R_{\infty}'$ is quasisyntomic and $R\to R_{\infty}'$ is a quasisyntomic cover. Also, note that $R_{\infty}'$ is $p$-completely finitely generated over $\mathcal{O}_{C}$, where $C = \widehat{\overline{K}}$. In fact, if $\mathcal{O}_{K}\langle x_{1},...,x_{n}\rangle\twoheadrightarrow R$, then the surjectivity of the map $\mathcal{O}_{C}\langle x_{1},...,x_{n}\rangle\simeq (\mathcal{O}_{\overline{K}}\otimes_{\mathcal{O}_{K}}\mathcal{O}_{K}\langle x_{1},...,x_{n}\rangle)^{\wedge_{p}}\to R_{\infty}'$ follows from derived Nakayama and the surjectivity of $(\mathcal{O}_{K}/p)[x_{1},...,x_{n}]\twoheadrightarrow R/p$. Thus, each terms $R_{\infty}^{\otimesh_{R}(k+1)}\simeq (R'_{\infty})^{\otimesh_{R}(k+1)}$ of the $p$-completed Cech nerve for the $\arcp$-cover $R\to R_{\infty}'$ is $p$-completely finitely generated over $\mathcal{O}_{C}$ and quasisyntomic. By Proposition \ref{prop:T(1)TChyppresperfd}, we have $L_{T(1)}\TC\left((R'_{\infty})^{\otimesh_{R}(k+1)}\right)\simeq(L_{T(1)}\TC)^{\sharp}\left((R'_{\infty})^{\otimesh_{R}(k+1)}\right)$ for all $k\geq0$. Hence, the right hand side of the equation (\ref{eq:T(1)TCprogal}) computes $(L_{T(1)}\TC)^{\sharp}(R)$ due to the $\arcp$-descent property of $(L_{T(1)}\TC)^{\sharp}$, and we conclude $L_{T(1)}\TC(R)\simeq (L_{T(1)}\TC)^{\sharp}(R)$ by Proposition \ref{prop:progal}. 
\end{proof}
\end{proposition}

\begin{corollary}\label{cor:T(1)TCthomasonfilt}
Suppose that $R$ is an animated commutative ring such that $\pi_{0}(R^{\wedge_{p}})$ satisfies the assumptions for the commutative ring $R$ in Proposition \ref{prop:T(1)TChyppresperfd} or Proposition \ref{prop:T(1)TChyppres}. Then, $L_{T(1)}\TC(R)$ admits a complete and exhaustive multiplicative $\bbZ$-indexed descending filtration whose $n$-th associated graded piece is naturally equivalent to $\RG_{\et}(\Spec R^{\wedge_{p}}[1/p],\bbZ_{p}(n))[2n]$. 
\begin{proof}
Since $L_{T(1)}\TC$ is truncating on $T(1)$-acyclic spectra and is invariant for $p$-completions, and since $\RG_{\et}(\Spec R^{\wedge_{p}}[1/p],\bbZ_{p}(n))$ depends only on the underlying ordinary commutative ring $\pi_{0}(R^{\wedge_{p}})[1/p]$, it suffices to check that $L_{T(1)}\TC(\pi_{0}(R^{\wedge_{p}}))$ admits a filtration with the properties as stated above. By Proposition \ref{prop:T(1)TChyppresperfd} or Proposition \ref{prop:T(1)TChyppres} depending on the assumption, $L_{T(1)}\TC(\pi_{0}(R^{\wedge_{p}}))\simeq (L_{T(1)}\TC)^{\sharp}(\pi_{0}(R^{\wedge_{p}}))$. The right hand side of the equivalence admits a filtration of the claimed form through Construction \ref{constr:filtonarcpsheaf}, which gives the claimed filtration on $L_{T(1)}\TC(\pi_{0}(R^{\wedge_{p}}))$. 
\end{proof}
\end{corollary}

\subsection{Construction of the filtration II}\label{subsec:filtconstr2}

Using Proposition \ref{prop:T(1)TChyppres}, we can construct a filtration on $T(1)$-local $\TC$ via left Kan extension from finitely generated polynomial algebras. First, we observe the following consequence of the \'etale comparison of Proposition \ref{prop:etcomp}, which can be perceived as an algebraic analogue of Lemma \ref{lem:T(1)TClke} for the cohomology theory appearing as associated graded pieces:

\begin{corollary}\label{cor:T(1)TCgrlke}
For each $n\in\bbZ$, the functor 
\begin{equation*}
R\mapsto \RG_{\et}(\Spec R^{\wedge_{p}}[1/p],\bbZ_{p}(n)):\CAlg^{\an}\to \Dh(\bbZ_{p})
\end{equation*}
commutes with sifted colimits.
\begin{proof}
First, we check the analogous statement for the same functor on $\CAlg^{\an}_{\overline{A}}$, where $\overline{A}$ is a perfectoid ring corresponding to a perfect prism $(A,(d))$. By Proposition \ref{prop:etcomp}, the functor is equivalent to $R\mapsto \left(\prism_{R/A}\{n\}[\frac{1}{d}]\right)^{F=1}\simeq \left(\prism_{R^{\wedge_{p}}/A}\{n\}[\frac{1}{d}]\right)^{F=1}$. Since $\prism_{-/A}\{n\}$ commutes with sifted colimits and that sifted colimits commute with fiber-cofiber sequences, the latter functor preserves sifted colimits.\\
\indent Now, we check the case of arbitrary base rings, i.e., the case of the functor on $\CAlg^{\an}$. It suffices to check that the functor 
\begin{equation*}
F_{0} = \left(R\mapsto \RG_{\et}(\Spec R^{\wedge_{p}}[1/p],\mu_{p}^{\otimes n})\right):\CAlg^{\an}\to \mathcal{D}(\bbF_{p})
\end{equation*}
commutes with sifted colimits. Consider the functor $F = (R\mapsto \RG_{\et}(\Spec R[\zeta_{p^{\infty}}]^{\wedge_{p}}[1/p],\mu_{p}^{\otimes n}))$ obtained as a composition of $F_{0}$ with the base change $-\otimes_{\bbZ}\bbZ_{p}^{\cyc}:\CAlg^{\an}\to\CAlg^{\an}_{\bbZ_{p}^{\cyc}}$. Since we know $F_{0}$ restricted to $\CAlg^{\an}_{\bbZ_{p}^{\cyc}}$ preserves sifted colimits by the previous paragraph applied to $\overline{A} = \bbZ_{p}^{\cyc}$, the functor $F$ commutes with sifted colimits. By construction, $F$ naturally lifts to a sifted colimit preserving functor valued in $\mathcal{D}(\mathbb{F}_{p})^{B\bbZ_{p}^{\times}}$, and by faithfully flat descent of \'etale cohomology, there is an equivalence $F_{0}\simeq F^{h\bbZ_{p}^{\times}}$ by taking a homotopy fixed point for this continuous $\bbZ_{p}^{\times}$-action. Write $G$ for the torsion subgroup of $\bbZ_{p}^{\times}$, which in particular is finite cyclic. Since $F^{h\bbZ_{p}^{\times}}\simeq (F^{hG})^{h\bbZ_{p}}$ and $(-)^{h\bbZ_{p}}\simeq\fib(\psi-1)$ for the action $\psi$ of a topological generator of $\bbZ_{p}$, we are reduced to check that the functor $R\mapsto F(R)^{hG}:\CAlg^{\an}\to\mathcal{D}(\bbF_{p})$ commutes with sifted colimits. The functor commutes with filtered colimits, due to the facts that $F$ (valued in $\mathcal{D}(\bbF_{p})$) commutes with filtered colimits and factors through $\mathcal{D}(\bbF_{p})^{\geq 0}$ and that $(-)^{hG}$ commutes with filtered colimits in $\mathcal{D}(\bbF_{p})^{\geq0}$. Finally, the functor also commutes with geometric realizations (which in particular are filtered colimits of $m$-skeletal geometric realizations over $\Delta^{\op}_{\leq m}$), since $F$ valued in $\mathcal{D}(\bbF_{p})$ preserves geometric realizations and factors through $\mathcal{D}(\bbF_{p})^{\geq0}$, the functor $(-)^{hG}$ preserves filtered colimits in $\mathcal{D}(\bbF_{p})^{\geq0}$, and $(-)^{hG}$ preserves finite colimits in $\mathcal{D}(\bbF_{p})$. 
\end{proof}
\end{corollary}

\begin{remarkn}
By Corollary \ref{cor:T(1)TCgrlke}, we also know that for each $n\in\bbZ$, the closely related functor 
\begin{equation*}
R\mapsto \RG_{\et}(\Spec R[1/p],\bbZ_{p}(n)):\CAlg^{\an}\to \Dh(\bbZ_{p})
\end{equation*}
is left Kan extended from its restriction to the full subcategory $\CAlg^{\sm}$ spanned by smooth $\bbZ$-algebras in $\CAlg^{\an}$. In fact, the fiber of the natural map $\RG_{\et}(\Spec R[1/p],\bbZ_{p}(n))\to\RG_{\et}(\Spec R^{\wedge_{p}}[1/p],\bbZ_{p}(n))$ satisfies rigidity for Henselian pairs, cf. \cite[Rem. 8.4.4]{bl}. The arguments in the proof of \cite[Prop. 8.4.10]{bl} shows that this rigidity combined with the property of left Kan extension from smooth (even polynomial) algebras for the functor of Corollary \ref{cor:T(1)TCgrlke} implies the stated property of left Kan extension from smooth algebras for $R\mapsto \RG_{\et}(\Spec R[1/p],\bbZ_{p}(n))$.
\end{remarkn}

Thus, for any $B\in\CAlg^{\an}$, we in particular know from Lemma \ref{lem:T(1)TClke} and Corollary \ref{cor:T(1)TCgrlke} that the functors $L_{T(1)}\TC:\CAlg^{\an}_{B}\to\Sp^{\Cpl(p)}$ and $R\mapsto \RG_{\et}(\Spec R^{\wedge_{p}}[1/p],\bbZ_{p}(n)):\CAlg^{\an}_{B}\to \Dh(\bbZ_{p})$ commute with sifted colimits and hence are left Kan extended from the full subcategory $\CAlg^{\poly}_{B}$ of finitely generated polynomial algebras over $B$ in $\CAlg^{\an}_{B}$. Below, we again drop $B$ from these notations when $B=\bbZ$. 

\begin{construction}[Filtration via left Kan extension]\label{constr:filtvialke}
We write $\Fil^{\sbullet}_{\mathrm{lke}}L_{T(1)}\TC$ for the functor
\begin{equation*}
\Fil^{\sbullet}_{\mathrm{lke}}L_{T(1)}\TC\simeq L\left(\Fil^{\sbullet}(L_{T(1)}\TC)^{\sharp}|_{\CAlg^{\poly}}\right)\in\Fun_{\Sigma}\left(\CAlg^{\an},\CAlg(\DFh(\bbS))\right)
\end{equation*}
obtained by taking left Kan extension $L:\Fun(\CAlg^{\poly},\CAlg(\DFh(\bbS)))\simeq \Fun_{\Sigma}(\CAlg^{\an},\CAlg(\DFh(\bbS)))$ of the functor $\Fil^{\sbullet}(L_{T(1)}\TC)^{\sharp}|_{\CAlg^{\poly}}\in\Fun(\CAlg^{\poly},\CAlg(\DFh(\bbS)))$ on finitely generated polynomial rings, the restriction of the filtration from Construction \ref{constr:filtonarcpsheaf}, along the inclusion $\CAlg^{\poly}\hookrightarrow \CAlg^{\an}$. 
\end{construction}

\begin{remarkn}
Let $\mathcal{D}$ be an $\infty$-category which admits sifted colimits. For each $F\in\Fun(\CAlg^{\poly}, \mathcal{D})$, the value of its left Kan extension $LF\in\Fun_{\Sigma}(\CAlg^{\an},\mathcal{D})$ at $R\in\CAlg^{\an}$ is computed as a colimit in $\mathcal{D}$ over the $\infty$-category $\CAlg^{\poly}_{/R}$, which admits finite coproducts and in particular is sifted. Suppose $\mathcal{D}$ moreover admits a symmetric monoidal structure which is compatible with sifted colimits. Then, sifted colimits in commutative algebra objects of $\mathcal{D}$ are computed in the underlying $\infty$-category $\mathcal{D}$ \cite[Cor. 3.2.3.2]{ha}, and hence the functor $L:\Fun(\CAlg^{\poly},\CAlg(\mathcal{D}^{\otimes}))\simeq \Fun_{\Sigma}(\CAlg^{\an},\CAlg(\mathcal{D}^{\otimes}))$ is compatible with the functor $L:\Fun(\CAlg^{\poly},\mathcal{D}) \simeq \Fun_{\Sigma}(\CAlg^{\an},\mathcal{D})$ through the forgetful functor $\CAlg(\mathcal{D}^{\otimes})\to\mathcal{D}$. For instance, $L(L_{T(1)}\TC|_{\CAlg^{\poly}})\simeq L_{T(1)}\TC$ in $\Fun_{\Sigma}(\CAlg^{\an},\CAlg(\Sp^{\Cpl(p)}))$ by Lemma \ref{lem:T(1)TClke}. Also, the underlying functor $R\mapsto \Fil^{\sbullet}_{\mathrm{lke}}L_{T(1)}\TC(R):\CAlg^{\an}\to\DFh(\bbS)$ of $\Fil^{\sbullet}_{\mathrm{lke}}L_{T(1)}\TC$ is equivalent to a left Kan extension of the functor $\Fil^{\sbullet}(L_{T(1)}\TC)^{\sharp}|_{\CAlg^{\poly}}:\CAlg^{\poly}\to\DFh(\bbS)$. 
\end{remarkn}

\begin{proposition}\label{prop:filtvialke}
The following holds for the functor $\Fil^{\sbullet}_{\mathrm{lke}}L_{T(1)}\TC$ of Construction \ref{constr:filtvialke}:\\
(1) There is a map 
\begin{equation*}
\Fil^{\sbullet}_{\mathrm{lke}}L_{T(1)}\TC\to L_{T(1)}\TC
\end{equation*} 
of $\Fun_{\Sigma}(\CAlg^{\an},\CAlg(\DFh(\bbS)))$ which realizes the source $\Fil^{\sbullet}_{\mathrm{lke}}L_{T(1)}\TC(R)$ as an exhaustive filtration on $L_{T(1)}\TC(R)$ for each $R\in\CAlg^{\an}$.  \\
(2) For each $n\in\bbZ$, there is a natural equivalence 
\begin{equation*}
\gr^{n}\Fil_{\mathrm{lke}}L_{T(1)}\TC(R)\simeq\RG_{\et}(\Spec R^{\wedge_{p}}[1/p],\bbZ_{p}(n))[2n]
\end{equation*} 
in $\Dh(\bbZ_{p})$ for $R\in\CAlg^{\an}$. 
\begin{proof}
(1) From Corollary \ref{cor:T(1)TCthomasonfilt} and Proposition \ref{prop:T(1)TChyppres} for the case of $\mathcal{O}_{K} = \bbZ_{p}$, we know there is a map $\Fil^{\sbullet}(L_{T(1)}\TC)^{\sharp}|_{\CAlg^{\poly}}\to L_{T(1)}\TC|_{\CAlg^{\poly}}$ of $\Fun(\CAlg^{\poly},\CAlg(\DFh(\bbS)))$ which realizes the source as an exhaustive filtration on $L_{T(1)}\TC|_{\CAlg^{\poly}}\simeq (L_{T(1)}\TC)^{\sharp}|_{\CAlg^{\poly}}$. The map of $\Fun_{\Sigma}(\CAlg^{\an},\CAlg(\DFh(\bbS)))$ obtained through left Kan extension, by Construction \ref{constr:filtvialke} and Lemma \ref{lem:T(1)TClke}, gives a desired map 
\begin{equation*}
\Fil^{\sbullet}_{\mathrm{lke}}L_{T(1)}\TC\simeq L\left(\Fil^{\sbullet}(L_{T(1)}\TC)^{\sharp}|_{\CAlg^{\poly}}\right) \to L\left(L_{T(1)}\TC|_{\CAlg^{\poly}}\right)\simeq L_{T(1)}\TC.
\end{equation*} 
Since values of left Kan extensions are computed by colimits (over categories of the form $\CAlg^{\poly}_{/R}$), the map realizes $L\left(\Fil^{\sbullet}(L_{T(1)}\TC)^{\sharp}|_{\CAlg^{\poly}}\right)$ as an exhaustive filtration on  $L\left(L_{T(1)}\TC|_{\CAlg^{\poly}}\right)$. \\
(2) It suffices to check that the functor $R\mapsto \RG_{\et}(\Spec R^{\wedge_{p}}[1/p],\bbZ_{p}(n)):\CAlg^{\an}\to\Dh(\bbZ_{p})$ is left Kan extended from its restriction to $\CAlg^{\poly}$. This property follows from Corollary \ref{cor:T(1)TCgrlke} above. 
\end{proof}
\end{proposition}

\subsection{Comparisons of the filtrations on $T(1)$-local $\TC$ and its hypercompletion} \label{subsec:compwithhyp}

In this subsection, we provide another characterization of the exhaustive, left Kan extended filtration on $L_{T(1)}\TC$ of Construction \ref{constr:filtvialke} in terms of a pullback of $\Fil^{\sbullet}(L_{T(1)}\TC)^{\sharp}$ from Construction \ref{constr:filtonarcpsheaf} along the $\arcp$-hypersheafification map (Construction \ref{constr:filtviapb} and Proposition \ref{prop:filtcompletecomparison}). Through this identification, we can verify useful colimit and limit-related properties for the filtration on $L_{T(1)}\TC$ which weren't immediate from each of the constructions, e.g., flat descent, preservation of sifted colimits, and a criterion for the filtration to be filtration-complete; we collect them in Theorem \ref{th:T(1)TCfilt}. For the properties of pullback and completion constructions for filtered objects we use here, see Appendix \ref{sec:appendix}.  

\begin{construction}[Filtration via pullback from the hypercompletion] \label{constr:filtviapb}
We write $\Fil^{\sbullet}_{\mathrm{pb}}L_{T(1)}\TC$ for the functor
\begin{equation*}
\Fil^{\sbullet}_{\mathrm{pb}}L_{T(1)}\TC\simeq \left(R\mapsto h^{\ast}\Fil^{\sbullet}(L_{T(1)}\TC)^{\sharp}(R)\right):\CAlg^{\an}\to \CAlg(\DFh(\bbS))
\end{equation*}
obtained by applying the pullback construction from Construction \ref{constr:filtpullback}, for each $R\in\CAlg^{\an}$, to the natural map 
\begin{equation*}
h:L_{T(1)}\TC(R)\simeq L_{T(1)}\TC(\pi_{0}(R^{\wedge_{p}}))\to (L_{T(1)}\TC)^{\sharp}(\pi_{0}(R^{\wedge_{p}}))
\end{equation*} 
of $\CAlg(\Sp^{\Cpl(p)})$ given by an $\arcp$-hypersheafification, and the map of multiplicative filtered $p$-complete spectra 
\begin{equation*}
\Fil^{\sbullet}(L_{T(1)}\TC)^{\sharp}(\pi_{0}(R^{\wedge_{p}}))\to(L_{T(1)}\TC)^{\sharp}(\pi_{0}(R^{\wedge_{p}}))
\end{equation*}
from Construction \ref{constr:filtonarcpsheaf}. Note that each of the pullbacks of commutative algebra objects in the construction is computed in the underlying $\infty$-category $\Sp^{\Cpl(p)}$ (or $\Fun(\CAlg^{\an},\Sp^{\Cpl(p)})$), cf. \cite[Cor. 3.2.2.4]{ha}. Also, recall that the latter map of filtered objects realizes the source $\Fil^{\sbullet}(L_{T(1)}\TC)^{\sharp}(R):=\Fil^{\sbullet}(L_{T(1)}\TC)^{\sharp}(\pi_{0}(R^{\wedge_{p}}))$ as an exhaustive filtration of $(L_{T(1)}\TC)^{\sharp}(\pi_{0}(R^{\wedge_{p}}))$. \\
\indent By construction, there is a map 
\begin{equation*}
\Fil^{\sbullet}_{\mathrm{pb}}L_{T(1)}\TC\to L_{T(1)}\TC
\end{equation*}
of presheaves of multiplicative filtered $p$-complete spectra on $\CAlg^{\an}$, and by Lemma \ref{lem:filtpullback} (2), this map realizes $\Fil^{\sbullet}_{\mathrm{pb}}L_{T(1)}\TC$ as an exhaustive filtration on $L_{T(1)}\TC$. By Lemma \ref{lem:filtpullback} (3), we also know that $\gr^{\sbullet}\Fil_{\mathrm{pb}}L_{T(1)}\TC\simeq \gr^{\sbullet}(L_{T(1)}\TC)^{\sharp}$. 
\end{construction}

Thus, through Construction \ref{constr:filtvialke} and Construction \ref{constr:filtviapb}, we have obtained two (\emph{a priori} different) constructions of exhaustive filtrations $\Fil^{\sbullet}_{\mathrm{lke}}L_{T(1)}\TC$ and $\Fil^{\sbullet}_{\mathrm{pb}}L_{T(1)}\TC$ of $L_{T(1)}\TC$. To compare these two filtered objects, consider the natural map 
\begin{equation} \label{eq:filtcomparison}
\Fil^{\sbullet}_{\mathrm{lke}}L_{T(1)}\TC\simeq L\left(\Fil^{\sbullet}(L_{T(1)}\TC)^{\sharp}|_{\CAlg^{\poly}}\right)\to\Fil^{\sbullet}(L_{T(1)}\TC)^{\sharp}
\end{equation} 
of presheaves of multiplicative filtered $p$-complete spectra on $\CAlg^{\an}$.

\begin{remarkn}\label{rem:filtundobj}
The map $\Fil^{\sbullet}_{\mathrm{lke}}L_{T(1)}\TC\to\Fil^{\sbullet}(L_{T(1)}\TC)^{\sharp}$ of (\ref{eq:filtcomparison}) induces the map 
\begin{equation*}
h:L_{T(1)}\TC\to (L_{T(1)}\TC)^{\sharp}(\pi_{0}(-)^{\wedge_{p}})
\end{equation*} 
of underlying objects in $\Fun\left(\CAlg^{\an},\CAlg(\Sp^{\Cpl(p)})\right)$, which we used in Construction \ref{constr:filtviapb}. In fact, for $R\in\CAlg^{\an}$, taking underlying objects for the map of filtrations of (\ref{eq:filtcomparison}) after evaluating $R$ gives the canonical map of $p$-complete spectra
\begin{equation}\label{eq:filtundobjintermediate}
\colim_{S\in(\CAlg^{\poly})_{/R}} (L_{T(1)}\TC)^{\sharp}(\pi_{0}(S^{\wedge_{p}}))\to (L_{T(1)}\TC)^{\sharp}(\pi_{0}(R^{\wedge_{p}})).
\end{equation}
To further describe this map, note that we have the following natural equivalence:
\begin{equation*}
\begin{tikzcd}
L_{T(1)}\TC(R)\simeq \colim_{S\in(\CAlg^{\poly})_{/R}}L_{T(1)}\TC(S) \arrow[r, "h", "\sim"'] & \colim_{S\in(\CAlg^{\poly})_{/R}}(L_{T(1)}\TC)^{\sharp}(\pi_{0}(S^{\wedge{p}})).
\end{tikzcd}
\end{equation*}
Here, we are using Lemma \ref{lem:T(1)TClke} to ensure the first map is an equivalence, while using Proposition \ref{prop:T(1)TChyppres} for $\mathcal{O}_{K}=\bbZ_{p}$ to know the second map induced from $h$ is also an equivalence. By composing this equivalence with the map (\ref{eq:filtundobjintermediate}), we know the latter map is equivalent to the map $h:L_{T(1)}\TC(R)\to (L_{T(1)}\TC)^{\sharp}(\pi_{0}(R^{\wedge_{p}}))$ as claimed. 
\end{remarkn}

By Remark \ref{rem:filtundobj} above, we obtain the following square of presheaves 
\begin{equation*}
\begin{tikzcd}
\Fil^{\sbullet}_{\mathrm{lke}}L_{T(1)}\TC \arrow[r, "(\ref{eq:filtcomparison})"] \arrow[d] & \Fil^{\sbullet}(L_{T(1)}\TC)^{\sharp} \arrow[d] \\
L_{T(1)}\TC \arrow[r, "h"'] & (L_{T(1)}\TC)^{\sharp}(\pi_{0}(-)^{\wedge_{p}})
\end{tikzcd}
\end{equation*}
of multiplicative filtered $p$-complete spectra on $\CAlg^{\an}$ whose top horizontal arrow is (\ref{eq:filtcomparison}) and the vertical arrows are the natural maps, which can also be understood as induced from unit maps for the adjunction $p_{!}\dashv p^{\ast}$ as discussed in the paragraph above Lemma \ref{lem:filtexhaustivegr}. From this square, we induce the natural comparison map 

\begin{equation}\label{eq:filtcomparison2}
\Fil^{\sbullet}_{\mathrm{lke}}L_{T(1)}\TC\to h^{\ast}\Fil^{\sbullet}(L_{T(1)}\TC)^{\sharp}\simeq \Fil^{\sbullet}_{\mathrm{pb}}L_{T(1)}\TC
\end{equation}
of presheaves of multiplicative filtered $p$-complete spectra on $\CAlg^{\an}$ over the object $L_{T(1)}\TC$.  \\
\indent Below, we explain that these two filtrations on $L_{T(1)}\TC$, one obtained via left Kan extension and the other obtained via pullback construction, are in fact equivalent to each other, and that the filtration on $(L_{T(1)}\TC)^{\sharp}$ we constructed earlier is their filtration-completion. 

\begin{proposition}[Comparisons of filtrations] \label{prop:filtcompletecomparison}
We have the following comparisons between the filtrations $\Fil^{\sbullet}_{\mathrm{lke}}L_{T(1)}\TC$ from Construction \ref{constr:filtvialke} and $\Fil^{\sbullet}_{\mathrm{pb}}L_{T(1)}\TC$ from Construction \ref{constr:filtviapb} on $L_{T(1)}\TC$, and $\Fil^{\sbullet}(L_{T(1)}\TC)^{\sharp}$ on an $\arcp$-hypersheafification $(L_{T(1)}\TC)^{\sharp}$ from Construction \ref{constr:filtonarcpsheaf}:\\
(1) The map $\Fil^{\sbullet}_{\mathrm{lke}}L_{T(1)}\TC\to \Fil^{\sbullet}_{\mathrm{pb}}L_{T(1)}\TC$ of (\ref{eq:filtcomparison2}) is an equivalence of filtered objects over $L_{T(1)}\TC$.\\
(2) The map $\Fil^{\sbullet}_{\mathrm{lke}}L_{T(1)}\TC\to\Fil^{\sbullet}(L_{T(1)}\TC)^{\sharp}$ of (\ref{eq:filtcomparison}) realizes the target as a filtration-completion of the source. 
\begin{proof}
We first check (2). Note that the map $\Fil^{\sbullet}_{\mathrm{lke}}L_{T(1)}\TC\to\Fil^{\sbullet}(L_{T(1)}\TC)^{\sharp}$ by construction induces an equivalence $\gr^{\sbullet}\Fil_{\mathrm{lke}}L_{T(1)}\TC\simeq L(\gr^{\sbullet}(L_{T(1)}\TC)^{\sharp}|_{\CAlg^{\poly}})\simeq \gr^{\sbullet}(L_{T(1)}\TC)^{\sharp}$ on associated graded objects and that the target filtered object $\Fil^{\sbullet}(L_{T(1)}\TC)^{\sharp}$ is filtration-complete. By Lemma \ref{lem:filtcompletion} (2), we have the claimed result. \\
\indent For (1), apply the associated graded objects functor on the diagram of $\Fun(\CAlg^{\an},\DFh(\bbS))$
\begin{equation*} 
\begin{tikzcd}
\Fil^{\sbullet}_{\mathrm{lke}}L_{T(1)}\TC \arrow[rd, "(\ref{eq:filtcomparison2})"'] \arrow[rrd, bend left=10, "(\ref{eq:filtcomparison})"] & & \\
       &  \Fil^{\sbullet}_{\mathrm{pb}}L_{T(1)}\TC \arrow[r] & \Fil^{\sbullet}(L_{T(1)}\TC)^{\sharp}
\end{tikzcd}
\end{equation*}
to obtain the diagram of $\Fun(\CAlg^{\an},\Dh(\bbZ_{p})^{\gr})$
\begin{equation*} 
\begin{tikzcd}
\gr^{\sbullet}\Fil_{\mathrm{lke}}L_{T(1)}\TC \arrow[rd, "\gr^{\sbullet}(\ref{eq:filtcomparison2})"'] \arrow[rrd, bend left=10, "\gr^{\sbullet}(\ref{eq:filtcomparison})", "\sim"'] & & \\
       &  \gr^{\sbullet}\Fil_{\mathrm{pb}}L_{T(1)}\TC \arrow[r, "\sim"'] & \gr^{\sbullet}\Fil(L_{T(1)}\TC)^{\sharp}.
\end{tikzcd}
\end{equation*}
We already have observed that the map $\gr^{\sbullet}(\ref{eq:filtcomparison})$ is an equivalence while checking the statement (2) in the previous paragraph. On the other hand, the lower horizontal map is an equivalence by Lemma \ref{lem:filtpullback} (3). Thus, we know the map $\gr^{\sbullet}(\ref{eq:filtcomparison2})$ is an equivalence. Now, since the source and the target filtered objects for the map (\ref{eq:filtcomparison2}) are exhaustive filtrations for the common underlying object $L_{T(1)}\TC$, we know the map (\ref{eq:filtcomparison2}) is an equivalence by Lemma \ref{lem:filtexhaustivegr}. 
\end{proof}
\end{proposition}

\begin{notation}\label{not:T(1)TCfilt}
We write $\Fil^{\sbullet}L_{T(1)}\TC$ for the common presheaf
\begin{equation*}
\Fil^{\sbullet}L_{T(1)}\TC:=\Fil^{\sbullet}_{\mathrm{lke}}L_{T(1)}\TC\simeq\Fil^{\sbullet}_{\mathrm{pb}}L_{T(1)}\TC
\end{equation*} 
of multiplicative $p$-complete spectra on $\CAlg^{\an}$; the equivalence of these two presheaves is guaranteed through Proposition \ref{prop:filtcompletecomparison} (1). We also write $\gr^{\sbullet}L_{T(1)}\TC:=\gr^{\sbullet}\Fil(L_{T(1)}\TC)$ for its associated graded object in $\Fun(\CAlg^{\an},\Dh(\bbZ_{p})^{\gr})$. By definition, there is a natural map 
\begin{equation*}
\Fil^{\sbullet}L_{T(1)}\TC\to L_{T(1)}\TC
\end{equation*} 
of multiplicative filtered objects which realizes $\Fil^{\sbullet}L_{T(1)}\TC$ as an exhaustive filtration on its underlying object $L_{T(1)}\TC$. 
\end{notation}

\begin{theorem}[Filtration on $T(1)$-local $\TC$] \label{th:T(1)TCfilt}
There exists a functor $\Fil^{\sbullet}L_{T(1)}\TC:\CAlg^{\an}\to\CAlg(\DFh(\bbS))$ which satisfies the following properties:\\
(1) There is a map 
\begin{equation*}
\Fil^{\sbullet}L_{T(1)}\TC\to L_{T(1)}\TC
\end{equation*} 
of presheaves of multiplicative filtered $p$-complete spectra on $\CAlg^{\an}$ which realizes the source as an exhaustive filtration of the target $L_{T(1)}\TC$ viewed as a constant filtered object. \\
(2) $\Fil^{\sbullet}L_{T(1)}\TC$ commutes with sifted colimits. \\
(3) $\Fil^{\sbullet}L_{T(1)}\TC$ induces equivalences on taking functors $\pi_{0}:\CAlg^{\an}\to\CAlg$ and $(-)^{\wedge_{p}}:\CAlg^{\an}\to \CAlg^{\an,p-\cpl}$, and satisfies $p$-complete fppf descent. \\
(4) For each $n\in\bbZ$, write $\gr^{n}L_{T(1)}\TC$ for the $n$-th associated graded objects functor for $\Fil^{\sbullet}L_{T(1)}\TC$. There is a natural equivalence
\begin{equation*}
\gr^{n}L_{T(1)}\TC(R)\simeq\RG_{\et}(\Spec R^{\wedge_{p}}[1/p],\bbZ_{p}(n))[2n]
\end{equation*} 
in $\Dh(\bbZ_{p})$ for $R\in\CAlg^{\an}$. \\
(5) There is a map 
\begin{equation*}
\Fil^{\sbullet}L_{T(1)}\TC\to \Fil^{\sbullet}(L_{T(1)}\TC)^{\sharp}
\end{equation*} 
of $\Fun(\CAlg^{\an},\CAlg(\DFh(\bbS)))$, where the target $\Fil^{\sbullet}(L_{T(1)}\TC)^{\sharp}$ is from Construction \ref{constr:filtonarcpsheaf}. This map realizes the target as a filtration-completion of the source, and induces the $\arcp$-hypersheafification map 
\begin{equation*}
h:L_{T(1)}\TC\to (L_{T(1)}\TC)^{\sharp}(\pi_{0}((-)^{\wedge_{p}}))
\end{equation*}
of $\Fun(\CAlg^{\an},\CAlg(\Sp^{\Cpl(p)}))$ between the underlying objects of filtrations. \\
(6) For $R\in\CAlg^{\an}$, the filtered object $\Fil^{\sbullet}L_{T(1)}\TC(R)$ of $p$-complete spectra is filtration-complete if and only if the natural map $h:L_{T(1)}\TC(R)\simeq L_{T(1)}\TC(\pi_{0}(R^{\wedge_{p}}))\to (L_{T(1)}\TC)^{\sharp}(\pi_{0}(R^{\wedge_{p}}))$ given by $\arcp$-hypersheafification is an equivalence. 
\begin{proof}
We take $\Fil^{\sbullet}L_{T(1)}\TC$ to be the filtered object of Notation \ref{not:T(1)TCfilt}. (1) follows from Proposition \ref{prop:filtvialke} (1) or the last paragraph of Construction \ref{constr:filtviapb}. (2) follows from Construction \ref{constr:filtvialke}. For (3), note that the functors $L_{T(1)}\TC$, $(L_{T(1)}\TC)^{\sharp}(\pi_{0}((-)^{\wedge_{p}}))$, and $\Fil^{\sbullet}(L_{T(1)}\TC)^{\sharp}$ are invariant under taking $\pi_{0}$ and $(-)^{\wedge_{p}}$ of input animated commutative rings, and that these functors satisfy $p$-complete fppf descent; the latter for $L_{T(1)}\TC$ follows from its \'etale and finite flat descent properties \cite[Th. A.4 and Th. 5.1]{ausrog} as in the proof of \cite[Th. 6.9]{etalek}. Being a pullback of these functors, $\Fil^{\sbullet}L_{T(1)}\TC$ also satisfies the claimed properties. (4) follows from Proposition \ref{prop:filtvialke} (2) or from Lemma \ref{lem:filtpullback} (3) applied to $h$ and $\Fil^{\sbullet}(L_{T(1)}\TC)^{\sharp}$. (5) follows from Proposition \ref{prop:filtcompletecomparison} (2) and Remark \ref{rem:filtundobj}. Finally, (6) follows from Lemma \ref{lem:filtpullback} (1) applied to $h$ and $\Fil^{\sbullet}(L_{T(1)}\TC)^{\sharp}$. 
\end{proof}
\end{theorem}

\subsection{Comparison with the Thomason filtration on $T(1)$-local algebraic $\K$-theory} \label{subsec:compwiththomason}

In this subsection, we provide a comparison of $\Fil^{\sbullet}L_{T(1)}\TC$ with the Thomason filtration on $T(1)$-local algebraic K-theory and explain how this entails a K-theoretic proof of the \'etale comparison for prismatic cohomology. Consider the functor $F := \tau_{\geq 2\sbullet}L_{T(1)}\K:\CAlg^{\an}\to\CAlg(\DFh(\bbS))$ of presheaf Postnikov filtration on $L_{T(1)}\K$, cf. Proposition \ref{prop:postnikovmult}. By taking \'etale sheafification $F^{\et}$ and inverting $p$ on the source, we have $F':=F^{\et}(-[\frac{1}{p}])$ and a natural map
\begin{equation}\label{eq:FtoF'}
F\to F^{\et}\to F^{\et}(-[1/p]) = F'
\end{equation} 
of presheaves of multiplicative filtered $p$-complete spectra on $\CAlg^{\an}$. Note that this is a map over $L_{T(1)}\K$, which is an \'etale sheaf, cf., \cite[Th. A.4]{ausrog} or more generally \cite[Th. 5.39]{etalek}, and is invariant under inverting $p$ on the source \cite[Th. 2.16]{bcm}. By definition, the underlying filtered object for $F^{\et}$ in $\Fun(\CAlg^{\an},\DFh(\bbS))$ is given by the \'etale Postinikov filtration $\tau_{\geq 2\sbullet}^{\et}L_{T(1)}\K$, cf. Remark \ref{rem:calgsheaf} and Remark \ref{rem:filtsheaf}. Similarly, we also consider the functor $G:=\tau_{\geq 2\sbullet-1}L_{T(1)}\K:\CAlg^{\an}\to \DFh(\bbS)$, its \'etale sheafification $G^{\et}$, and a natural map $G\to G^{\et}\to G^{\et}(-[1/p])=:G'$ of presheaves of multiplicative filtered $p$-complete spectra on $\CAlg^{\an}$. 

\begin{remarkn} \label{rem:thomasonfilt}
For $R\in\CAlg^{\an}$, the filtered object $F'(R)$ gives the Thomason filtration for $L_{T(1)}\K(R[\frac{1}{p}])\simeq L_{T(1)}\K(R)$. In fact, we have an equivalence
\begin{equation*}
\begin{tikzcd}
L_{(\Spec R[1/p])_{\et}}\left(\tau_{\geq2\sbullet}L_{T(1)}\K|_{(\Spec R[1/p])_{\et}}\right)(R[1/p]) \arrow[r, "\sim"] & (\tau_{\geq2\sbullet}L_{T(1)}\K)^{\et}|_{(\Spec R[1/p])}(R[1/p]) = F'(R),
\end{tikzcd}
\end{equation*}
where $L_{(\Spec R[1/p])_{\et}}$ is an \'etale sheafification functor for presheaves on $(\Spec R[1/p])_{\et}$, cf. \cite[Prop. 7.6]{etalek}. The left hand side of the equivalence gives the Thomason filtration, which in particular gives an exhaustive filtration on $L_{T(1)}\K(R[1/p])\simeq L_{T(1)}\K(R)$ and is filtration-complete under the presence of the finiteness condition on $R[1/p]$ as in \cite[Th. 7.14]{etalek}. As identified by Thomason \cite{tt}, the $n$-th associated graded piece of the Thomason filtration is naturally equivalent to $\RG_{\et}(\Spec R[1/p],\bbZ_{p}(n))[2n]$ for $R\in\CAlg^{\an}$, and $F' = \tau_{\geq 2\sbullet}^{\et}L_{T(1)}\K(-[1/p])\simeq \tau_{\geq 2\sbullet-1}^{\et}L_{T(1)}\K(-[1/p]) = G'$. 
\end{remarkn}

\begin{proposition}\label{prop:compwiththomason}
The maps $F\to F'$ of (\ref{eq:FtoF'}) and $G\to G'$ induce equivalences after taking $\arcp$-hypersheafifications on $\CAlg^{p-\cpl}$. 
\begin{proof}
Since the case of $G = \tau_{\geq 2\sbullet -1}L_{T(1)}\K$ follows in the same way as the case of $F$, let us explain the latter. To check the induced map $(F|_{\CAlg^{p-\cpl}})^{\sharp}\to (F'|_{\CAlg^{p-\cpl}})^{\sharp}$ of hypersheaves is an equivalence, it suffices to check that the map $F|_{\CAlg^{p-\cpl}}\to F'|_{\CAlg^{p-\cpl}}$ gives an equivalence $F(\prod_{T}V_{t})\simeq F'(\prod_{T}V_{t})$ for any products $\prod_{T}V_{t}\in\CAlg^{p-\cpl}$ of rank 1 absolutely integrally closed $p$-complete valuation rings $V_{t}$, as these form a basis for the $\arcp$-topology on $\CAlg^{p-\cpl}$, cf. \cite[Prop. 3.30]{arc}. \\
\indent Now, observe that both $F$ and $F'$ on $\CAlg^{\an}$ commute with filtered colimits; the case of $F$ is obvious from the definition and the analogous property for algebraic $\K$-theory. For the case of $F'$, it suffices to check that the \'etale sheafification $F^{\et} = \tau_{\geq2\sbullet}^{\et}L_{T(1)}\K$ of $F$ commutes with filtered colimits. For each $m\in\bbZ$, consider the functor $R\mapsto \tau_{\leq m}L_{T(1)}\K(R):\CAlg^{\an}\to\Sp^{\Cpl(p)}$, which commutes with filtered colimits and is $m$-truncated as a presheaf. By \cite[Prop. 7.11]{etalek}, we know its \'etale sheafification $\tau_{\leq m}^{\et}L_{T(1)}\K$ commutes with filtered colimits. Combining this with the fact that $L_{T(1)}\K$ commutes with filtered colimits, we know each $\tau_{\geq 2n}^{\et}L_{T(1)}\K$ for $n\in\bbZ$, and hence $F^{\et}$ commutes with filtered colimits. \\
\indent By \cite[Cor. 3.18]{arc}, the preceding discussion further reduces us to the task of verifying that the map $F(V)\to F'(V)$ is an equivalence for ultraproducts $V$ of $V_{t}$'s; these rings $V$ are in particular absolutely integrally closed $p$-Henselian valuation rings. By construction, the map $F(V)\to F'(V)$ factors as $F(V)\simeq F(V[1/p])\to F^{\et}(V[1/p]) = F'(V)$, and we have to check that the \'etale sheafification map $F(V[1/p])\to F^{\et}(V[1/p])$ on $V[1/p]$ is an equivalence. However, note that the localization $V[1/p]$ of $V$ remains to be an absolutely integrally closed valuation ring, and hence is in particular a strictly Henselian local ring \cite[Lem. 5.3]{arc}. Since both $F$ and $F^{\et}$ commute with filtered colimits as observed above, we know the map $F(V[1/p])\to F^{\et}(V[1/p])$ is a map between stalks induced from an \'etale sheafification map, and hence is an equivalence of filtered $p$-complete spectra.
\end{proof}
\end{proposition}

Below, let us write $\Fil^{\sbullet}L_{T(1)}\K$ for the Thomason filtration on $L_{T(1)}\K$; using our notation from Remark \ref{rem:thomasonfilt}, we can write $F'(R) = \Fil^{\sbullet}L_{T(1)}\K(R[1/p])$ for $R\in\CAlg^{\an}$. 

\begin{theorem}[Comparison with the Thomason filtration]\label{th:compwiththomason}
There is a natural diagram 
\begin{equation*}
\begin{tikzcd}
(\Fil^{\sbullet}L_{T(1)}\K)(R^{\wedge_{p}}[1/p]) \arrow[r, "\sim"'] \arrow[d] & (\Fil^{\sbullet}L_{T(1)}\TC)(R) \arrow[d]  \\
L_{T(1)}\K(R^{\wedge_{p}}[1/p])\simeq L_{T(1)}\K(R^{\wedge_{p}}) \arrow[r, "\tr"', "\sim"] & L_{T(1)}\TC(R) 
\end{tikzcd}
\end{equation*}
of multiplicative filtered $p$-complete spectra for $R\in\CAlg^{\an}$, where the upper horizontal map is an equivalence of multiplicative filtered $p$-complete spectra which induces the equivalence $L_{T(1)}\K(R^{\wedge_{p}}[1/p])\simeq L_{T(1)}\K(R^{\wedge_{p}}) \underset{\sim}{\xrightarrow{\tr}} L_{T(1)}\TC(R)$ of underlying objects in $\CAlg(\Sp^{\Cpl(p)})$, given by the cyclotomic trace map, as the lower horizontal map between constant filtered objects. 
\begin{proof}
Consider the following diagram 
\begin{equation}\label{eq:thomason1}
\begin{tikzcd}
F'|_{\CAlg^{p-\cpl}} \arrow[r] \arrow[d] & (F'|_{\CAlg^{p-\cpl}})^{\sharp} \arrow[d] & (F|_{\CAlg^{p-\cpl}})^{\sharp} \arrow[l, "\sim", "\text{Prop.} \ref{prop:compwiththomason}"'] \arrow[d] \arrow[r, "(\tau_{\geq 2\sbullet-1}\tr)^{\sharp}", "\sim"'] & \Fil^{\sbullet}(L_{T(1)}\TC)^{\sharp} \arrow[d] \\
L_{T(1)}\K|_{\CAlg^{p-\cpl}} \arrow[r] & (L_{T(1)}\K|_{\CAlg^{p-\cpl}})^{\sharp} & (L_{T(1)}\K|_{\CAlg^{p-\cpl}})^{\sharp} \arrow[l, "="] \arrow[r, "\sim", "\tr^{\sharp}"'] & (L_{T(1)}\TC)^{\sharp}
\end{tikzcd}
\end{equation}
of $\Fun(\CAlg^{p-\cpl},\CAlg(\DFh(\bbS)))$. The left square is induced from an $\arcp$-hypersheafification of the natural map $F'\to L_{T(1)}\K$ over $\CAlg^{p-\cpl}$, while the middle square is obtained from the fact that the map $F\to F'$ is defined over $L_{T(1)}\K$. The right square is induced from taking an $\arcp$-hypersheafification of the equivalence $\tr:L_{T(1)}\K|_{\CAlg^{p-\cpl}}\simeq L_{T(1)}\TC|_{\CAlg^{p-\cpl}}$ and its truncations; more precisely, the upper horizontal arrow of the square is the composition $(F|_{\CAlg^{p-\cpl}})^{\sharp}\simeq (G|_{\CAlg^{p-\cpl}})^{\sharp}\overset{\tau^{\sharp}_{\geq 2\sbullet-1}\tr}{\simeq}\Fil^{\sbullet}(L_{T(1)}\TC)^{\sharp}$. Here, the natural map $(F|_{\CAlg^{p-\cpl}})^{\sharp}\to (G|_{\CAlg^{p-\cpl}})^{\sharp}$ is an equivalence through Proposition \ref{prop:compwiththomason} and the fact that the natural map $F'\to G'$ is an equivalence. To check that this map is in $\Fun(\CAlg^{p-\cpl},\CAlg(\DFh(\bbS)))$, note that it is equivalent to the composition $\tau^{\sharp}_{\geq 2\sbullet}(L_{T(1)}\K|_{\CAlg^{p-\cpl}})^{\sharp}\overset{\tau^{\sharp}_{\geq 2\sbullet}\tr}{\simeq} \tau^{\sharp}_{\geq 2\sbullet}(L_{T(1)}\TC)^{\sharp}\to \tau^{\sharp}_{\geq 2\sbullet-1}(L_{T(1)}\TC)^{\sharp} = \Fil^{\sbullet}(L_{T(1)}\TC)^{\sharp}$, and hence it suffices to check that the second map $\tau^{\sharp}_{\geq 2\sbullet}(L_{T(1)}\TC)^{\sharp}\to \tau^{\sharp}_{\geq 2\sbullet-1}(L_{T(1)}\TC)^{\sharp}= \Fil^{\sbullet}(L_{T(1)}\TC)^{\sharp}$ of the composition is in $\Fun(\CAlg^{p-\cpl},\CAlg(\DFh(\bbS)))$. In fact, this map is in $\Shv^{\hyp}_{\CAlg(\DFh(\bbS))}((\CAlg^{p-\cpl})^{\op}_{\arcp})$, since the corresponding map of $\arcp$-hypersheaves $\tau_{\geq 2\sbullet}L_{T(1)}\TC|_{\CAlg^{\qrsp}_{\bbZ_{p},\mathcal{O}_{C}}}\to \Fil^{\sbullet}L_{T(1)}\TC|_{\CAlg^{\qrsp}_{\bbZ_{p},\mathcal{O}_{C}}}$, being induced from the map $\tau_{\geq 2\sbullet}L_{T(1)}\TC|_{\CAlg^{\qrsp}_{\bbZ_{p},\mathcal{O}_{C}}}\to \tau_{\geq 2\sbullet}L_{T(1)}\TP|_{\CAlg^{\qrsp}_{\bbZ_{p},\mathcal{O}_{C}}}$ of $\Fun(\CAlg^{\qrsp}_{\bbZ_{p},\mathcal{O}_{C}},\CAlg(\DFh(\bbS)))$, is in $\Shv^{\hyp}_{\CAlg(\DFh(\bbS))}((\CAlg^{\qrsp}_{\bbZ_{p},\mathcal{O}_{C}})^{\op}_{\arcp})$.  \\
\indent By Proposition \ref{prop:compwiththomason}, we can take the outer square of (\ref{eq:thomason1}) as a well-defined diagram, and by the description of $\Fil^{\sbullet}L_{T(1)}\TC\simeq \Fil_{\mathrm{pb}}^{\sbullet}L_{T(1)}\TC$ as a pullback filtration of Construction \ref{constr:filtviapb}, we have the following diagram
\begin{equation}\label{eq:thomason2}
\begin{tikzcd}
F'|_{\CAlg^{p-\cpl}} \arrow[r] \arrow[d] & (\Fil_{\mathrm{pb}}^{\sbullet}L_{T(1)}\TC)|_{\CAlg^{p-\cpl}} \arrow[r] \arrow[d] & \Fil^{\sbullet}(L_{T(1)}\TC)^{\sharp} \arrow[d] \\
L_{T(1)}\K|_{\CAlg^{p-\cpl}} \arrow[r, "\tr"', "\sim"] & L_{T(1)}\TC|_{\CAlg^{p-\cpl}} \arrow[r, "h"'] & (L_{T(1)}\TC)^{\sharp}
\end{tikzcd}
\end{equation}
in $\Fun(\CAlg^{p-\cpl},\CAlg(\DFh(\bbS)))$ whose outer square agrees with the outer square of the previous diagram. Since $L_{T(1)}\K$ and $L_{T(1)}\TC$ are truncating on $T(1)$-acyclic spectra \cite[Cor. 4.22]{lmmt} and the filtered objects $\Fil^{\sbullet}L_{T(1)}\TC$ and $\Fil^{\sbullet}(L_{T(1)}\TC)^{\sharp}$ are truncating on $\CAlg^{\an}$, we have a diagram 
\begin{equation}\label{eq:thomason3}
\begin{tikzcd}
F'((-)^{\wedge_{p}}) \arrow[r] \arrow[d] & \Fil_{\mathrm{pb}}^{\sbullet}L_{T(1)}\TC \arrow[r] \arrow[d] & \Fil^{\sbullet}(L_{T(1)}\TC)^{\sharp} \arrow[d] \\
L_{T(1)}\K((-)^{\wedge_{p}}) \arrow[r, "\tr"', "\sim"] & L_{T(1)}\TC \arrow[r, "h"'] & (L_{T(1)}\TC)^{\sharp}(\pi_{0}((-)^{\wedge_{p}}))
\end{tikzcd}
\end{equation}
in $\Fun(\CAlg^{\an},\CAlg(\DFh(\bbS)))$ by construction of the involved objects and the previous diagram. \\
\indent We check that the map $F'((-)^{\wedge_{p}}) \to \Fil_{\mathrm{pb}}^{\sbullet}L_{T(1)}\TC$ from the top left horizontal arrow of (\ref{eq:thomason3}) induces an equivalence on associated graded objects. Consider the natural map $F'|_{\CAlg^{p-\cpl}}\to (F'|_{\CAlg^{p-\cpl}})^{\sharp}$ from the top left horizontal arrow of (\ref{eq:thomason1}). It induces an equivalence of associated graded objects, due to the fact that $\gr^{n}(F'|_{\CAlg^{p-\cpl}})\simeq \RG_{\et}(\Spec(-)[1/p],\bbZ_{p}(n))[2n]$ for each $n\in\bbZ$ is already an $\arcp$-hypersheaf valued in $\Dh(\bbZ_{p})$. Since the target is a complete filtered object (through the identification with $\Fil^{\sbullet}(L_{T(1)}\TC)^{\sharp}$ using Proposition \ref{prop:compwiththomason} and the cyclotomic trace map), we know that this map realizes the target as a filtration-completion of the source by Lemma \ref{lem:filtcompletion} (2). Of course, the same can be said for the composite top horizontal map $F'|_{\CAlg^{p-\cpl}}\to\Fil^{\sbullet}(L_{T(1)}\TC)^{\sharp}$ of (\ref{eq:thomason1}) and (\ref{eq:thomason2}), equivalently as well as for $F'((-)^{\wedge_{p}})\to\Fil^{\sbullet}(L_{T(1)}\TC)^{\sharp}$ of (\ref{eq:thomason3}). This also implies that the induced map $F'((-)^{\wedge_{p}})\to \Fil_{\mathrm{pb}}^{\sbullet}L_{T(1)}\TC$, again using the pullback description of the target, gives an equivalence after taking associated graded objects functor.  \\
\indent Now, note that both $F'\to L_{T(1)}\K(-[1/p])\simeq L_{T(1)}\K$ and $\Fil^{\sbullet}_{\mathrm{pb}}L_{T(1)}\TC\to L_{T(1)}\TC$ realize the source as an exhaustive filtration of the target; the former is an \'etale Postnikov filtration, while the latter is a filtered object pulled-back from an $\arcp$-Postnikov filtration, cf. Theorem \ref{th:T(1)TCfilt} (1). By Lemma \ref{lem:filtexhaustivegr}, we know the natural map $F'((-)^{\wedge_{p}})\to \Fil^{\sbullet}_{\mathrm{pb}}L_{T(1)}\TC$ is an equivalence from the preceding paragraph. 
\end{proof}
\end{theorem}

In the course of the proof of Theorem \ref{th:compwiththomason} given above, we in particular observed the following fact, which can be regarded as a topological or K-theoretic proof of the \'etale comparison. Note that we already stated this comparison result and gave an algebraic proof in Proposition \ref{prop:etcomp}; the proof we give here is logically independent of Proposition \ref{prop:etcomp}, cf. Remark \ref{rem:etcomp}. 

\begin{corollary}[\'Etale comparison via cyclotomic trace map]\label{cor:etcomp}
The map  
\begin{equation*}
\Fil^{\sbullet}L_{T(1)}\K((-)^{\wedge_{p}}[1/p]) \xrightarrow{\sim} \Fil^{\sbullet}L_{T(1)}\TC
\end{equation*} 
of $\Fun(\CAlg^{\an},\CAlg(\DFh(\bbS)))$ from Theorem \ref{th:compwiththomason} induces an equivalence 
\begin{equation*}
\gr^{\sbullet}L_{T(1)}\K((-)^{\wedge_{p}}[1/p]) \xrightarrow{\sim} \gr^{\sbullet}L_{T(1)}\TC
\end{equation*} 
of $\Fun(\CAlg^{\an},\Dh(\bbZ_{p})^{\gr})$ after taking associated graded objects functor. In particular, for each $n\in\bbZ$ and a perfect prism $(A,(d))$, we have a natural equivalence
\begin{equation}\label{eq:etcomp2}
\RG_{\et}(\Spec R^{\wedge_{p}}[1/p],\bbZ_{p}(n)) \xrightarrow{\sim} \left(\prism_{R/A}\{n\}[1/d]\right)^{F=1}
\end{equation}
in $\Dh(\bbZ_{p})$ for $R\in\CAlg^{\an}_{\overline{A}}$. 
\begin{proof}
By Theorem \ref{th:compwiththomason}, we have an equivalence between the given filtered objects, and hence the equivalence between the associated graded objects. Here, we take $\Fil^{\sbullet}L_{T(1)}\TC$ as $\Fil^{\sbullet}_{\mathrm{pb}}L_{T(1)}\TC$ of Construction \ref{constr:filtviapb} as in the proof of Theorem \ref{th:compwiththomason}. Note that the aforementioned proof and in particular the proof of the equivalence $\gr^{\sbullet}L_{T(1)}\K((-)^{\wedge_{p}}[1/p]) \simeq \gr^{\sbullet}L_{T(1)}\TC$ given in the paragraph below the diagram (\ref{eq:thomason3}) does not use the results of subsection \ref{subsec:filtconstr2} and hence is independent of Proposition \ref{prop:etcomp}. Now, the equivalence 
\begin{equation*}
\gr^{n}L_{T(1)}\K(R^{\wedge_{p}}[1/p])\simeq \RG_{\et}(\Spec R^{\wedge_{p}}[1/p],\bbZ_{p}(n))[2n]
\end{equation*} 
follows from Thomason's description \cite{tt}, while the equivalence 
\begin{equation*}
\gr^{n}L_{T(1)}\TC(R)\simeq \left(\prism_{R/A}\{n\}[1/d]\right)^{F=1}[2n]
\end{equation*} 
follows from $\gr^{\sbullet}\Fil_{\mathrm{pb}}L_{T(1)}\TC\simeq \gr^{\sbullet}(L_{T(1)}\TC)^{\sharp}$ of Construction \ref{constr:filtviapb} and the description of the target in Construction \ref{constr:filtonarcpsheaf}.
\end{proof}
\end{corollary}

In fact, the two comparison maps we obtained from Corollary \ref{cor:etcomp} and Proposition \ref{prop:etcomp} agree with each other:

\begin{remarkn}[Description of (\ref{eq:etcomp2}) in terms of the prismatic logarithm] \label{rem:etcompmaps}
The comparison map 
\begin{equation*}
\gr^{n}(\tr)[-2n]:\gr^{n}L_{T(1)}\K(R^{\wedge_{p}}[1/p])[-2n]\simeq \gr^{n}L_{T(1)}\K(R^{\wedge_{p}})[-2n]\to \gr^{n}L_{T(1)}\TC(R)[-2n]
\end{equation*} 
of Corollary \ref{cor:etcomp}, (\ref{eq:etcomp2}) is equivalent to the comparison map 
\begin{equation*}
\log_{\prism}^{\otimes n}:\RG_{\et}(\Spec R^{\wedge_{p}}[1/p],\bbZ_{p}(n))\to \left(\prism_{R/A}\{n\}[1/d]\right)^{F=1}
\end{equation*}
of Proposition \ref{prop:etcomp}, (\ref{eq:etcomplog3}). To check this, it suffices to compute $\gr^{n}(\tr)[-2n]$ for $R\in\CAlg^{p-\cpl}_{\overline{A}}$. As both the source and the target of the map are $\arcp$-hypersheaves on $\CAlg^{p-\cpl}_{\overline{A}}$ and preserve filtered colimits on $\CAlg^{\an}_{\overline{A}}$, we are further reduced to the case of absolutely integrally closed $p$-Henselian valuation rings $V$ over $\overline{A}$ which admits a map from $\mathcal{O}_{\bbC_{p}}$. Moreover, it suffices to check the case of $n=1$. The cyclotomic trace map is multiplicative and $(T_{p}(V^{\wedge_{p}})^{\times})^{\otimes_{\bbZ_{p}} n}[0]\simeq \RG_{\et}(\Spec V^{\wedge_{p}}[1/p],\bbZ_{p}(n)) = \gr^{n}L_{T(1)}\K(V^{\wedge_{p}}[1/p])[-2n]$ for $n\in\bbZ$ are invertible $\bbZ_{p}$-modules which satisfy natural equivalences $(\gr^{1}L_{T(1)}\K(V^{\wedge_{p}}[1/p])[-2])^{\otimes_{\bbZ_{p}}n}\simeq \gr^{n}L_{T(1)}\K(V^{\wedge_{p}}[1/p])[-2n]$; in particular, $\gr^{n}(\tr)[-2n]\simeq (\gr^{1}(\tr)[-2])^{\otimes_{\bbZ_{p}} n}$. By construction, $\log_{\prism}^{\otimes n}$ has the analogous property on $V^{\wedge_{p}}$. For $\gr^{1}(\tr)[-2]$ on $V$, we have $\gr^{1}(\tr)[-2]\simeq \pi_{2}(\tr):\pi_{2}(\K_{\geq0}(V^{\wedge}_{p},\bbZ_{p})[\beta^{-1}])\to \pi_{2}(\TC(V^{\wedge_{p}},\bbZ_{p})[\beta^{-1}])$. By \cite[Th. 1.2]{al} (and \cite[Prop. 2.6.10]{bl}), we know $\pi_{2}(\tr)$ for $\tr$ before $p$-completely inverting $\beta$ is given by the prismatic logarithm, and hence after $p$-completely inverting $d$ and factoring through the Frobenius fixed point induces the map $\log_{\prism}^{\otimes 1}$. 
\end{remarkn}

We have the following consequence concerning the filtration-completeness of $\Fil^{\sbullet}L_{T(1)}\TC(R)$:

\begin{corollary}\label{cor:compwiththomason}
For $R\in\CAlg^{\an}$ such that $\pi_{0}R^{\wedge_{p}}[1/p]$ has finite Krull dimension and admits a uniform bound on the mod $p$ virtual cohomological dimensions of the residue fields, the filtered object $(\Fil^{\sbullet}L_{T(1)}\TC)(R)$ gives a filtration-complete and exhaustive filtration on $L_{T(1)}\TC(R)$. Equivalently, the natural map $h:L_{T(1)}\TC(R) \to (L_{T(1)}\TC)^{\sharp}(\pi_{0}(R^{\wedge_{p}}))$ is an equivalence. 
\begin{proof}
Under the finiteness condition on $\pi_{0}R^{\wedge_{p}}[1/p]$ as indicated, $L_{T(1)}\K$ on $\Spec(R^{\wedge_{p}}[1/p])_{\et}$ is \'etale Postnikov complete, and the map $F'(R^{\wedge_{p}})\to L_{T(1)}\K(R^{\wedge_{p}})$ realizes the source as a complete and exhaustive filtration of the target \cite[Th. 7.14]{etalek}. From the natural equivalence $F'(R^{\wedge_{p}})\simeq \Fil^{\sbullet}L_{T(1)}\TC(R)$ of Theorem \ref{th:compwiththomason}, we know the filtration-completeness. Equivalently, we also know that the map $L_{T(1)}\TC(R)\to (L_{T(1)}\TC)^{\sharp}(\pi_{0}(R^{\wedge_{p}}))$ is an equivalence by Theorem \ref{th:T(1)TCfilt} (6). 
\end{proof}
\end{corollary}

\begin{remarkn}
The construction of \cite{matTR}, cf. Remark \ref{rem:mathewfilt1}, defines a natural filtration on $T(1)$-local $\TR$ for each formally smooth algebra $R$ over the ring $\mathcal{O}_{K}$ of integers of $K$, where $K$ is a complete nonarchimedean field of mixed characteristic $(0,p)$ which is either discretely valued or perfectoid \cite[Th. 1.8]{matTR}. By \cite[Rem. 7.28]{matTR}, the natural filtration on $T(1)$-local $\TC$ of these $p$-complete commutative rings obtained by taking Frobenius fixed points recovers the Thomason filtration, and in particular agrees with our filtration on $T(1)$-local $\TC$.  
\end{remarkn}

\appendix

\section{Preliminaries on filtered objects} \label{sec:appendix}

In this appendix, we record some generalities on filtered objects of presentable stable $\infty$-categories. In particular, we discuss completion and pullback constructions for filtered objects, sheafifications, and the multiplicativity of Postnikov filtrations. \\

\noindent Let $\mathcal{D}$ be a presentable stable $\infty$-category. Let us use the notations $\mathcal{D}^{\mathrm{fil}} = \Fun(\N\bbZ^{\op},\mathcal{D})$ and $\mathcal{D}^{\gr} = \Fun(\N\bbZ^{\mathrm{disc}},\mathcal{D})$ for the categories of filtered objects and graded objects in $\mathcal{D}$ respectively, and $\gr:\mathcal{D}^{\mathrm{fil}}\to\mathcal{D}^{\gr} = (X^{\geq\sbullet}\mapsto \gr^{\sbullet} X)$ for the associated graded objects functor. Also, let us denote $\mathcal{D}^{\mathrm{fil, \wedge}}$ for the full subcategory of $\mathcal{D}^{\mathrm{fil}}$ spanned by filtration-complete filtered objects. The following construction is well-known:

\begin{construction}[Filtration-completion] \label{constr:filtcompletion}
Let $\mathcal{D}$ be a presentable stable $\infty$-category. Consider the map $p:\N\bbZ^{\op}\to\ast$ of simplicial sets and the induced adjunction $p^{\ast}\dashv p_{\ast}:\mathcal{D}^{\mathrm{fil}}\to\mathcal{D}$. Here, $p^{\ast}$ is a constant filtration functor, while for each object $X^{\geq\sbullet}\in\mathcal{D}^{\mathrm{fil}}$, the object $p_{\ast}X^{\geq\sbullet}$ is $\lim X = \lim_{n\to\infty} X^{\geq n}$, a right Kan extension of $X^{\geq\sbullet}$ along $p$. We write $X^{\geq\sbullet}/\lim X$ for a cofiber of the counit map $\lim X\simeq p^{\ast} p_{\ast} X^{\geq\sbullet}\to X^{\geq\sbullet}$ of the adjunction. By construction $id\simeq p_{\ast}p^{\ast}$, and hence the object $X^{\geq\sbullet}/\lim X$ is filtration-complete. 
\end{construction}

\begin{lemma}\label{lem:filtcompletion}
Let $\mathcal{D}$ be a presentable stable $\infty$-category. Then, the followings hold:\\
(1) The inclusion $\mathcal{D}^{\mathrm{fil},\wedge}\hookrightarrow\mathcal{D}^{\mathrm{fil}}$ admits a left adjoint, which we call a filtration-completion functor. In fact, the functor $(X^{\geq\sbullet}\mapsto X^{\geq\sbullet}/\lim X):\mathcal{D}^{\mathrm{fil}}\to\mathcal{D}^{\mathrm{fil},\wedge}$ gives a left adjoint for the inclusion.  \\
(2) Let $X^{\geq \sbullet}\to Y^{\geq\sbullet}$ be a map in $\mathcal{D}^{\mathrm{fil}}$ such that the target $Y^{\geq\sbullet}$ is filtration-complete and that the map induces an equivalence of associated graded objects $\gr^{\sbullet} X\simeq \gr^{\sbullet} Y$ in $\mathcal{D}^{\gr}$. Then, the map realizes $Y^{\geq\sbullet}$ as a filtration-completion of $X^{\geq\sbullet}$.
\begin{proof}
For (1), we have to check that for any filtration-complete filtered object $Z^{\geq\sbullet}\in\mathcal{D}^{\mathrm{fil},\wedge}$, the composition map $-\circ(X^{\geq\sbullet}\to X^{\geq\sbullet}/\lim X):\Map_{\mathcal{D}^{\mathrm{fil,\wedge}}}(X^{\geq\sbullet}/\lim X, Z^{\geq\sbullet})\to \Map_{\mathcal{D}^{\mathrm{fil}}}(X^{\geq\sbullet}, Z^{\geq\sbullet})$ is an equivalence of spaces. By Construction \ref{constr:filtcompletion}, this map factors through the equivalence 
\begin{equation*}
\Map_{\mathcal{D}^{\mathrm{fil},\wedge}}(X^{\geq\sbullet}/\lim X, Z^{\geq\sbullet}) \simeq \fib\left(\Map_{\mathcal{D}^{\mathrm{fil}}}(X^{\geq\sbullet}, Z^{\geq\sbullet})\to \Map_{\mathcal{D}^{\mathrm{fil}}}(p^{\ast} p_{\ast} X^{\geq\sbullet}, Z^{\geq\sbullet})\right). 
\end{equation*}
Since the target of the map which we are taking a fiber can be further described as 
\begin{equation*}
\Map_{\mathcal{D}^{\mathrm{fil}}}(p^{\ast} p_{\ast} X^{\geq\sbullet}, Z^{\geq\sbullet})\simeq \Map_{\mathcal{D}}(p_{\ast} X^{\geq\sbullet}, p_{\ast}Z^{\geq\sbullet})\simeq \ast
\end{equation*} 
due to the assumption $p_{\ast}Z^{\geq\sbullet}\simeq 0$, we know the composition map is indeed an equivalence. \\
\indent For (2), note that taking the filtration-completion functor on the given map induces a map $X^{\geq\sbullet}/\lim X\to Y^{\geq\sbullet}$ in $\mathcal{D}^{\mathrm{fil},\wedge}$. By (1), it suffices to check that this map is an equivalence. Since the associated graded objects functor $\gr$ restricted to $\mathcal{D}^{\mathrm{fil},\wedge}$ reflects equivalences (in fact is fully faithful), one is reduced to checking that the induced map $\gr^{\sbullet}(X^{\geq\sbullet}/\lim X)\to\gr^{\sbullet}Y$ is an equivalence. The latter holds due to the assumption that the given map induces an equivalence $\gr^{\sbullet} X\simeq \gr^{\sbullet} Y$, combined with the fact that the natural map $\gr^{\sbullet} X\simeq \gr^{\sbullet}(X^{\geq\sbullet}/\lim X)$ is an equivalence by construction. 
\end{proof}
\end{lemma}

We will consider the situation where the map $Y^{\geq\sbullet}\to p^{\ast}(y)$ realizes the source as an exhaustive filtration on $y$. As in Construction \ref{constr:filtcompletion}, let $p:\N\bbZ^{\op}\to\ast$ be the natural map of simplicial sets and consider the induced adjunction $p_{!}\dashv p^{\ast}:\mathcal{D}\to \mathcal{D}^{\mathrm{fil}}$, where for each object $Y^{\geq\sbullet}\in\mathcal{D}^{\mathrm{fil}}$, the object $p_{!}Y^{\geq\sbullet}$ is $\colim Y = \colim_{n\to-\infty} Y^{\geq n}$, a left Kan extension of $Y^{\geq\sbullet}$ along $p$. Recall that we say the map $Y^{\geq\sbullet}\to p^{\ast}(y)$ realizes $Y^{\geq\sbullet}$ as an exhaustive filtration on $y$ if applying $p_{!}$ to the map induces, through the equivalence $p_{!}p^{\ast}\simeq id$, an equivalence $\colim Y\simeq y$ in $\mathcal{D}$. We have the following statement for maps of exhaustive filtrations with a common underlying object, dual to the situation of Lemma \ref{lem:filtcompletion}:

\begin{lemma}\label{lem:filtexhaustivegr}
Let $\mathcal{D}$ be a presentable stable $\infty$-category. Also, let $f:X^{\geq\sbullet}\to Y^{\geq\sbullet}$ be a map in $\mathcal{D}^{\mathrm{fil}}_{/p^{\ast}(x)}$, i.e, a map of filtrations over a constant filtration $p^{\ast}(x)$ associated with an object $x$ of $\mathcal{D}$. Suppose $X^{\geq\sbullet}\to p^{\ast}(x)$ and $Y^{\geq\sbullet}\to p^{\ast}(x)$ realize the sources as exhaustive filtrations of $x$. Then, $f$ is an equivalence if it induces an equivalence $\gr^{\sbullet}X\simeq \gr^{\sbullet}Y$ in $\mathcal{D}^{\gr}$. 
\begin{proof}
By considering the natural map $x/Y^{\geq-\sbullet}\to x/X^{\geq-\sbullet}$ in $(\mathcal{D}^{\op})^{\mathrm{fil},\wedge}$, this again follows from the fact that the functor $\gr:(\mathcal{D}^{\op})^{\mathrm{fil},\wedge}\to(\mathcal{D}^{\op})^{\gr}$ reflects equivalences.
\end{proof}
\end{lemma}

We also consider the following pullback construction for filtrations. 

\begin{construction}[Pullback filtration] \label{constr:filtpullback}
Let $\mathcal{D}$ be a presentable stable $\infty$-category. Suppose that $f:x\to y$ is a map of $\mathcal{D}$, and that $Y^{\geq\sbullet}\in\mathcal{D}^{\mathrm{fil}}$ is a filtered object equipped with a map of filtered objects $Y^{\geq\sbullet}\to p^{\ast}(y)$ (i.e., a map from $Y^{\geq\sbullet}$ into the constant filtration associated with $y$). We write $f^{\ast}Y^{\geq\sbullet}$ for a pullback $p^{\ast}(x)\times_{p^{\ast}(y)}Y^{\geq\sbullet}$ of $Y^{\geq\sbullet}\to p^{\ast}(y)$ along the map $p^{\ast}(f)$ in $\mathcal{D}^{\mathrm{fil}}$. In particular, it is equipped with a map of filtered objects $f^{\ast}Y^{\geq\sbullet}\to p^{\ast}(x)$. 
\end{construction}

\begin{lemma}\label{lem:filtpullback}
Let $\mathcal{D}$ be a presentable stable $\infty$-category, and let $f:x\to y$ and $Y^{\geq\sbullet}\to p^{\ast}(y)$ be as in Construction \ref{constr:filtpullback}. Then, the followings hold: \\
(1) $\lim f^{\ast}Y^{\geq\sbullet}\simeq x\times_{y}\lim Y$. In particular, if $Y^{\geq\sbullet}$ is filtration-complete, then $f^{\ast}Y^{\geq\sbullet}$ is filtration-complete if and only if $f$ is an equivalence. \\
(2) $\colim f^{\ast}Y^{\geq\sbullet}\simeq x\times_{y}\colim Y$. In particular, if $Y^{\geq\sbullet}\to p^{\ast}(y)$ realizes $Y^{\geq\sbullet}$ as an exhaustive filtration of the object $y$ (i.e., the map induces an equivalence $\colim Y\simeq y$), then $f^{\ast}Y^{\geq\sbullet}\to p^{\ast}(x)$ realizes $f^{\ast}Y^{\geq\sbullet}$ as an exhaustive filtration of the object $x$.\\
(3) $\gr^{\sbullet}f^{\ast}Y^{\geq\sbullet}\simeq \gr^{\sbullet}Y$. 
\begin{proof}
All the statements are straightforward from the fact that the functors $p_{!},p_{\ast}:\mathcal{D}^{\mathrm{fil}}\to\mathcal{D}$ and $\gr:\mathcal{D}^{\mathrm{fil}}\to\mathcal{D}^{\gr}$ are exact. For (1), one has $p_{\ast}(f^{\ast}Y^{\geq\sbullet})\simeq p_{\ast}(p^{\ast}(x))\times_{p_{\ast}(p^{\ast}(y))}p_{\ast}(Y^{\geq\sbullet})\simeq x\times_{y} p_{\ast}(Y^{\geq\sbullet})$, and observes that when $p_{\ast}Y^{\geq\sbullet}\simeq 0$, one has $p_{\ast}(f^{\ast}Y^{\geq\sbullet})\simeq\fib(f)$ to conclude. The statement (2) is checked analogously, replacing $p_{\ast}$ in the previous argument by $p_{!}$. For (3), one has $\gr^{\sbullet}f^{\ast}Y^{\geq\sbullet}\simeq \gr^{\sbullet}p^{\ast}(x)\times_{\gr^{\sbullet}p^{\ast}(y)}\gr^{\sbullet}Y\simeq 0\times_{0}\gr^{\sbullet}Y\simeq\gr^{\sbullet}Y$. 
\end{proof}
\end{lemma}

Let us briefly discuss sheafifications of commutative algebra presheaves and filtered presheaves. 

\begin{remarkn}[Sheafification of commutative algebra presheaves] \label{rem:calgsheaf}
Let $\mathcal{D}$ be a closed symmetric monoidal presentable $\infty$-category, i.e., an object of $\CAlg(\Pr^{\mathrm{L},\otimes})$. In particular, $\CAlg(\mathcal{D}^{\otimes})$ is presentable \cite[Cor. 3.2.3.5]{ha} and the forgetful functor $\CAlg(\mathcal{D}^{\otimes})\to\mathcal{D}$ reflects equivalences and preserves small limits \cite[Lem. 3.2.2.6 and Cor. 3.2.2.5]{ha}. Moreover, suppose that $\mathcal{D}$ is compactly generated. Then, for each $\infty$-topos $\mathcal{X}$, the $\infty$-category of $\mathcal{D}$-valued sheaves $\Shv_{\mathcal{D}}(\mathcal{X})$ is equipped with the pointwise symmetric monoidal structure induced from that of the presheaves $\Fun(\mathcal{X}^{\op},\mathcal{D})$, cf. \cite[Lem. 1.3.4.4 and Proof of Prop. 1.3.4.6]{sag}, and from the forgetful functor we have a natural equivalence 
\begin{equation*}
\Shv_{\CAlg(\mathcal{D}^{\otimes})}(\mathcal{X})\simeq \Fun^{\mathrm{R}}(\mathcal{X}^{\op},\CAlg(\mathcal{D}^{\otimes}))\simeq \CAlg(\Fun^{\mathrm{R}}(\mathcal{X}^{\op},\mathcal{D})^{\otimes})\simeq \CAlg(\Shv_{\mathcal{D}}(\mathcal{X})^{\otimes}).
\end{equation*}
By definition, this equivalence is compatible with the natural equivalence 
\begin{equation*}
\Fun(\mathcal{X}^{\op},\CAlg(\mathcal{D}^{\otimes}))\simeq \CAlg(\Fun(\mathcal{X}^{\op},\mathcal{D})^{\otimes})
\end{equation*} 
through canonical embeddings. Thus, we in particular have a natural diagram
\begin{equation*}
\begin{tikzcd}
\Fun(\mathcal{X}^{\op},\CAlg(\mathcal{D}^{\otimes})) \arrow[r, "\sim"'] \arrow[d, "L"'] & \CAlg(\Fun(\mathcal{X}^{\op},\mathcal{D})^{\otimes}) \arrow[d, "L"]\\
\Shv_{\CAlg(\mathcal{D}^{\otimes})}(\mathcal{X}) \arrow[r, "\sim"] & \CAlg(\Shv_{\mathcal{D}}(\mathcal{X})^{\otimes})
\end{tikzcd}
\end{equation*}
whose vertical arrows are induced from the sheafification functors. 
\end{remarkn}

\begin{remarkn}[Sheafification of filtered presheaves]\label{rem:filtsheaf}
Let $\mathcal{D}$ be a presentable $\infty$-category, and let $\mathcal{X}$ be an $\infty$-topos. We have a natural equivalence 
\begin{equation*}
\Shv_{\Fun(\N\bbZ^{\op},\mathcal{D})}(\mathcal{X})\simeq\Fun^{\mathrm{R}}(\mathcal{X}^{\op},\Fun(\N\mathbb{Z}^{\op},\mathcal{D}))\simeq \Fun(\N\bbZ^{\op},\Fun^{\mathrm{R}}(\mathcal{X}^{\op},\mathcal{D}))\simeq \Fun(\N\bbZ^{\op},\Shv_{\mathcal{D}}(\mathcal{X})).
\end{equation*} 
By definition, this equivalence is compatible with the natural equivalence 
\begin{equation*}
\Fun(\mathcal{X}^{\op},\Fun(\N\mathbb{Z}^{\op},\mathcal{D}))\simeq \Fun(\N\bbZ^{\op},\Fun(\mathcal{X}^{\op},\mathcal{D}))
\end{equation*} 
through canonical embeddings. Thus, we have a natural diagram 
\begin{equation*}
\begin{tikzcd}
\Fun(\mathcal{X}^{\op},\Fun(\N\bbZ^{\op},\mathcal{D})) \arrow[r, "\sim"'] \arrow[d, "L"'] & \Fun(\N\bbZ^{\op},\Fun(\mathcal{X}^{\op},\mathcal{D})) \arrow[d, "L"]\\
\Shv_{\Fun(\N\bbZ^{\op},\mathcal{D})}(\mathcal{X}) \arrow[r, "\sim"] & \Fun(\N\bbZ^{\op},\Shv_{\mathcal{D}}(\mathcal{X}))
\end{tikzcd}
\end{equation*}
whose vertical arrows are induced from the sheafification functors. 
\end{remarkn}

In particular, sheafification of filtration valued (resp. commutative algebra valued) presheaves has underlying sheaves given by levelwise sheafifications (resp. sheafification of underlying presheaves forgetting the commutative algebra structure). Note that analogous statements hold for hypersheafification as well: 

\begin{remarkn}
Let $\mathcal{X}$ be an $\infty$-topos. Recall that its full subcategory $\mathcal{X}^{\hyp}$ of hypercomplete objects gives a geometric morphism $\imath^{\ast}\dashv\imath_{\ast}:\mathcal{X}^{\hyp}\hookrightarrow\mathcal{X}$ of $\infty$-topoi \cite[Rem. 1.3.3.2]{sag}. By applying above constructions in Remark \ref{rem:calgsheaf} and Remark \ref{rem:filtsheaf} to this geometric morphism, we know that the $\infty$-topos $\mathcal{X}$ appearing in the bottom horizontal arrows of the natural diagrams can be replaced by $\mathcal{X}^{\hyp}$. In other words, the vertical arrows in the diagrams of Remark \ref{rem:calgsheaf} and Remark \ref{rem:filtsheaf} can be replaced by the hypersheafification functors. 
\end{remarkn}

Finally, let us turn our attention to Postnikov filtrations. 

\begin{remarkn}[Conventions for Postnikov filtrations]
Let $\mathcal{D}$ be a closed symmetric monoidal presentable stable $\infty$-category equipped with a t-structure which is compatible with the monoidal structure, e.g., $\mathcal{D} = \Sp$. Here, compatibility of the t-structure with the monoidal structure means that $\mathcal{D}_{\geq 0}$ has the monoidal unit of $\mathcal{D}$ and is closed under tensor products. Note that $\mathcal{D}^{\mathrm{fil}} = \Fun(\N\mathbb{Z}^{\op},\mathcal{D})$ is equipped with a symmetric monoidal structure $\mathcal{D}^{\mathrm{fil},\otimes}$ whose underlying tensor product is given by the Day convolution, cf. \cite[Ex. 2.2.6.17]{ha}. Let $k$ be a positive integer, and consider the functor $\tau_{\geq k\sbullet}:\mathcal{D}\to \mathcal{D}^{\mathrm{fil}}$ sending each objects $x$ of $\mathcal{D}$ to its ($k$-speed) Postnikov filtration $\tau_{\geq k\sbullet}x$.
\end{remarkn}
 
\begin{proposition}[Multiplicativity of $k$-speed Postnikov filtrations]\label{prop:postnikovmult}
Let $\mathcal{D}$ be a closed symmetric monoidal presentable stable $\infty$-category equipped with a t-structure which is compatible with the monoidal structure, and let $k$ be a positive integer. Then, the functor $\tau_{\geq k\sbullet}:\mathcal{D}\to \mathcal{D}^{\mathrm{fil}}$ is lax symmetric monoidal, and in particular induces the functor $\tau_{\geq k\sbullet}:\CAlg(\mathcal{D}^{\otimes})\to \CAlg(\mathcal{D}^{\mathrm{fil},\otimes})$. 
\end{proposition}   

The proof requires several steps. Consider the full subcategory $\mathcal{C}$ of $\mathcal{D}^{\mathrm{fil}}$ spanned by objects $X^{\geq\sbullet}\in\mathcal{D}^{\mathrm{fil}}$ such that for each $n\in\bbZ$, the object $X^{\geq n}$ is in $\mathcal{D}_{\geq kn}$, cf. \cite[Proof of Lem. 7.2.1.11]{htt} for an analogue. The functor $\tau_{\geq k\sbullet}$ factors through $\mathcal{C}$, and induces the functor $\tau_{\geq k\sbullet}:\mathcal{D}\to\mathcal{C}$ for which we retain the same notation for convenience.

\begin{lemma}\label{lem:postnikovmult1}
The full subcategory $\mathcal{C}$ admits a natural symmetric monoidal structure making the canonical embedding $\mathcal{C}\hookrightarrow \mathcal{D}^{\mathrm{fil}}$ symmetric monoidal. 
\begin{proof}
It suffices to check that $\mathcal{C}$ in $\mathcal{D}^{\mathrm{fil}}$ has the monoidal unit of $\mathcal{D}^{\mathrm{fil}}$ and is closed under tensor products, cf. \cite[Prop. 2.2.1.1 and Rem. 2.2.1.2]{ha}. The monoidal unit $\mathbf{1}_{\mathcal{D}^{\mathrm{fil}}}$ of $\mathcal{D}^{\mathrm{fil}}$ is described as $(\mathbf{1}_{\mathcal{D}^{\mathrm{fil}}})^{\geq n}\simeq 0$ for $n>0$ and $(\mathbf{1}_{\mathcal{D}^{\mathrm{fil}}})^{\geq n}\simeq\mathbf{1}$ for $n\leq0$ (here, $\mathbf{1}$ denotes the monoidal unit of $\mathcal{D}$) and hence is in $\mathcal{C}$. For objects $X^{\geq\sbullet}$ and $Y^{\geq\sbullet}$ of $\mathcal{C}$ and $n\in\bbZ$, we have $(X^{\geq\sbullet}\otimes Y^{\geq\sbullet})^{\geq n}\simeq \colim_{i+j\geq n}X^{\geq i}\otimes Y^{\geq j}$, and since $X^{\geq i}\otimes Y^{\geq j}\in\mathcal{D}_{\geq k(i+j)}\subseteq\mathcal{D}_{\geq kn}$ for each $i$ and $j$ satisfying $i+j\geq n$, we know $X^{\geq\sbullet}\otimes Y^{\geq\sbullet}\in\mathcal{C}$. 
\end{proof}
\end{lemma}

We again use the notation $p_{!}\dashv p^{\ast}:\mathcal{D}\to \mathcal{D}^{\mathrm{fil}}$ for the adjunction of the functor $p^{\ast}$ induced by the map $p:\N\bbZ^{\op}\to\Delta^{0}$ and the functor of left Kan extensions $p_{!}$ taking underlying objects of filtered objects. For the lemma below, we only require that $\mathcal{D}$ is an object of $\CAlg(\Pr^{\mathrm{L},\otimes}_{\mathrm{st}})$. 

\begin{lemma}\label{lem:postnikovmult2}
Let $\mathcal{D}$ be a closed symmetric monoidal presentable stable $\infty$-category. The adjunction $p_{!}\dashv p^{\ast}:\mathcal{D}\to \mathcal{D}^{\mathrm{fil}}$ induces a symmetric monoidal adjunction. 
\begin{proof}
It suffices to check that the functor $p_{!}:\mathcal{D}^{\mathrm{fil}}\to\mathcal{D}$ refines to a symmetric monoidal functor. For that, we first check the case of $\mathcal{D} = \Sp$ and deduce the general case from that. To check that the functor $p_{!}:\Sp^{\mathrm{fil}} = \Fun(\N\bbZ^{\op},\Sp) \to\Sp$ admits a symmetric monoidal structure, consider the stable Yoneda map $j_{\mathrm{st}}:\N\bbZ\to\Fun(\N\bbZ^{\op},\Sp)$ obtained as the composition of the Yoneda embedding $j:\N\bbZ\to\mathcal{P}(\N\bbZ)$ and the stabilization $\Sigma^{\infty}_{+}\circ -:\mathcal{P}(\N\bbZ)\to \Fun(\N\bbZ^{\op},\Sp)$. Restriction along $j_{\mathrm{st}}$ induces the equivalence $\Fun^{\mathrm{L}}(\Fun(\N\bbZ^{\op},\Sp),\Sp)\simeq \Fun(\N\bbZ,\Sp)$ by the universal properties of each functors composing $j_{\mathrm{st}}$; in fact, this equivalence respects symmetric monoidal structures. More precisely, we have an equivalence
\begin{equation*}
\begin{tikzcd}
\Fun^{\mathrm{L},\otimes}(\Fun(\N\bbZ^{\op},\Sp),\Sp) \arrow[r, "-\circ\Sigma^{\infty}_{+}", "\sim"'] & \Fun^{\mathrm{L},\otimes}(\mathcal{P}(\N\bbZ),\Sp) \arrow[r, "-\circ j", "\sim"'] & \Fun^{\otimes}(\N\bbZ,\Sp),
\end{tikzcd}
\end{equation*} 
where in the above composition, the first map is an equivalence by \cite[Prop. 5.4 (iv)]{ggn}, while the second map is an equivalence by \cite[Cor. 4.8.1.11]{ha}. Thus, in order to check that $p_{!}\in \Fun^{\mathrm{L}}(\Fun(\N\bbZ^{\op},\Sp),\Sp)$ is symmetric monoidal, it suffices to check that the functor $p_{!}\circ j_{\mathrm{st}}\in \Fun(\N\bbZ,\Sp)$ is symmetric monoidal. From the description that $j_{\mathrm{st}}(m)^{\geq n}\simeq 0$ for $n>m$ and $j_{\mathrm{st}}(m)^{\geq n}\simeq \bbS$ for $n\leq m$, we know $p_{!}\circ j_{\mathrm{st}}$ is equivalent to the constant functor valued in $\bbS$, which in particular is naturally symmetric monoidal.\\
\indent Now, we check that the functor $p_{!} = p^{\mathcal{D}}_{!}:\Fun(\N\bbZ^{\op},\mathcal{D})\to\mathcal{D}$ admits a symmetric monoidal structure from the case of $\mathcal{D} = \Sp$ we verified above. In fact, from the symmetric monoidal functors $\Sp\to\mathcal{D}$, $\Fun(\N\bbZ^{\op},\Sp)\to \Fun(\N\bbZ^{\op},\mathcal{D})$, and $\mathbf{1}_{\mathcal{D}^{\mathrm{fil}}}\otimes-:\mathcal{D}\to\Fun(\N\bbZ^{\op},\mathcal{D})$, we obtain symmetric monoidal functors $\Fun(\N\bbZ^{\op},\Sp)\otimes\mathcal{D}\to\Fun(\N\bbZ^{\op},\mathcal{D})$ and $\Sp\otimes\mathcal{D}\to \mathcal{D}$ of $\CAlg(\Pr^{\mathrm{L},\otimes}_{\mathrm{st}})$. Since $p_{!}:\Fun(\N\bbZ^{\op},\Sp)\to\Sp$ is symmetric monoidal, we know the induced map $p_{!}\otimes\mathcal{D}:\Fun(\N\bbZ^{\op},\Sp)\otimes\mathcal{D}\to \Sp\otimes\mathcal{D}$ is also in $\CAlg(\Pr^{\mathrm{L},\otimes}_{\mathrm{st}})$. To relate the underlying functor of $p_{!}\otimes\mathcal{D}$ in $\Pr^{\mathrm{L}}_{\mathrm{st}}$ with the functor $p^{\mathcal{D}}_{!}$ in $\Pr^{\mathrm{L}}_{\mathrm{st}}$, consider the diagram

\begin{equation*}
\begin{tikzcd}
\Fun(\N\bbZ^{\op},\Sp)\times\mathcal{D} \arrow[r, "p_{!}\times\mathcal{D}"] \arrow[d] & \Sp\times\mathcal{D} \arrow[d] \\
\Fun(\N\bbZ^{\op},\mathcal{D}) \arrow[r, "p^{\mathcal{D}}_{!}"'] & \mathcal{D}
\end{tikzcd}
\end{equation*}
in $\Pr^{\mathrm{L}}_{\mathrm{st}}$, where the vertical arrows are induced from the symmetric monoidal structures on $\Fun(\N\bbZ^{\op},\mathcal{D})$ and $\mathcal{D}$ as well as the symmetric monoidal functors $\Sp\to\mathcal{D}$, $\Fun(\N\bbZ^{\op},\Sp)\to \Fun(\N\bbZ^{\op},\mathcal{D})$, and $\mathbf{1}_{\mathcal{D}^{\mathrm{fil}}}\otimes-:\mathcal{D}\to\Fun(\N\bbZ^{\op},\mathcal{D})$. For instance, the left vertical arrow is the composition $\Fun(\N\bbZ^{\op},\Sp)\times\mathcal{D}\xrightarrow{\mathrm{can}\times (\mathbf{1}_{\mathcal{D}^{\mathrm{fil}}}\otimes-)} \Fun(\N\bbZ^{\op},\mathcal{D})\times \Fun(\N\bbZ^{\op},\mathcal{D})\xrightarrow{\otimes}\Fun(\N\bbZ^{\op},\mathcal{D})$, which can be described as sending $(X^{\geq\sbullet},d)$ to $X^{\geq\sbullet}\otimes d$. Note that the vertical arrows commute with colimits by construction, and that the diagram is defined through the preservation of colimits by the action of $\Sp$ on $\mathcal{D}$. Now, the diagram induces an equivalence between the induced functor $p_{!}\otimes\mathcal{D}:\Fun(\N\bbZ^{\op},\Sp)\otimes\mathcal{D}\to \Sp\otimes\mathcal{D}$ in $\Pr^{\mathrm{L}}_{\mathrm{st}}$ and $p^{\mathcal{D}}_{!}$ through the maps $\Fun(\N\bbZ^{\op},\Sp)\otimes\mathcal{D}\to\Fun(\N\bbZ^{\op},\mathcal{D})$ and $\Sp\otimes\mathcal{D}\to \mathcal{D}$ in $\Pr^{\mathrm{L}}_{\mathrm{st}}$ which are equivalences. Since these three functors (other than $p^{\mathcal{D}}_{!}$) are the underlying functors in $\Pr^{\mathrm{L}}_{\mathrm{st}}$ of the analogous functors in $\CAlg(\Pr^{\mathrm{L},\otimes}_{\mathrm{st}})$ discussed earlier, we know $p^{\mathcal{D}}_{!}$ (which is equivalent to $p_{!}\otimes\mathcal{D}$) admits a symmetric monoidal structure. 
\end{proof}
\end{lemma}

\begin{proof}[Proof of Proposition \ref{prop:postnikovmult}]
By Lemma \ref{lem:postnikovmult1} and Lemma \ref{lem:postnikovmult2}, we know the functor $p_{!}|_{\mathcal{C}}:\mathcal{C}\hookrightarrow\mathcal{D}^{\mathrm{fil}}\to\mathcal{D}$ is symmetric monoidal. Moreover, it is a left adjoint functor admitting $\tau_{\geq k\sbullet}:\mathcal{D}\to\mathcal{C}$ as its right adjoint. In fact, for each $X^{\geq\sbullet}\in\mathcal{C}$ and $y\in\mathcal{D}$, we have natural equivalences 
\begin{align*}
\Map_{\mathcal{D}}(p_{!}X^{\geq\sbullet}, y) & \simeq \Map_{\mathcal{D}^{\mathrm{fil}}}(X^{\geq\sbullet}, p^{\ast}(y))\simeq \int_{\bbZ^{\op}}\Map_{\mathcal{D}}(X^{\geq\sbullet},p^{\ast}(y))\\
 & \simeq \lim_{\{(i,j)\in\bbZ\times\bbZ^{\op}~|~i\geq j\}}\Map_{\mathcal{D}}(X^{\geq i}, p^{\ast}(y)^{\geq j}) = \lim_{\{(i,j)\in\bbZ\times\bbZ^{\op}~|~i\geq j\}}\Map_{\mathcal{D}}(X^{\geq i}, y) \\
 & \simeq \lim_{\{(i,j)\in\bbZ\times\bbZ^{\op}~|~i\geq j\}}\Map_{\mathcal{D}}(X^{\geq i}, \tau_{\geq kj}y) \simeq \int_{\bbZ^{\op}}\Map_{\mathcal{D}}(X^{\geq\sbullet},\tau_{\geq k\sbullet}y)\\
 & \simeq \Map_{\mathcal{D}^{\mathrm{fil}}}(X^{\geq\sbullet}, \tau_{\geq k\sbullet}y).
\end{align*} 
Here, for the second, the third, and the last two equivalences, we used the description of the mapping space of $\mathcal{D}^{\mathrm{fil}}$ via end construction \cite[Prop. 2.3]{gla} and the identification of $\mathrm{Tw}(\bbZ^{\op})$ with $\{(i,j)\in\bbZ\times\bbZ^{\op}~|~i\geq j\}$. For the fifth equivalence, we used that $X^{\geq i}\in \mathcal{D}_{\geq kj}$ for $i\geq j$. \\
\indent Being a right adjoint of a symmetric monoidal functor, we know $\tau_{\geq k\sbullet}:\mathcal{D}\to\mathcal{C}$ is lax symmetric monoidal. By composing with the symmetric monoidal embedding $\mathcal{C}\hookrightarrow\mathcal{D}^{\mathrm{fil}}$ of Lemma \ref{lem:postnikovmult1}, we know the composite $\tau_{\geq k\sbullet}:\mathcal{D}\to\mathcal{D}^{\mathrm{fil}}$ is lax symmetric monoidal. 
\end{proof}

\begin{remarkn} \label{rem:postnikovmult}
Let $\mathcal{D}$ be as in Proposition \ref{prop:postnikovmult}. For $x\in\CAlg(\mathcal{D}^{\otimes})$, the natural map $\tau_{\geq k\sbullet}x\to p^{\ast}(x)$ from the Postnikov filtration to the constant filtration is a map in $\CAlg(\mathcal{D}^{\mathrm{fil,\otimes}})$. In fact, the map is equivalent to a composition of the unit map $\tau_{\geq k\sbullet}x\to p^{\ast}p_{!}(\tau_{\geq k\sbullet}x) = p^{\ast}(p_{!}|_{\mathcal{C}})(\tau_{\geq k\sbullet}x)$ for the adjunction $p_{!}\dashv p^{\ast}$ and the image by $p^{\ast}$ of the counit map $(p_{!}|_{\mathcal{C}})(\tau_{\geq k\sbullet}x)\to x$ of the adjunction $p_{!}|_{\mathcal{C}}\dashv \tau_{\geq k\sbullet}:\mathcal{D}\to\mathcal{C}$. By Lemma \ref{lem:postnikovmult2} and the Proof of Proposition \ref{prop:postnikovmult} above, both of the maps in $\mathcal{D}^{\mathrm{fil}}$ are in fact maps in $\CAlg(\mathcal{D}^{\mathrm{fil},\otimes})$.
\end{remarkn}
 
\printbibliography[heading=bibintoc]
\noindent \textit{\small E-mail address}: \texttt{\small hyungseop.kim@universite-paris-saclay.fr}\\
\noindent \textsc{\small CNRS, Laboratoire de Math\'ematiques d'Orsay, Universit\'e Paris-Saclay, Bâtiment 307, rue Michel Magat, F-91405 Orsay Cedex, France}
\end{document}